\documentclass[10pt,oneside]{amsart}
\usepackage{latexsym,amsfonts,amssymb,amsmath}
\usepackage{amsthm}

\usepackage[a4paper,hmargin=1in,vmargin={1.65in,0.7in}]{geometry}
\usepackage{diagrams}

\DeclareMathOperator{\hth}{ht}
\DeclareMathOperator{\depth}{depth}

\DeclareMathOperator{\codim}{codim}
\DeclareMathOperator{\Supp}{Supp}
\DeclareMathOperator{\Ann}{Ann}
\DeclareMathOperator{\Ind}{Ind}
\DeclareMathOperator{\tot}{tot}

\DeclareMathOperator{\Spec}{Spec}
\DeclareMathOperator{\sgn}{sgn}

\DeclareMathAlphabet{\mathpzc}{OT1}{pzc}{m}{it}

\newcommand{\mbb}{\mathbb}

\newcommand{\mc}{\mathcal}

\newcommand{\pim}{\mathrm{Im}}

\newcommand{\del}{\partial}

\newcommand{\FS}{\mathcal{O}}
\newcommand{\comp}[1]{{#1}^{\bullet}}

\newcommand{\Coh}{\mathrm{Coh}}
\newcommand{\id}{{\rm id}}
\newcommand{\Hom}{{\rm Hom}}

\newcommand{\G}{\mathfrak}

\newcommand{\B}{\mathbf}

\newcommand{\Ext}{{\rm Ext}}
\newcommand{\Tor}{{\rm Tor}}
\newcommand{\Hilb}{{\rm Hilb}}

\newcommand{\GHilb}{{\rm GHilb}}

\newcommand{\per}{\times}
\newcommand{\ra}{\rightarrow}
\newcommand{\perm}{\mathfrak{S}}
\newcommand{\an}[1]{{#1}_{\rm{an}}}
\newcommand{\rest}{ {\bigg |} }
\newcommand{\trest}{{\big |} }

\newcommand{\bkrh}{\mathbf{\Phi}}

\newcommand{\mf}{\mathfrak}

\newcommand{\Aut}{\mathrm{Aut}}

\newcommand{\tens}{\otimes}
\newcommand{\Tens}{\bigotimes}
\newcommand{\w}{\wedge}

\newcommand{\Stab}{\mathrm{Stab}}

\newcommand{\bicomp}[1]{{#1}^{\bullet,\bullet}}
\newcommand{\mcomp}[1]{{#1}^{\bullet, \dots,\bullet}}
\newcommand{\second}{\prime \prime}

\newtheorem{theorem}{Theorem}[subsection]
\newtheorem{lemma}[theorem]{Lemma}
\newtheorem{pps}[theorem]{Proposition}
\newtheorem{crl}[theorem]{Corollary}

\newtheorem{thma}{Theorem}[section]
\newtheorem{lemmaa}[thma]{Lemma}
\newtheorem{ppsa}[thma]{Proposition}
\newtheorem{crla}[thma]{Corollary}
\theoremstyle{definition}
\newtheorem{definition}[theorem]{Definition}

\newtheorem{remark}[theorem]{Remark}
\newtheorem{remarka}[thma]{Remark}

\theoremstyle{remark}
\newtheorem{notata}[thma]{Notation}
\newtheorem{notat}[theorem]{Notation}

\numberwithin{equation}{section}

\mathchardef\phi="0127
\mathchardef\varphi="011E

\mathchardef\alpha="710B
\mathchardef\beta="710C
\mathchardef\gamma="710D
\mathchardef\delta="710E
\mathchardef\epsilon="7122
\mathchardef\zeta="7110
\mathchardef\eta="7111
\mathchardef\theta="7112
\mathchardef\iota="7113
\mathchardef\kappa="7114
\mathchardef\lambda="7115
\mathchardef\mu="7116
\mathchardef\nu="7117
\mathchardef\xi="7118
\mathchardef\pi="7119
\mathchardef\rho="711A
\mathchardef\sigma="711B
\mathchardef\tau="711C
\mathchardef\upsilon="711D
\mathchardef\chi="711F
\mathchardef\psi="7120
\mathchardef\omega="7121
\mathchardef\varepsilon="710F
\mathchardef\vartheta="7123
\mathchardef\varpi="7124
\mathchardef\varrho="7125
\mathchardef\varsigma="7126
\begin{document}
\title[Tautological bundles on Hilbert schemes of points]{Cohomology of the Hilbert scheme of points on a surface with values in 
representations of tautological bundles}
\author{Luca Scala} 
\date{}
\hyphenation{rep-re-sen-ta-tion}

\maketitle

\begin{abstract}Let $X$ a smooth quasi-projective algebraic surface,
$L$ a line bundle on $X$. Let $X^{[n]}$ the Hilbert scheme of $n$ 
points on $X$ and $L^{[n]}$ the tautological bundle on $X^{[n]}$
naturally associated to the line bundle $L$ on $X$. We explicitely 
compute the image $\bkrh(L^{[n]})$ of the tautological bundle $L^{[n]}$ for the
Bridgeland-King-Reid equivalence $\bkrh : \B{D}^b(X^{[n]}) \ra \B{D}^b_{\perm_n}(X^n)$ in terms of a
complex $\comp{\mc{C}}_L$ of $\perm_n$-equivariant sheaves in
$\B{D}^b_{\perm_n}(X^n)$. We give, moreover, a characterization of the image $\bkrh(L^{[n]} \tens \cdots \tens L^{[n]})$ 
in terms of 
of the hyperderived spectral sequence $E^{p,q}_1$ associated to the 
derived $k$-fold
tensor power of the complex
$\comp{\mc{C}}_L$. The study of
the $\perm_n$-invariants of this spectral sequence allows 
to get the derived direct images 
of the double 
tensor power 
 and of the general $k$-fold exterior power 
of the tautological bundle  for the 
Hilbert-Chow morphism,
providing Danila-Brion-type formulas in these two cases.
This yields easily the computation of the cohomology of 
$X^{[n]}$ with values in $L^{[n]} \tens L^{[n]}$ and $\Lambda^k L^{[n]}$. 
\end{abstract}
\section*{Introduction}
The aim of this article is the study of tautological bundles and their tensor powers on Hilbert schemes $X^{[n]}$  of $n$ points  over a smooth quasi-projective algebraic surface $X$, with a particular interest in their cohomology.
A tautological bundle $L^{[n]}$ is a rank $n$ vector bundle on the Hilbert scheme $X^{[n]}$ built in a natural way starting from a line bundle $L$ on the surface $X$; 
it arises as the image of $L$ for the Fourier-Mukai functor 
$\Phi^{\Xi}_{X \ra X^{[n]}}:\B{D}^b(X) \rTo \B{D}^b(X^{[n]})$ with kernel the universal family $\Xi \subseteq X^{[n]} \times X$ of the Hilbert scheme. 

Tautological bundles have turned out being of fundamental importance for the investigation of the topology of the Hilbert scheme $X^{[n]}$, in particular for 
the study of its Chern classes and its cohomology ring, as shown
in \cite{Lehn1999}, \cite{EllingsrudGoettscheLehn2001}. Moreover the 
cohomology
of the Hilbert scheme $\mbb{P}_2^{[n]}$ with values in symmetric powers of tautological bundles is closely related to 
Le Potier's Strange Duality conjecture on the projective plane, as shown in~
\cite{Danilathese}, \cite{Danila2000}, \cite{Danila2002}, \cite{Scala2007}.

Here we present a point of view inspired to the derived McKay 
correspondence and strongly influenced by the fundamental works of Bridgeland-King-Reid \cite{BridgelandKingReid2001} and Haiman \cite{Haiman2001}, \cite{Haiman2002}. By Haiman's results \cite{Haiman2001}, the Hilbert scheme $X^{[n]}$
 of $n$ points on a surface can be identified with the Hilbert scheme 
$\Hilb^{\perm_n}(X^n)$ of $\perm_n$-orbits on $X^n$, in the sense of Nakamura \cite{ItoNakamura1996}, 
\cite{ItoNakamura1999}, \cite{Nakamura2001}. For the action of a finite group 
$G$ on a smooth quasi-projective algebraic variety $M$, the  deep work by Bridgeland, King and Reid
shows that the Nakamura $G$-Hilbert scheme $\Hilb^G(M)$ is the right scheme to consider in order to formulate and prove the derived McKay correspondence: under some hypothesis, the scheme $\Hilb^G(M)$ turns out to be smooth and to provide
a crepant resolution of the quotient $M/G$; furthermore, the  
Fourier-Mukai functor $$\Phi^{\mc{Z}}_{\Hilb^G(M) \ra M} : \B{D}^b(\Hilb^G(M)) \rTo 
\B{D}^b_G(M) $$between the bounded derived category of coherent sheaves on $\Hilb^G(M)$ and the bounded derived category of $G$-equivariant coherent sheaves on $M$, with kernel the universal family $\mc{Z} \subseteq \Hilb^G(M) \times M$, is an equivalence. Haiman proved that the 
Bridgeland-King-Reid theorem can be applied
in the case of the Hilbert scheme of points $X^{[n]}$; consequently we get a Fourier-Mukai equivalence:
\begin{equation*} \bkrh := \Phi^{B^n}_{X^{[n]} \ra X^n} = \B{R}p_*\circ q^*:   \B{D}^b(X^{[n]}) 
\rTo \B{D}^b_{\perm_n}(X^n) \; ,\end{equation*}where the universal family $\mc{Z}$ above is identified here with Haiman's isospectral  Hilbert scheme~$B^n$ and where 
$p$ and $q$ denote the projections from $B^n$ to $X^n$ and $X^{[n]}$, respectively. 

It is natural to expect that the Fourier-Mukai equivalence $\bkrh$ is bound to give new insights on the Hilbert scheme $X^{[n]}$, by translating any object in $\B{D}^b(X^{[n]})$ --~and therefore any coherent sheaf on $X^{[n]}$~-- in an 
object of $\B{D}^b_{\perm_n}(X^n)$, that is, a hopefully simpler object in $\B{D}^b(X^n)$ together with an additional \emph{combinatorial} structure, given by the action of the symmetric group $\perm_n$. After the results by Bridgeland, 
King, Reid and Haiman, it is then natural, in order to understand objects on the Hilbert scheme $X^{[n]}$, to compute their
image by the Fourier-Mukai transform~$\bkrh$. 

This is what we planned to perform for a tautological bundle $L^{[n]}$ 
and its tensor powers $L^{[n]}~\tens~\cdots~\tens~L^{[n]}$. The image 
$\bkrh(L^{[n]} \tens \cdots \tens L^{[n]})$
of the general $k$-fold tensor power of a tautological bundle for the Bridgeland-King-Reid 
transform can be expressed, making fundamental use of 
Haiman's difficult vanishing theorem~\cite{Haiman2002}, in terms of a Fourier-Mukai functor $\B{D}^b(X^k) \rTo \B{D}^b(X^n)$ whose kernel is the structural sheaf of the polygraph $D(n,k) \subseteq X^n \times X^k$.
For $k=1$ the existence of a nice $\perm_n$-equivariant resolution of the
structural sheaf of $D(n,1)$ allows us to
 obtain an explicit expression for the image $\bkrh(L^{[n]})$ of $L^{[n]}$ for the Bridgeland-King-Reid transform; more precisely we defined an explicit 
{\v C}ech-type $\perm_n$-equivariant complex $\mc{C}^{\bullet}_L \in \B{D}^b_{\perm_n}(X^n)$ in terms of the partial diagonals $\Delta_I$ of $X^n$: 
$$ \mc{C}^p_L := \oplus_{|I|=p+1} L_I \quad \; \quad (\partial^p_{L} x)_J = \sum_{i \in J} \epsilon_{i,J} x_{J \setminus \{ i \} } \trest_{\Delta_{J}} $$where 
$\emptyset \neq I,J \subseteq \{ 1, \dots, n \}$ are multi-indexes, $L_I$ denotes the pull-back of $L$ for the projection $\Delta_I \simeq \Delta \times X^{\bar{I}} \rTo \Delta \simeq X$, $x$ denotes a local section of $\mc{C}^p_L$ and $\epsilon_{i,J}$ is an adequate sign. One of the main result of this article is the 
proof that  the image $\bkrh(L^{[n]})$ is quasi-isomorphic to the complex $\comp{\mc{C}}_L$: 
\begin{equation*} \bkrh(L^{[n]}) \simeq \comp{\mc{C}}_L \;. \tag{A}
\end{equation*}
 By taking $\perm_n$-invariants in (A), there follows immediately the Brion-Danila formula for the derived direct images of the tautological bundle for the Hilbert-Chow morphism (see \cite[proposition 6.1]{Danila2001}).
 
For the $k$-fold 
tensor power of $L^{[n]}$ 
we succeed as well in characterizing the image $\bkrh(L^{[n]} \tens \cdots \tens L^{[n]})$ of the Bridgeland-King-Reid transform, but
the information we get becomes less clear and combinatorially more and more 
involved as $k$ grows up, reflecting the more and more complicated combinatorial structure of the polygraph $D(n,k)$ and of its structural sheaf for high $k$. We start from the comparison of the image $\bkrh(L^{[n]} \tens \cdots \tens  L^{[n]})$ with the derived tensor product $\comp{\mc{C}}_L \tens^L \cdots \tens^L \comp{\mc{C}}_L$ 
via a natural morphism $\alpha$ in~$\B{D}^b_{\perm_n}(X^n)$: 
\begin{equation*}\alpha : \comp{\mc{C}}_L \tens^L \cdots \tens^L \comp{\mc{C}}_L \rTo  \bkrh(L^{[n]} \tens \cdots \tens L^{[n]})\;.
\tag{B}\end{equation*}A direct consequence of Haiman's vanishing theorem is 
that the mapping cone of the morphism $\alpha$ is acyclic in degree $>0$, which means that $\bkrh(L^{[n]} \tens \cdots \tens  L^{[n]})$ is  concentrated in degree $0$, or, equivalently, that $ H^i(\bkrh(L^{[n]} \tens \cdots \tens  L^{[n]}))= 
R^i p_* q^* (L^{[n]} \tens \cdots \tens  L^{[n]})=~0$ for all $i>0$. In degree $0$ we get a surjection: 
$$ p_*q^*(L^{[n]}) \tens \cdots \tens p_*q^*(L^{[n]}) \rOnto p_* q^* (L^{[n]} \tens \dots \tens L^{[n]}) \;;$$since the sheaf $p_* q^* (L^{[n]} \tens \dots \tens L^{[n]})$ is  torsion free,  the kernel of the previous epimorphism turns out 
to be the torsion subsheaf. Consequently we obtain that the sheaf $p_* q^* (L^{[n]} \tens \dots \tens L^{[n]}) $ is isomorphic to the term $E^{0,0}_{\infty}$ of the 
$\perm_n$-equivariant hyperderived spectral sequence 
$$E^{p,q}_1 = \bigoplus_{i_1 + \dots + i_k=p} \Tor_{-q}(\mc{C}^{i_1}_L , \dots, \mc{C}^{i_k}_L)$$ associated to the derived tensor product $
\comp{\mc{C}}_L \tens^L \cdots \tens^L \comp{\mc{C}}_L$ and abutting to $\Tor_{-p-q}(\comp{\mc{C}}_L, \dots, \comp{\mc{C}}_L)$: this is the second main result of this article. Even though this characterization
yields a theoretical answer to the problem, we face remarkable combinatorial difficulties when working out explicitely the term 
$E^{0,0}_{\infty}$.

The first idea in order to extract effective information from the $\perm_n$-equivariant 
spectral sequence $E^{p,q}_1$ is to take invariants for the $\perm_n$-action: we obtain in this way a far simpler spectral sequence of invariants 
$\mc{E}^{p,q}_1$ on $S^nX$, whose term $\mc{E}^{0,0}_{\infty}$ can be identified with the direct image $\mu_*(L^{[n]} \tens \cdots \tens L^{[n]})$ of the $k$-fold tensor power of tautological bundles for the Hilbert-Chow morphism $\mu: X^{[n]} \rTo S^nX$. This sequence 
 provides  useful information at least for the case $k=2$, since it degenerates at level $\mc{E}_2$, as shown in 
~\cite{ScalaPhD},~\cite{Scala2006}. 

Here we simplify the argument looking at the restriction $j^*\mc{E}^{p,q}_1$ 
of $\mc{E}^{p,q}_1$ to the big open set $S^n_{**}X$, given by $0$-cycles having 
at most a multiple point of multiplicity $2$. For $k=2$ 
the spectral sequence $j^*\mc{E}^{p,q}_1$ becomes drastically simpler and 
degenerates at level $j^*\mc{E}_2$, yielding the formula: $j^* 
\B{R} \mu_*(L^{[n]} \tens L^{[n]}) \simeq 
(j^* \comp{\mc{C}}_L \tens j^* \comp{\mc{C}}_L)^{\perm_n}$ on $S^n_{**}X$. A local cohomology argument for the complementary of $S^n_{**}X$ in $S^nX$ allows us to extend it to all the symmetric variety, thus providing the Brion-Danila-type formula:
\begin{equation*}\B{R} \mu_*(L^{[n]} \tens L^{[n]}) \simeq 0 \rTo (\mc{C}_L^0 \tens \mc{C}^0_L)^{\perm_n} \rOnto^{(\partial^0_L \tens \id)^{\perm_n}} (\mc{C}_L^1 \tens \mc{C}^0_L)^{\perm_n} \rTo 0 \;.
\tag{C}\end{equation*}
For higher $k$ difficulties appear since nor $\mc{E}^{p,q}_1$ nor its restriction $j^*\mc{E}^{p,q}_1$ degenerates at level~$\mc{E}_2$. 

The next step is  to consider the $k$-fold tensor power 
$L^{[n]} \tens \cdots \tens L^{[n]}$ as a $\perm_k$-equivariant sheaf on $X^{[n]}$: 
the group $\perm_k$ here permutes the factors of the tensor product. 
This permutative $\perm_k$-action can be extended to the $k$-fold derived tensor product
$\comp{\mc{C}}_L \tens^L \cdots \tens^L \comp{\mc{C}}_L$ in such a way that the morphism $\alpha$ 
above (B) is $\perm_k$-equivariant: the spectral sequence $E^{p,q}_1$ above inherits the $\perm_k$-action and is therefore $\perm_n \times 
\perm_k$-equivariant.
 If $\lambda$ is a partition of $k$, we can then characterize 
the image $\bkrh(S^{\lambda}L^{[n]})$  of any Schur functor $S^{\lambda}L^{[n]}$ of $L^{[n]}$ 
for the  transform $\bkrh$ as  the $\perm_k$-invariants  
$(E^{0,0}_{\infty} \tens V_{\lambda})^{\perm_k}$, where $V_{\lambda}$ is the irreducible representation of $\perm_k$ indexed by the partition $\lambda$. Taking furthermore invariants by the geometric symmetric group $\perm_n$ yields a characterization of  the direct image $\mu_*(S^{\lambda} L^{[n]})$ of the Schur functor $S^{\lambda}L^{[n]}$ for the Hilbert-Chow morphism as the $\perm_n \times \perm_k$-invariants: 
$$ \mu_*(S^{\lambda}L^{[n]}) \simeq (E^{0,0}_{\infty} \tens V_{\lambda})^{\perm_n \times \perm_k } \simeq (\mc{E}^{0,0}_{\infty} \tens V_{\lambda})^{\perm_k}\;.$$We study in detail the case of the simpler Schur functor, that is, the $k$-fold exterior power $\Lambda^k L^{[n]}$; 
in this case we take $\lambda=\lambda_{\epsilon}:=(1, \dots,1 )$, so that $V_{\lambda}= \epsilon_k$ is the alternant representation of the symmetric group $\perm_k$.
Restricting everything to the open set $S^n_{**}X$ the spectral sequence $(j^* \mc{E}^{p,q}_1 \tens \epsilon_k)^{\perm_k}$ becomes treatable and it is easy to show that it degenerates at level $\mc{E}_2$. The degeneration yields on $S^n_{**}X$ the Brion-Danila type formula: $j^* \B{R} \mu_*(\Lambda^k L^{[n]}) \simeq j^* (\Lambda^k \comp{\mc{C}}_L)^{\perm_n}$. Since the complex $j^*(\Lambda^k \comp{\mc{C}}_L)^{\perm_n}$ is quasi-isomorphic to the  sheaf $j^*(\Lambda^k \mc{C}^0_L)^{\perm_n}$ and since $\Lambda^kL^{[n]}$ and $\Lambda^k \mc{C}^0_L$ are locally free, a local cohomology argument allows to get an answer on all the symmetric variety $S^nX$: we prove that 
\begin{equation*}  \B{R} \mu_*(\Lambda^k L^{[n]}) \simeq  (\Lambda^k \mc{C}^0_L)^{\perm_n} \;.\tag{D} \end{equation*}

We finally give an immediate application of the results obtained so far to the computation of the cohomology of the Hilbert scheme with values in the double tensor power
and the general $k$-fold exterior power of tautological bundles. 
Making use of (C), we get the following explicit formula for 
$H^*(X^{[n]}, L^{[n]}~\tens~L^{[n]})$: 
\begin{equation*}
 H^*(X^{[n]}, L^{[n]} \tens L^{[n]} ) \simeq H^*(L)^{\tens^ 2} \tens S^{n-2}H^*(\FS_X) \bigoplus H^*(L^{\tens ^2}) \tens \mc{J}
\end{equation*}where $\mc{J}$ is the ideal of classes in $S^{n-1}H^*(\FS_X)$ vanishing on the scheme $y + S^{n-2}X$, for a fixed point $y \in X$. The formula gives an isomorphism of $\mbb{Z}$-graded modules and of $\perm_2$-representations. Taking invariants and anti-invariants for the $\perm_2$-action, we get analogous formulas for the double 
symmetric and exterior power, generalizing known formulas \cite[theorems 1.1--1.4]{Danila2004} 
for all $n$. More generally, let $A$ be a second line bundle on $X$ and consider  the determinant line bundle $\mc{D}_A := \mu^*(A^{\boxtimes ^n}/\perm_n)$ on the Hilbert scheme $X^{[n]}$. As a consequence of (C), we  establish a long exact cohomology sequence for the cohomology of $L^{[n]} \tens L^{[n]}$ twisted by the determinant $\mc{D}_A$, generalizing \cite[theorems 1.6 -- 1.8]{Danila2004} for all~$n$:  
\begin{multline*}
\cdots \rTo H^*(X^{[n]}, L^{[n]} \tens L^{[n]} \tens \mc{D}_A) \rTo \begin{array}
{c}H^*(L^{\tens ^2} \tens A) \tens S^{n-1} H^*(A) \\ \bigoplus \\  
H^*(L \tens A)^{\tens^2} \tens S^{n-2}H^*(A)  \end{array}\rTo \\ 
\rTo H^*(L^{\tens ^2} \tens A^{\tens ^2} ) \tens S^{n-2}H^*(A) \rTo H^{*+1} 
(X^{[n]}, L^{[n]} \tens L^{[n]} \tens \mc{D}_A) \rTo \cdots  \;.
\end{multline*}
Finally, formula (D) implies immediately the isomorphism of 
$\mbb{Z}$-graded modules, for all $0 \leq k \leq n$: 
$$ H^*(X^{[n]}, \Lambda^k L^{[n]} \tens \mc{D}_A) \simeq 
\Lambda^kH^*(L \tens A) \tens S^{n-k}H^*(A) \;.$$It can be thought as a generalization of results 
\cite{Danila2001} for all $k$.

The article is organized as follows. After recalling some preliminary results in section \ref{sect:prel}, in section \ref{section3} we prove the quasi-isomorphism (A) : 
$\bkrh(L^{[n]}) \simeq \comp{\mc{C}}_L$ and we characterize the image $\bkrh(L^{[n]} \tens \cdots \tens L^{[n]})$ in terms  of the hyperderived spectral sequence associated to the derived tensor product 
$\comp{\mc{C}}_L \tens ^L \tens \cdots \tens^L \comp{\mc{C}}_L$. Taking invariants, in sections \ref{sect:inv} and \ref{sect:ext} we get Brion-Danila's formulas (C) and (D) for 
$\B{R} \mu_*(L^{[n]} \tens L^{[n]})$ and $\B{R} \mu_* (\Lambda^k L^{[n]})$, respectively. In section \ref{sect:cohom} we apply these results to the explicit computation 
of the cohomology of $X^{[n]}$ with values in $L^{[n]} \tens L^{[n]}$ and $\Lambda^k L^{[n]}$. In order not to break the main line of the argument, we isolated in the appendixes some independent results and some technical propositions which are used throughout the main text.

{\bf Acknowledgements.} Most of the results presented here were obtained  
during my Ph.D. thesis \cite{ScalaPhD} at University Paris 7, under the 
direction of Prof. Joseph Le Potier, to whom I am deeply indebted: I will 
never forget his unvaluable help, his force, encouragement and generosity, 
his beautiful ideas, his way of doing Mathematics. 

This work was completed during my stay at 
Max-Planck Institut f\"ur Mathematik in Bonn; I thank all the staff for the 
excellent working conditions. 

\section{Preliminary material}\label{sect:prel}
We will always work with varieties and schemes over $\mbb{C}$. 
Unless otherwise explicitely indicated, by point we will always mean a closed point. 
\subsection{Hilbert schemes of points on a surface}
Let $X$ be a smooth quasi-projective complex algebraic surface. Let $n \in \mbb{N}^*$. We will denote with $S^nX$ the $n$-fold symmetric product of $X$, that is, the quotient $X^n/\perm_n$ of $X^n$ by the symmetric group $\perm_n$. 
The symmetric product $S^nX$ is normal and Gorenstein, and, since it is a quotient of a smooth variety by a finite group, it has only rational singularities (see~\cite{Burns1974},~\cite{Boutot1987}).

We will indicate with $X^{[n]}$ the Hilbert scheme of zero dimensional 
closed subschemes of $X$ of length $n$, or, more briefly, the Hilbert scheme of 
$n$ points on $X$. It is well known that $X^{[n]}$ is smooth and that the Hilbert-Chow morphism 
\begin{diagram}[height=0.5cm]
\mu: X^{[n]} & \rTo  & S^nX \\
\quad \xi & \rMapsto & \sum_{x \in X} ({\rm lg}_x \xi) x
\end{diagram}
provides a crepant, semismall resolution of singularities for $S^nX$.
We indicate with $\Xi \subset X^{[n]} \times X$ the universal family and with 
$p_{X^{[n]}}$ and $p_{X}$ the projections from $X^{[n]} \times X$ onto $X^{[n]}$ and $X$, respectively.
\subsection{The $G$-Hilbert scheme}
We will briefly describe the $G$-Hilbert scheme as explained in Reid \cite{Reid2002} and
Nakamura \cite{Nakamura2001}. See also \cite{ItoNakamura1996},
\cite{ItoNakamura1999} and \cite{CrawReid2002}.
Let $G$ be a finite group and $M$ a smooth quasi-projective variety on
which the group $G$ acts. 
\begin{definition} The $G$-Hilbert scheme $\GHilb(M)$ of $G$-clusters on $M$ is the
scheme representing the functor: 
$$ \underline{\GHilb(M)}: \textrm{Sch}/\mbb{C} \rTo \textrm{Sets}
$$associating to a scheme $S$ the set 
\begin{multline*} \underline{\GHilb(M)}(S):= \{ Z \subset S \times M,
  Z \textrm{
closed
$G$-invariant subscheme}, \\ \textrm{flat and finite over 
 $S$ such that  $H^0(\FS_{Z_{s}}) \simeq
\mbb{C} [G] $ for all $s \in S$}  \} \;. 
\end{multline*}
\end{definition}
The closed points of $\GHilb(M)$ parametrize 
$G$-clusters, or, in other words, 
$G$-invariant subschemes of $M$ of length $|G|$ such that 
$H^0(\FS_Z) \simeq \mbb{C}[G]$. We have a $G$-Hilbert-Chow morphism: 
$\GHilb(M) \rTo ^{\tau} M/G $; $\tau$ is projective and surjective, 
birational on the irreducible component containing the free orbits. 
The $G$-Hilbert scheme $\GHilb(M)$ just defined is generally
very bad: it is not irreducible, not equidimensional or even connected.
In order to avoid such problems and to have a good candidate for a crepant resolution of $M/G$, define the \emph{Nakamura $G$-Hilbert scheme} 
$\Hilb^G(M)$ as the irreducible component of $\GHilb(M)$ containing free orbits.
The Nakamura $G$-Hilbert scheme is a fine moduli space for $G$-clusters, with universal family $\mc{Z}$ the restriction of the universal family of $\GHilb(M)$ to $\Hilb^G(M)$. 
The $G$-Hilbert-Chow morphism 
$ \tau : \Hilb^G(M) \rTo M/G$ is now birational and surjective.
The Nakamura $G$-Hilbert scheme $\Hilb^G(M)$ was built by Ito and Nakamura in 
\cite{ItoNakamura1996} and~\cite{ItoNakamura1999}.
\subsection{The BKR theorem}
Let $M$ be a smooth quasi-projective variety and $G$ a finite group acting
on $M$ with the property that the canonical line bundle $\omega_M$ is
locally trivial as a $G$-sheaf. This is equivalent to the fact that 
the stabilizer $G_x$ of a point $x \in M$ acts on $T_xM$ as a subgroup of
$SL(T_xM)$, or to the fact that the quotient $M/G$ is Gorenstein.
Denote with $Y$ the Nakamura $G$-Hilbert scheme $Y=\Hilb^G(M)$.
Consider
the universal family $\mc{Z} \subseteq 
Y \per M$. In the diagram:
\begin{diagram}
\mc{Z} & \rTo^p & M \\
\dTo^q &        & \dTo^{\pi} \\
Y & \rTo ^{\mu} & M/G 
\end{diagram}$\pi$ and $q$ are 
finite of degree $|G|$, $q$ is flat, and $p$ and 
$\mu$ are birational. The structural sheaf $\FS_{\mc{Z}}$ can be seen as 
as a $\{ 1 \} \per G$-equivariant sheaf on $Y \per M$. Therefore we can define
an equivariant Fourier-Mukai functor:
$$ \Phi^{\FS_{\mc{Z}}}_{Y \ra M} : \B{R}p_* \circ q^*: \B{D}^b(Y) \rTo \B{D}^b_G(M) \; .$$
The main result proved by Bridgeland, King and Reid in 
\cite{BridgelandKingReid2001} is the following theorem.
\begin{theorem}\label{BKR}Let $M$ be a smooth quasi-projective variety of dimension $n$, $G$ a 
finite subgroup of ${\rm Aut}(M)$ such that the canonical line bundle 
$\omega_M$ is locally 
trivial as a $G$-sheaf. Let $Y=\Hilb^G(M)$ and $\mc{Z} \subseteq
Y \per M$ the universal closed subscheme. Suppose that 
$ \dim Y \per_{M/G} Y \leq n+1 \; .$ Then $Y$ is a crepant 
resolution of $M/G$ and the Fourier-Mukai functor
$$ \Phi^{\FS_{\mc{Z}}}_{Y \ra M} : \B{D}^b(Y) \rTo \B{D}^b_G(M) $$is
an equivalence
of categories. 
\end{theorem}
\subsection{Isospectral Hilbert schemes and polygraphs}
We will now present the construction of the isospectral Hilbert 
scheme, as defined by Haiman \cite{Haiman2001}, and a brief description 
of its properties. Haiman proves everything for the affine plane $\mbb{A}^2_{\mbb{C}}$, but all 
works for a general smooth quasi-projective surface; we sketch here how to
extend 
some of his 
proofs.
\begin{definition}
The isospectral Hilbert scheme $B^n$ is the reduced fibered product:
\begin{diagram}[LaTeXeqno]\label{dg: isospectral}
B^n & \rTo^{p} &X^n \\
\dTo^{q} & & \dTo ^{\pi} \\
X^{[n]}& \rTo^{\mu}& S^n X  
\end{diagram}that is $B^n : = \big( X^{[n]}\times_{S^nX} X^n \big)_{\textrm{red}}$.
\end{definition}
\begin{remark}
In the definition above it is \emph{necessary} to take the reduced scheme underlying 
the fiber product, since the  fiber product $X^{[n]} \times _{S^n X}X^n $ is \emph{never} reduced, if $n \geq 2$. 
\end{remark}
The following simple lemma allows to extend several of 
Haiman's results to an arbitrary smooth
quasi-projective surface.
\begin{lemma}\label{qproj}
Let $X$ be a quasi-projective variety. Then each point in $X^n$ has an affine 
open neighbourhood of the form $U^n$, where $U$ is an affine open set in $X$.
\end{lemma}\begin{proof}It suffices to prove that given a 
quasi-projective variety $X$ and $n$ points $x_1$, \dots, $x_n$, there 
exists an affine open set $U$ such that all $x_i \in U$. To prove this, embed
$X$ in a projective space $\mbb{P}^N$ and take its projective closure 
$Y = \bar{X}$. Then $Z=Y \setminus X$ is a closed subset of the 
 projective variety $Y$. For large $l$, there exists sections $s_i \in H^0(Y, \mc{I}_Z(l)) \subseteq   H^0(Y, \FS_Y(l))$, vanishing on $Z$ 
but nonzero on $x_i$. If $l$ is large, the subspace $H_i \subseteq 
H^0(Y, \mc{I}_Z(l))$ consisting of sections of $\mc{I}_Z(l)$ vanishing on 
$x_i$ form a hyperplane in $H^0(Y, \mc{I}_Z(l))$ for all $i$. Consider
now a section
$u \in H^0(Y, \mc{I}_Z(l)) \setminus \cup_{i=1}^n H_i$. The affine open set $U$
defined by $u \neq 0$ in $Y$ is contained in $X$ and contains all the points~$x_i$. \end{proof}

The isospectral Hilbert scheme has the following fundamental properties.
\begin{theorem}\label{CM}The isospectral 
Hilbert scheme $B^n$ is irreducible of dimension $2n$, normal, Cohen-Macauley and Gorenstein. Moreover it can be identified with the blow up of $X^n$ along the 
scheme-theoretic union of all its pairwise diagonals. 
\end{theorem}
\begin{proof}The irreducibility follows in all generality, exactly as in 
the case of the affine plane, from the product structure for 
the isospectral Hilbert scheme 
(see \cite[lemma 3.3.1, proposition 3.3.2]{Haiman2001}), which holds true 
for  an arbitrary smooth quasi-projective surface. 
The other properties are proved by Haiman
in the case of the affine plane \cite{Haiman2001}. 
For a smooth affine surface the proof goes exactly as in the case of 
$\mbb{A}^2_{\mbb{C}}$. 
Passing to a smooth quasi-projective surface $X$ 
is now simple, owing to the 
preceding lemma. Consider the projection 
$p: B^n \rTo X^n$, $Q$ a point on $B^n$ and its image $P=p(Q) \in X^n$ 
for the map $p$. 
Let $U^n$ be the affine open set containing $P$, 
found in lemma \ref{qproj}. The isospectral Hilbert scheme $B^n_U$ 
associated to the affine surface $U$ is now the blow-up $B^n_U \rTo U^n$ 
of the scheme-theoretic union of the pairwise diagonals in $U^n$. 
Now the scheme $B^n_U$ has the wanted properties and can be 
identified with the open set $p^{-1}(U^n)$, the inverse image of $U^n$ for the projection $p$. Since the statement is local on $B^n$ and  
on $X^n$, we are done.
\end{proof}
\begin{remark}\label{rem:CMflatmor}
Let $f: X \rTo Y$ be a finite and surjective morphism of schemes.
If $X$ is Cohen-Macauley and $Y$ is
 smooth, then $f$ is flat 
\cite[exercise 18.17]{EisenbudCA}. On the other hand if $f$ is flat and $Y$ is 
Cohen-Macauley
$Y$, then $X$ is Cohen-Macauley. In this case $X$ is Gorenstein if and only if $f$ 
has Gorenstein fibers~
\cite[chapter 10, n.7,~\S 2,~\S 3]{BourbakiAC}. 
\end{remark}
\begin{remark}The previous remark applies  to the
  isospectral Hilbert scheme $B^n$, since by theorem \ref{CM} it 
is Cohen-Macauley and $q: B^n \rTo X^{[n]}$ is finite and surjective with $X^{[n]}$ smooth. Hence the Cohen-Macauley property of 
$B^n$ is equivalent to the flatness of the morphism $q$. Therefore in the diagram (\ref{dg: isospectral})
$q$ is flat of degree $n!$, even if $\pi$ is not. 
\end{remark}
\begin{remark}\label{rem:isounivfam}
The subscheme $Z \subseteq B^n \times X$, defined as the pull-back 
$(q \times \id)^{-1}(\Xi)$, 
where $\Xi$ is the universal family over $X^{[n]}$, is called the 
 \emph{universal subscheme} for the isospectral Hilbert scheme. 
Since $\Xi$ is flat and finite over $X^{[n]}$, 
$Z$ is flat and finite over $B^n$, hence Cohen-Macauley. It is
reduced, because generically reduced. By definition, the scheme $Z$ is isomorphic to the fibered product $Z \simeq B^n \times_{X^{[n]}} \Xi$ of $B^n$ and $\Xi$ over the Hilbert scheme $X^{[n]}$.
While $\Xi$ is irreducible, $Z$ has $n$ irreducible components 
$$Z_i = (p \times \id)^{-1}(D_i) \;,$$ where $D_i$ is the diagonal 
$$D_i= \Delta_{i,n+1} \subseteq X^n \times X \;.$$ 
It is clear that $Z_i \simeq B^n$, hence $Z_i$ are normal, 
Cohen-Macauley and Gorenstein.
\end{remark}
\subsection{The Hilbert scheme $X^{[n]}$ as the $\Hilb^{\perm_n}(X^n)$ scheme.}
We will now explain briefly why the scheme
$\Hilb^{\perm_n}(X^n)$ can be identified with the Hilbert scheme $X^{[n]}$ 
of $n$ points on $X$ and why the action of the symmetric group $\perm_n$  on 
the product $X^n$ of a smooth 
quasi-projective surface $X$
satisfies
the hypothesis of the BKR theorem \ref{BKR}. 
See \cite{Haiman2001} and \cite{Haiman1999}. 

The key point for the comparison between $X^{[n]}$ and $\Hilb^{\perm_n}(X^n)$ is  the Cohen-Macauley property of $B^n$, or, equivalently, 
the flatness of the morphism $q: B^n \rTo X^n$. Since $B^n$ inherits from $X^n$ a $\perm_n$-action\footnote{because $B^n$ is the blow-up of $X^n$ along a 
$\perm_n$-invariant subscheme, or directly from the definition}, the direct image $q_* \FS_{B^n}$ is a rank $n!$ locally free $\perm_n$-equivariant sheaf of algebras over $X^{[n]}$: the $\perm_n$-action induces locally a morphism: $\FS_{X^{[n]}}[\perm_n] \rTo q_* \FS_{B^n}$, which is easily an isomorphism, since it is 
such generically. As a consequence, 
$X^{[n]}$ can be viewed as the quotient of $B^n$ by the $\perm_n$-action. Moreover, for every closed subscheme $\xi \in X^{[n]}$, the fiber $B^n_{\xi}$ is a $\perm_n$-invariant 
subscheme of $X^n$, such that $H^0(\FS_{B^n_{\xi}}) \simeq \mbb{C}[\perm_n]$. Hence, 
$B^n$ 
is a flat family of $\perm_n$-clusters on $X^n$. By the universal property of the $\perm_n$-Hilbert scheme $\perm_n\Hilb(X^n)$, and since $X^{[n]}$ is irreducible, the family $B^n$ gives rise to a morphism: $$ \varphi: X^{[n]} \rTo \Hilb^{\perm_n}(X^n) \;.$$
Haiman proves in \cite{Haiman2001} that $\varphi$ is an isomorphism, 
by explicitly exhibiting an inverse. The proof needs by no means  the fact that $X$ is isomorphic to $\mbb{A}^2_{\mbb{C}}$ or that $X$ is affine, and it can be carried on in exactly the same way for any smooth quasi-projective surface; hence:
\begin{theorem}The Hilbert scheme of $n$ points on a smooth 
quasi-projective surface $X$ is isomorphic, over the symmetric variety $S^n X$, to the scheme
$\Hilb^{\perm_n}(X^n)$. 
\end{theorem}
At this point, to prove that the $\perm_n$-action on $X^n$ satisfies the hypothesis 
of theorem \ref{BKR} we have to prove that $\omega_{X^n}$ is locally trivial as 
$\perm_n$-sheaf --- which is equivalent to the fact that $S^nX$ is Gorenstein
--- and the smallness condition 
$ \dim (X^{[n]} \times_{S^n X} X^{[n]}) \leq 2n +1 \; ;$
but this is a direct consequence of the fact that the Hilbert-Chow morphism:
$ \mu : X^{[n]} \rTo S^n X $ is a semismall resolution  
\cite[proposition 2.11]{DeCataldoMigliorini2004}.
As a consequence, we get the following remarkable particular case of theorem \ref{BKR}:
\begin{theorem}[Haiman] \label{BKRH}
Let $X$ be a smooth quasi-projective surface, $X^{[n]}$ the Hilbert scheme of 
$n$-points on $X$, $B^n$ the isospectral Hilbert scheme. Let $q: B^n \rTo 
X^{[n]}$ and $p: B^n \rTo X^n$ be the projections on the Hilbert scheme and on the product variety, 
respectively.
The Fourier-Mukai functor:
\begin{equation}
\label{BKRHequiv}
 \bkrh = 
 \B{R}p_* \circ q^*:
\B{D}^b(X^{[n]}) \rTo \B{D}^b_{\perm_n}(X^n) \end{equation}
is an equivalence. 
\end{theorem}
\subsection{Polygraphs}
\begin{definition}
Let $X$ be a smooth quasi-projective variety. Let $D \subseteq X^n \times
X$ be the scheme-theoretic union of 
pairwise diagonals $D_i = \cup_{i=1}^n \Delta_{i,n+1}$. The
polygraph
$D(n,k) \subseteq X^n \times X^k$ is the reduced $k$-fold  fiber product 
$$ D(n,k): = \left(
\underset{\textrm{$k$-times}}{\underbrace{ D \times_{X^n} \dots
    \times_{X^n} D }} \right)_{\textrm{red}}  \;.$$ 
\end{definition}The polygraph $D(n,k)$ is finite over $X^n$, but 
\emph{not} flat. Its generic degree over $X^n$ is~$n^k$.
The term polygraph is explained by the following characterization. {\sloppy  Let $f : \{1, \dots, k\} \rTo \{1, \dots,n \}$ be a map.
Let $E_f \subseteq X^n \times X^k$ be the graph of the morphism $X^n \rTo X^k$, sending $(x_1, \dots, x_n) \rMapsto (x_{f(1)}, \dots, x_{f(k)})$. Then the scheme $D(n,k)$ is exactly the scheme-theoretic union $D(n,k) = \cup_f E_f$, where the union runs over all maps $f : \{1, \dots,k \} \rTo \{1, \dots,n \}$. }

We now take the $k$-fold fibered products of $\Xi$ and $Z$ (over $X^{[n]}$ and $B^n$, respectively) and we compare them to the polygraph $D(n,k)$.
Let us denote with 
$\Xi(n,k)$  the $k$-fold fibered product:
$$ \Xi(n,k) := \underset{\textrm{$k$-times}}{\underbrace{
\Xi \per_{X^{[n]}} \dots \per_{X^{[n]}}
  \Xi}} \; .$$The scheme $\Xi(n,k)$ is flat and finite over $X^{[n]}$ of degree $n^k$ and hence Cohen-Macauley. 
 By analogy
with the previous definitions we set:
$$Z(n,k):= \underset{\textrm{$k$-times}}{\underbrace{Z \times_{B^n} \dots \times_{B^n} Z
  }} \;. $$The scheme $Z(n,k)$ is flat and finite over $B^n$ of degree $n^k$: this implies that
it is Cohen-Macauley and hence reduced, since generically
reduced. Since 
$Z$ is isomorphic to the fiber product 
of $\Xi$ and $B^n$ over the Hilbert
  scheme (see remark~\ref{rem:isounivfam}), it is immediate to see that:
\begin{equation}\label{EEE}Z(n,k) = \Xi(n,k) \times_{X^{[n]}} B^n \;,
\end{equation}or, equivalently, that it is 
isomorphic to the pull back:
$Z(n,k)= (q \times \id)^{-1}(\Xi(n,k))$, where $ (q \times \id) :
B^n \times X^k \rTo X^{[n]} \times X^k$. We come now to the fundamental 
Haiman's vanishing theorem \cite[proposition~5.1]{Haiman2002}.
In~\cite{ScalaPhD} we gave an independent proof for the case $k=1$.
\begin{theorem}[Haiman]\label{thm: Haimanvanishing} Consider the map:
$$ f := p \times \id : B^n \times X^k \rTo X^n \times X^k \quad .$$
Then the derived direct image $\B{R}f_* \FS_{Z(n,k)}$ of the
structural sheaf of $Z(n,k)$ is the structural sheaf of the polygraph 
$D(n,k)$:
$$ \B{R}f_* \FS_{Z(n,k)} \simeq \FS_{D(n,k) } \; .$$
\end{theorem}
\begin{proof}
The case of the affine plane $X=\mbb{A}^2_{\mbb{C}}$ has been proved Haiman \cite{Haiman2002}. To
prove it for an arbitrary smooth quasi-projective surface we use the following comparison formula between algebraic and complex-analytic higher direct images (see 
\cite[theorem 2.48]{KollarMori1998}, or \cite[lemma 2.14]{ScalaPhD}, for a proof): if $g : S^{\prime} \rTo S$ is a projective morphism between complex algebraic varieties, and $g_{\rm {an}} : S^{\prime}_{\rm{an}} \rTo S_{\rm{an}}$ is the associated morphism between complex analytic spaces, we have: 
$$ \B{R} {g_{\rm{an}}}_* \FS_{S^{\prime}_{\rm{an}}} \simeq \B{R}{g_*} \FS_{S^{\prime}} \tens_{\FS_{S}} \FS_{S_{\rm{an}}} \;.$$We apply this remark to the morphism
$ f : Z(n,k) \rTo D(n,k)$, keeping into account that, by GAGA principle \cite{SerreGAGA}, 
for all complex algebraic variety $S$, the morphism of ringed spaces 
$ ( \an{S}, \FS_{\an{S}}) \rTo ( S, \FS_{S}) $ is faithfully
flat. This implies that:
$$ \B{R}f_*  \FS_{Z(n,k)} \simeq \FS_{D(n,k) } \iff \B{R}f_*
\FS_{\an{Z(n,k)}} 
\simeq \FS_{\an{D(n,k)} } \; .$$Since the facts are local in nature,
it suffices to prove the statement on a small analytic open subset $V
$ of 
$D(n,k)$. We can always choose it of the form:
$$ V \simeq \prod_{j=1}^s D_{U_j}(n_j,k_j) $$with $U_j$ small analytic
open set of $\mbb{C}^2$, $n_j,k_j$ positive natural numbers such that 
$\sum_j n_j =n$ and $\sum_j k_j = k$
and $D_{U_j}(n_j,k_j)$ the analytic polygraph relative
to $U_j$. Over $V$, the analytic space  $\an{Z(n,k)}$ is now of the form 
$$ \an{Z(n,k)}\rest_V \simeq \prod_{j=1}^s Z_{U_j}(n_j, k_j) $$and the map $f$
is now the product map. Since the $U_j$ are now analytic open sets of
$\mbb{C}^2$, and since the result is true for analytic open sets of
$\mbb{C}^2$, because it is true algebraically for $\mbb{C}^2$, we are
done.\end{proof}
\subsection{Equivariant Fourier-Mukai functors and cohomology}
We  will now
describe a general method to compute cohomology of a quotient via 
equivariant hypercohomology. First, a lemma. 
\begin{lemma}\label{lem:invariants}Let $k$ be a field with ${\rm char}(k)=0$, $R$ a $k$-algebra, $G$ a finite group 
and $M$ an $R[G]$-module. Let $N$ be a 
$R$-module (that is, $G$ acts trivially on $N$).  Then 
$$ ( M\tens^L_R N)^G = M^G \tens^L_R N \; .$$
\end{lemma}
\begin{proof}Resolve $M$ with a projective resolution
$P^{\bullet} \rTo M$. Applying the fixed points functor we get a
resolution:
$ (P^G)^{\small \bullet} \rTo M^G$ of $M^G$. Now the
$(P^i)^G$ are projective elements, because they are direct factors of
the $P^i$, which are themselves projective, and a direct factor of a
projective is projective. Hence $(P^G)^{\bullet} \rTo M^G$ is a
projective resolution of $M^G$. Now the tensor product $- \tens N$
commutes with the fixed points functor, since $G$ does not act on $N$:
$ (-)^G \tens N \simeq (- \tens N)^G$. Deriving, since the fixed
points functor takes projectives to projectives, we are done. 
\end{proof}
\begin{pps}\label{pps:invamor}
Let $G$ be a finite group.
Let $p: X \rTo Y$ be a proper $G$-equivariant morphism of algebraic varieties. 
Let $q :X \rTo X/G$ and 
$\pi: Y \rTo Y/G$ be the quotients by $G$ of $X$ and $Y$, respectively. Let $\mu: X/G \rTo Y/G$ be the map induced by $p$ on the 
quotient level. Let $\Phi$ be the equivariant Fourier-Mukai functor $\Phi = \B{R} p_* \circ \B{L} q^*
: \B{D}^{-}(X/G) \rTo \B{D}^{-}_G(Y)$. Then:
\begin{enumerate}
\item $\B{R} \mu_* \simeq \pi_*^G \circ \Phi $;
\item the hypercohomology of a complex $\mc{F}^{\bullet} \in \B{D}^{-}(X/G)$ is 
the $G$-equivariant hypercohomology of $\Phi(\mc{F}^{\bullet})$, or the hypercohomology of the quotient $Y/G$ with values in the invariants 
of the image $\Phi(\comp{\mc{F}})^G$:
$$ \mbb{H}(X/G, \mc{F}^{\bullet}) \simeq  \mbb{H}_G(Y, \Phi(\mc{F}^{\bullet})) 
\simeq \mbb{H}(Y/G, \Phi(\comp{\mc{F}})^G)\;.$$
\end{enumerate}  
\end{pps}\begin{proof}\begin{eqnarray*}
\pi_*^G \circ \Phi (\mc{F}^{\bullet})& =&  
\pi_*^G \circ \B{R} p_* \circ \B{L} q^* (\mc{F}^{\bullet})
 \simeq  [ \B{R} \mu_* \circ q_*( \B{L} q^* (\mc{F}^{\bullet}) )]^G \\
 & \simeq & \B{R} \mu_* q^G_* ( \B{L}q^* (\mc{F}^{\bullet})) \simeq 
\B{R} \mu_* (
\comp{\mc{F}} \tens^L_{\FS_{X/G}} q_*^G \FS_X) \\
 & \simeq   & \B{R} \mu_* (
\comp{\mc{F}} )
\end{eqnarray*}because $\mu$ is $G$-invariant, because of lemma \ref{lem:invariants}, the projection formula and the
fact that $q_*^G \FS_X \simeq \FS_{X/G}$. 
The second statement is now an easy consequence:
\begin{eqnarray*} \mbb{H}_G(Y, \Phi(\mc{F}^{\bullet})) & = & \B{R} \Gamma_Y^G \circ  \Phi(\mc{F}^{\bullet})  = \B{R} \Gamma_{Y/G} \circ \pi_*^G \circ
\Phi(\mc{F}^{\bullet}) \\
& \simeq & \B{R} \Gamma_{Y/G}  \circ \B{R} \mu_* ( \mc{F}^{\bullet}) 
\simeq  \B{R} \Gamma_{X/G} ( \mc{F}^{\bullet}) \\ & \simeq &\mbb{H}(X/G, \mc{F}^{\bullet})\;.
\end{eqnarray*}
\end{proof}
As an application of the preliminary results explained in this section we prove the following vanishing.
\begin{pps}\label{pps: direct}Consider the morphism $p : B^n \rTo X^n$.
We have $ \B{R} p_* \FS_{B^n} \simeq \FS_{X^n}$.
\end{pps}
\begin{proof}Since $\bkrh: \B{D}^b(X^{[n]}) \rTo \B{D}^b_{\perm_n}(X^n)$ is an equivalence, for all complex $\comp{G}
\in  \B{D}^b(X^{[n]})$ and for all~$i$:
\begin{eqnarray*}
\Hom^i_{\B{D}^b_{\perm_n}(X^n) } ( \bkrh( \FS_{X^{[n]}} ), \bkrh(
\comp{G}))  & \simeq  & \Hom_{\B{D}^b(X^{[n]})}^i (\FS_{X^{[n]}},
\comp{G}) 
  \simeq  \Ext^i_{X^{[n]}}(\FS_{X^{[n]}}, \comp{G}) \\
 & \simeq & \mbb{H}^i( X^{[n]} , \comp{G}) \;.
\end{eqnarray*}
Now, by proposition\footnote{Remark that the equivariant hypercohomology 
$\mbb{H}^i_{\perm_n}( X^n, \bkrh(
\comp{G}))$ coincides with the invariant hypercohomology $\mbb{H}^i ( X^n, \bkrh(
\comp{G}))^{\perm_n}$ because the fixed point functor $[-]^{\perm_n}$ is exact} \ref{pps:invamor} 
$\mbb{H}^i( X^{[n]} , \comp{G}) \simeq \mbb{H}^i ( X^n, \bkrh(
\comp{G}))^{\perm_n}$ and the last term is
$$ \mbb{H}^i ( X^n, \bkrh(
\comp{G}))^{\perm_n} \simeq \perm_n\Ext^i_{X^{n}}( \FS_{X^n}, \bkrh(
\comp{G}))\simeq \Hom^i_{\B{D}^b_{\perm_n}(X^n) } ( \FS_{X^n}, \bkrh(
\comp{G})) \; .$$Therefore for all $\comp{G} \in \B{D}^b(X^{[n]})$:
$$ \Hom_{\B{D}^b_{\perm_n}(X^n) } ( \bkrh( \FS_{X^{[n]}} ), \bkrh(
\comp{G})) \simeq \Hom_{\B{D}^b_{\perm_n}(X^n) } ( \FS_{X^n}, \bkrh(
\comp{G})) \; $$and since every object in $\B{D}^b_{\perm_n}(X^n)$ can
be written as $\bkrh(
\comp{G})$ for some $\comp{G} \in \B{D}^b(X^{[n]})$, because $\bkrh$ 
is an equivalence, 
by Yoneda Lemma we
obtain:
$ \bkrh( \FS_{X^{[n]}}) \simeq \FS_{X^n}$, which is exactly:
$\B{R} p_* \FS_{B^n}
\simeq \FS_{X^n} $. \end{proof}
\section{The image of a tensor product of tautological bundles for the BKR transform} \label{section3}
In this section we are going to compute the image $\bkrh(L^{[n]} \tens \cdots \tens L^{[n]})$ of a $k$-fold 
tensor power of a
tautological bundle for the BKR transform. 
We will give a simple and explicit answer for $k=1$ in terms of a $\perm_n$-equivariant complex $\comp{\mc{C}}_L \in \B{D}^b_{\perm_n}(X^n)$; for higher $k$ 
we will give a characterization of the image in terms of a 
$\perm_n$-equivariant 
spectral sequence of coherent sheaves on $X^n$. 
\subsection{Tautological bundles}
\begin{definition}Let $L$ be a line bundle on the smooth quasi-projective 
algebraic surface $X$. 
The \emph{tautological bundle} $L^{[n]}$ on the Hilbert scheme $X^{[n]}$, associated to the line bundle $L$ on $X$, is the rank $n$ vector bundle defined as the 
image of the line bundle $L$ for Fourier-Mukai functor 
$\Phi^{\FS_{\Xi}}_{X \ra X^{[n]}} : \B{D}^b(X) \rTo \B{D}^b(X^{[n]} )$ :
$$ L^{[n]} := \Phi^{\FS_{\Xi}}_{X \ra X^{[n]}}(L) = (p_{X^{[n]}})_*(\FS_{\Xi} \tens p_X^*L) \;.$$
\end{definition}
\begin{remark}
The tautological bundle $L^{[n]}$ is a rank $n$ vector bundle because the universal family $\Xi$ is flat and finite of degree $n$ over the Hilbert scheme $X^{[n]}$ (see 
\cite[corollary 5.5]{KollarMori1998}).
\end{remark}
We want now to express as well the tensor power
of a tautological bundle on the Hilbert scheme in terms of a 
Fourier-Mukai functor. Consider the following diagram:
\begin{diagram}
\Xi(n,k) & \rInto ^{j} & \Xi ^k &
\rTo ^{p_X^k} & X^k \\
\dTo ^{w} & \Box  & \dTo _{p^k_{X^{[n]}}} & & \\
X^{[n]} & \rInto ^{i} & (X^{[n]})^k & & 
\end{diagram}where the square is cartesian, and where $i$ and $j$ denote the
diagonal immersions. We remark that $p_{X^{[n]}}^k$ and $w$ are flat and finite of
degree $n^k$. 
Consider a line bundle $L$ on $X$ and
its $k$-fold exterior tensor 
product $L \boxtimes \cdots \boxtimes L$ on $X^k$. 
It's clear that 
$$ i^* {p_{X^{[n]}}^k}_* (p_X^k )^* (L \boxtimes \cdots \boxtimes L ) 
= L^{[n]} \tens
\cdots \tens L^{[n]} = w_* j^* (p_X^k)^*(L \boxtimes \cdots \boxtimes L)
$$by base change. Therefore we can express
the tensor product of tautological bundles $L^{[n]} \tens
\cdots \tens L^{[n]}$ in terms of the Fourier-Mukai functor with kernel $\Xi(n,k)$:
\begin{equation}\label{A4}
L^{[n]} \tens
\cdots \tens L^{[n]} = \Phi^{\FS_{\Xi(n,k)}}_{X^k \ra X^{[n]}}(L
\boxtimes \cdots \boxtimes L) \;. 
\end{equation}

We want now to find the image $\bkrh(L^{[n]} \tens \cdots \tens L^{[n]})$ of 
the tensor product of tautological bundles by the BKR equivalence 
(\ref{BKRHequiv}). By 
(\ref{A4}) and by the fact that a composition of Fourier-Mukai functors is a Fourier-Mukai functor \cite[proposition 1.3]{Mukai1981}, we have:
\begin{eqnarray*}
\bkrh (L^{[n]}
\tens \cdots \tens L^{[n]} ) & \simeq & 
 \Phi^{ \FS_{B^n} }_{ X^{[n]} \ra X^n } 
\Phi^{\FS_{\Xi(n,k)}}_{X^k \ra X^{[n]}}(L \boxtimes \cdots \boxtimes
L)  \\ & \simeq & \Phi^{ \mc{K} }_{ X^k \ra X^n}(L 
\boxtimes \cdots \boxtimes
L)
\end{eqnarray*}where the kernel 
$\mc{K}$ of the composition is 
$ \mc{K} \simeq \B{R}f_* ( \FS_{B^n}
\tens^L_{\FS_{X^{[n]}}} \FS_{\Xi(n,k)} )$, where we naturally consider 
$\FS_{B^n}
\tens^L_{\FS_{X^{[n]}}} \FS_{\Xi(n,k)}$ as a complex on $B^n \times X^k$.
Since $\Xi(n,k)$ is flat over
$X^{[n]}$, the latter reduces to:
$$ \mc{K} \simeq \B{R}f_* ( \FS_{B^n}
\tens_{\FS_{X^{[n]}}} \FS_{\Xi(n,k)} ) \simeq \B{R}f_* ( \FS_{B^n \times
  _{X^{[n]}} \Xi(n,k) } ) \;.$$Now, by (\ref{EEE}),
$\FS_{ B^n \times
  _{X^{[n]}} \Xi(n,k) } \simeq \FS_{Z(n,k)}$; Haiman's vanishing theorem \ref{thm: Haimanvanishing} 
then yields
$$ \mc{K} \simeq \B{R} f_* \FS_{Z(n,k)} \simeq
\FS_{D(n,k)} \;.$$Therefore the image of the 
tensor power of a tautological bundle for the BKR transform~$\bkrh$~is:
\begin{equation}\label{C4}
 \bkrh (L^{[n]}
\tens \cdots \tens L^{[n]} ) \simeq \Phi^{\FS_{D(n,k)}}_{X^k \ra X^n}(L \boxtimes \dots \boxtimes
L) \;.
\end{equation}
\subsection{The case $k=1$}
In the case $k=1$, for a line bundle $L$ on $X$ we have: 
\begin{equation} \label{k=1}\bkrh(L^{[n]}) \simeq \Phi^{\FS_D}_{X \ra X^n}(L) \;.
\end{equation}
We are going to express the image $\bkrh(L^{[n]})$ in terms of a bounded complex
$\comp{\mc{C}}_L$ of $\perm_n$-equivariant sheaves on $X^n$. The method we follow consists in finding a $\perm_n$-equivariant resolution of the sheaf $\FS_D$ in terms of structural 
sheaves of its components (the pairwise diagonals $D_i = \Delta_{i,n+1}$) and their intersections. 
For any multi-index $\emptyset \neq 
I \subseteq \{1, \dots, n \}$ denote with $D_I$ the partial diagonal
$D_I := \cap_{i \in I}D_i$. Applying corollary \ref{ccx} we find a 
{\v C}ech-type 
right complex $\comp{\check{\mc{K}}}$ 
defined by: 
$$ {\check{\mc{ K}}^p} := \bigoplus_{|I|=p+1} \FS_{D_I} \qquad \partial^p(f)_J = 
\sum_{i \in J} \epsilon_{i,J} f_{J\setminus \{i\}} \trest_{D_J} \;,$$
where $\epsilon_{i,J}$ is the sign $\epsilon_{i,J}:= (-1)^{\sharp \{h \in J \; | \; 
h <i \}}$;
the complex $\comp{\check{\mc{K}}}$ is $\perm_n$-equivariant and provides a resolution of the sheaf $\FS_D$:
$$ \FS_D \simeq  \comp{\check{\mc{K}}} \;.$$
\emph{The complex $\comp{\mc{C}}_L$.}
We now introduce the 
complex of $\perm_n$-equivariant sheaves $\comp{\mc{C}}_L$ in $\B{D}^b_{\perm_n}(X^n)$.
Let $p_I : X^n \rTo X^I$ be the projection onto 
the factors in $I$, and let $i_I : X \rInto X^I$ be the diagonal immersion of $X$ into $X^I$. If $L$ is a line bundle on $X$, let us denote with $L_I$ the sheaf
$L_I := p_I^* (i_I)_* L$, supported on $\Delta_I$. Let us define now define 
the complex $(\comp{\mc{C}}_L, \partial^{\bullet}_{L})$ in the following way:
$$ \mc{C}^p_L := \bigoplus_{|I|=p+1} L_I \qquad \partial^p_L(x)_J = \sum_{i \in J} \epsilon_{i,J} x_{J\setminus \{i\}} \trest_{\Delta_J} \;.$$
\begin{remark}\label{farfalla}
  Let $\sigma \in \perm_n$; denote with $\sigma_* $
the automorphism $\sigma_* \in \Aut(X^n)$ defined as: $\sigma_*(x_1, \dots, x_n) := (x_{\sigma^{-1}(1)}, \dots, x_{\sigma^{-1}(n)})$. We have $\sigma_*(L_I) = L_{\sigma(I)}$ for all 
$ \emptyset \neq I \subseteq \{1, \dots,n \}$. The $\perm_n$-linearization on $\mc{C}^p_L := 
\oplus_{|I|=p+1} L_I$ can then be defined setting $(\sigma.x)_J := \epsilon_{\sigma, J} \sigma_* x_{\sigma^{-1}(J)}$, where $(x_{I})_I$ is a local section of $\mc{C}^p_L$ and where
$\epsilon_{\sigma,J}$ is the signature of the only permutation 
$\tau$ such that $\sigma^{-1}\tau$ is strictly increasing. The sign $\epsilon_{\sigma,J}$ is necessary in order to balance the signs present in the definition of the differential $\partial^p_L$ and to make it $\perm_n$-equivariant. \end{remark}
\begin{theorem}\label{B3}
Let $X$ be a smooth quasi-projective algebraic surface and $L$ a line bundle
on $X$. Let $L^{[n]}$ be the tautological bundle on the Hilbert
scheme $X^{[n]}$ associated to $L$. Let 
$$ \bkrh=\Phi^{\FS_{B^n}}_{X^{[n]} \ra X^n} : \B{D}^b(X^{[n]}) \rTo
  \B{D}^b_{\perm_n}(X^n) $$be the Bridgeland-King-Reid
  equivalence. Then the image of the tautological bundle 
$L^{[n]}$ for the equivalence $\bkrh$ is
  isomorphic in $\B{D}^b_{\perm_n}(X^n)$ to the complex
  $(\comp{\mc{C}}_L, \comp{\del}_L )$:
$$ \bkrh ( L^{[n]}) \simeq
\comp{\mc{C}}_L \; .$$
\end{theorem}
\begin{proof}By (\ref{k=1}) the image 
 $\bkrh(L^{[n]})$ of the tautological bundle $L^{[n]}$ is given by the image of $L$ for the Fourier-Mukai functor $\Phi^{\FS_D}_{X \ra X^n}$:
$$ \bkrh(L^{[n]}) \simeq \Phi^{\FS_D}_{X \ra X^n}(L) \;.$$The latter reduces simply to:
$ \Phi^{\FS_D}_{X \ra X^n}(L) \simeq (p_{X^n})_*( \FS_D \tens p_X^*L)$. 
Replacing $\FS_D$ with its {\v C}ech resolution $\comp{\check{\mc{K}}}$ we get:
$$ \bkrh(L^{[n]}) \simeq (p_{X^n})_*( \comp{\check{\mc{K}}} \tens p_X^*L) \;.$$
It remains now to identify the complex $(p_{X^n})_*( \comp{\check{\mc{K}}} \tens p_X^*L) $ with 
the complex $\comp{\mc{C}}_L$. The term in degree $0$ is isomorphic to:
$$ (p_{X^n})_*(\check{\mc{K}}^0 \tens p_X^* L ) \simeq \bigoplus_{i=1}^n 
(p_{X^n})_*(\FS_{D_i} \tens p_X^* L ) \;.$$Since $p_{X^n}: D_i \rTo X^n$ 
is an isomorphism and the projection
$p_X :D_i  \rTo   X$ factorizes as the composition  
$$ D_i \rTo^{p_{X^n}} X^n \rTo^{p_i} X \;,$$
by projection formula
$ (p_{X^n})_*(\FS_{D_i} \tens p_X^* L ) \simeq p_i^* L \simeq L_i$.
For the terms in higher degree, it is sufficient to note that 
$p_{X^n}\trest_{D_I} : D_I \rTo \Delta_I $ is an isomorphism, and consequently that
$$  (p_{X^n})_*(\check{\mc{K}}^p \tens p_X^* L ) \simeq \bigoplus_{|I|=p+1}
(p_{X^n})_*(\FS_{D_I} \tens p_X^* L ) \simeq \bigoplus_{|I|=p+1} L_I \;.$$The differentials are 
induced by the differentials of $\comp{\check{\mc{K}}}$ and clearly coincide with $\partial^p_L$.
\end{proof}
\subsection{The general case}For the image $\bkrh (L^{[n]} \tens \cdots \tens L^{[n]} ) $ we have the following characterization:
\begin{theorem}\label{treotto}
There is a natural morphism in $\B{D}^b_{\perm_n}(X^n)$:
\begin{equation}\label{E4}
\alpha: \comp{\mc{C}}_{L} \tens^L \cdots \tens^L \comp{\mc{C}}_{L} \rTo 
\bkrh (L^{[n]} \tens \cdots \tens L^{[n]} ) 
\end{equation}whose mapping cone is acyclic in degree $>0$. This means that
$$H^i(\bkrh (L^{[n]} \tens \cdots \tens L^{[n]} )) \simeq 
R^ip_* q^*(L^{[n]} \tens \cdots \tens L^{[n]} )=0 \quad \forall i>0 $$
and in degree $0$ the morphism:
\begin{equation}\label{eq:cf} p_*q^*(L^{[n]}) \tens \cdots \tens p_*q^*(L^{[n]}) \rTo 
p_* q^*(L^{[n]} \tens \cdots \tens L^{[n]} )\end{equation}is surjective and the 
kernel is the torsion subsheaf.
\end{theorem}
\begin{proof}
To build the morphism $\alpha$ consider the cartesian diagram:
\begin{diagram}[LaTeXeqno] \label{numero}
\widetilde{D(n,k)} & \rInto ^j & D^k & \rTo ^{p^k_X} & X^k \\
\dTo ^{\tilde{w}} &  \Box & \dTo _{{p^k_{X^n}}} & &\\
X^n & \rInto ^{i} & (X^n)^k & & 
\end{diagram}where $\widetilde{D(n,k)}$ is the (non reduced) $k$-fold fibered 
product $ D\times_{X^n} \dots \times_{X^n}D$ of $D$ over $X^n$. 
There is a natural morphism of functors:
\begin{equation*}\label{A}
 \B{L}i^* \circ {p^k_{X^n}}_* \rTo \tilde{w}_* \circ \B{L} j^* \;.\end{equation*}
On the other hand, the epimorphism of sheaves:
\begin{equation}\label{eq:epimorsheaves} \FS_{\widetilde{D(n,k)}} \rOnto \FS_{D(n,k)} \rTo 0\end{equation}induces a morphism of 
Fourier-Mukai functors: $\Phi^{\FS_{\widetilde{D(n,k)}}}_{X^k \ra X^n} \rTo \Phi^{\FS_{D(n,k)}}_{X^k \ra X^n}$. Knowing that $\Phi^{\FS_{\widetilde{D(n,k)}}}_{X^k \ra X^n} \simeq 
\tilde{w}_* \circ \B{L} j^* \circ (p_X^k)^*$ we get a  morphism:
$$ \B{L}i^* \circ {p^k_{X^n}}_*\circ (p_X^k)^* \rTo \Phi^{\FS_{D(n,k)}}_{X^k \ra X^n} \;.$$Applying the previous morphism of functors to $L \boxtimes \dots \boxtimes L$ we get the desired morphism:
$$ \comp{\mc{C}}_{L} \tens^L \cdots \tens^L \comp{\mc{C}}_{L} \simeq 
\B{L}i^* \circ {p^k_{X^n}}_*
\circ (p_X^k)^*(L \boxtimes \cdots \boxtimes L)\rTo^{\alpha}
\bkrh (L^{[n]} \tens \cdots \tens L^{[n]} ) \;.$$

The fact  that $\bkrh (L^{[n]} \tens \cdots \tens L^{[n]} )$ is 
quasi-isomorphic to the sheaf $p_* q^*(L^{[n]} \tens \cdots \tens L^{[n]} )$
comes immediately from Haiman's vanishing theorem \ref{thm: Haimanvanishing} and from the fact that
$L \boxtimes \cdots \boxtimes L$ is locally free. Indeed:
$$ \bkrh (L^{[n]} \tens \cdots \tens L^{[n]} ) \simeq 
\Phi^{\FS_{D(n,k)}}_{X^k \ra X^n}(L \boxtimes \cdots \boxtimes L) \simeq 
\tilde{w}_* (\FS_{D(n,k)} \tens_{\FS_{X^k}} L \boxtimes \cdots \boxtimes L)$$which is a sheaf concentrated in degree $0$. 

In degree $0$ the morphism (\ref{eq:cf}) is an epimorphism since it is the composition:
$$ i^* {p^k_{X^n}}_* (p_X^k)^*(L \boxtimes \cdots \boxtimes L) \simeq 
\tilde{w}_*(\FS_{\widetilde{D(n,k)}} \tens_{\FS_{X^k}} L \boxtimes \cdots \boxtimes L ) \rOnto \tilde{w}_* (\FS_{D(n,k)} \tens_{\FS_{X^k}} L \boxtimes \cdots \boxtimes L)) $$
where the first isomorphism is a consequence of the base change for the square in 
(\ref{numero}), because the vertical arrows are finite and 
$L \boxtimes \dots \boxtimes L$ locally free, and the second epimorphism is induced by the epimorphism (\ref{eq:epimorsheaves}). 

To prove that the sheaf 
$p_* q^*(L^{[n]} \tens \cdots \tens L^{[n]} )$ is torsion free, remark that
it is torsion free as $p_* \FS_{B^n}$-module, since $q^*(L^{[n]} \tens 
\cdots \tens L^{[n]})$ is 
locally free. Hence, by proposition \ref{pps: direct}, it is torsion free as $\FS_{X^n}$-module. 
Now, the epimorphism (\ref{eq:cf}) is an isomorphism out of the diagonal, 
hence the kernel is torsion; since the sheaf $p_*q^*(L^{[n]} \tens \cdots \tens L^{[n]})$ is torsion free, the kernel has to coincide with the torsion 
subsheaf of 
$p_*q^*(L^{[n]}) \tens \cdots \tens p_*q^*(L^{[n]})$.
\end{proof}
\begin{crl}\label{cor:imageidentif}\sloppy
The term $p_* q^*( L^{[n]} \tens \cdots \tens L^{[n]})$ can be
identified with the term $ E^{0,0}_{\infty} $ of the hyperderived
spectral sequence \begin{equation}\label{eq: fss} E^{p,q}_1 := \bigoplus_{i_1 + \dots +i_k=p}
\Tor_{-q}(\mc{C}^{i_1}_{L}, \dots, \mc{C}^{i_k}_{L}) \;,\end{equation}associated to $\comp{\mc{C}}_{L} \tens^L
\cdots \tens^L \comp{\mc{C}}_{L}$ and abutting to the 
hypercohomology 
$ H^{p+q}(\comp{\mc{C}}_{L} \tens^L
\cdots \tens^L 
\comp{\mc{C}}_{L}) \simeq \Tor_{-p-q}(p_*q^*L^{[n]}, \dots,
p_*q^*L^{[n]}) 
$. 
\end{crl}
\begin{proof} 
The term  $E^{0,0}_{\infty}$ is 
torsion-free, because subsheaf of the vector bundle $E^{0,0}_{1} \simeq \mc{C}_L^0 \tens \cdots \tens \mc{C}^0_L$. 
Furthermore the kernel of the 
epimorphism
$H^0(\comp{\mc{C}}_{L} \tens^L
\cdots \tens^L 
\comp{\mc{C}}_{L}) \simeq p_*q^*L^{[n]} \tens \cdots \tens
p_*q^*L^{[n]}
\rOnto E^{0,0}_{\infty}$ is torsion, because its support is contained in the 
union of supports of $E^{p,-p}_1$ for $p>0$, hence in the big diagonal of $X^n$.
Therefore the kernel is exactly the torsion subsheaf, and $E^{0,0}_{\infty}$ can be identified with $p_*q^*( L^{[n]} \tens \cdots \tens L^{[n]})$.
\end{proof}
\begin{remark}
The spectral sequence (\ref{eq: fss}) can be built in a standard way with a procedure analogous to the construction done in \cite{EGA3.2} for the spectral sequence (6.5.4.1)\footnote{in our case, since we use the level $E^{p,q}_1$ instead of $E^{p,q}_2$, we don't take the cohomology $\mc{H}^p$ in \cite[(6.5.4.1)]{EGA3.2}}. 
The spectral sequence $E^{p,q}_1$ is obtained as the standard spectral sequence ${}^{\prime}E^{p,q}_1(\tilde{M}^{\bullet,\bullet}):= H^{q}_{II}(\tilde{M}^{p,\bullet})$
of the double complex $\tilde{M}^{\bullet,\bullet} := R^{\bullet,\bullet}\tens \cdots \tens  R^{\bullet,\bullet}$ (for the sum of the first degrees and for the sum of the second ones) built starting from a (term to term) locally
free resolution $R^{\bullet, \bullet} \rTo \comp{\mc{C}}_L$ of the complex 
$\comp{\mc{C}}_L$ with a double complex $R^{\bullet, \bullet}$. See also~\cite[XVII, 2, ($1^{\prime}$)]{CartanEilenbergHA}. 
\end{remark}

The $\perm_n$-equivariant spectral sequence $E^{p,q}_1$ carries a
very rich piece of information about the tensor power of tautological bundles; however, 
working it out in all generality is hard, due to evident technical and combinatorial issues. 
In the following section we try to extract information from $E^{p,q}_1$ by 
taking its $\perm_n$-invariants.
\section{Invariants}\label{sect:inv}
The aim of this section is the computation the invariants 
$\bkrh(L^{[n]} \tens \cdots \tens L^{[n]})^{\perm_n}$ of the image of 
a tensor power of tautological bundles for the Bridgeland-King-Reid 
transform. By proposition \ref{pps:invamor}, point~1, 
they coincide with the derived direct image of $L^{[n]} \tens \cdots \tens L^{[n]}$ for the Hilbert-Chow morphism $\mu$:
\begin{equation}\label{GGG}
\bkrh(L^{[n]} \tens \cdots \tens L^{[n]})^{\perm_n} := 
\pi^{\perm_n}_* \bkrh(L^{[n]} \tens \cdots \tens L^{[n]}) = \B{R} \mu_* (
L^{[n]} \tens \cdots \tens L^{[n]}) \;.\end{equation}

This information, though  poorer than the one given by the full image, is sufficient to 
compute the cohomology of the Hilbert scheme with values in the 
tensor power of tautological bundles. Indeed, applying 
proposition \ref{pps:invamor}, point 2, to our situation, 
we can compute the cohomology of any coherent sheaf $F$ on 
the Hilbert scheme $X^{[n]}$ as the $\perm_n$-equivariant hypercohomology of 
$X^n$ with values in $\bkrh(F)$, or as the hypercohomology
of $S^nX$ with values in the $\perm_n$-invariants of $\bkrh(F)$:
\begin{equation*}
H^*(X^{[n]},F) \simeq \mbb{H}^*_{\perm_n}(X^n, \bkrh(F)) = 
\mbb{H}^*(S^nX, \bkrh(F)^{\perm_n})\;.\end{equation*}

\subsection{Preliminary lemmas}\label{preliminary}
We recall here Danila's lemma, which will be fundamental in  what follows. In appendix \ref{app: Danila} we discuss a version of this lemma for morphisms.

Let $G$ be a finite group acting transitively on a finite nonempty set $I$. Let $k$ be a field of characteristic zero and $R$ a $k$-algebra. Let $M$ be a $R[G]$-module admitting a decomposition: $M = \oplus_{i \in I} M_i$ compatible with the $G$-action on $I$, that is, for all $i \in I$ and for all $g \in G$, 
the action of $g$ on $M$, restricted to $M_i$, gives an
isomorphism $g :M_i \rTo M_{g(i)}$. For all $g \in \Stab_G(i)$ we have 
$g \in \Aut(M_i)$. Danila's lemma \cite[lemma 2.2]{Danila2001} states:
\begin{lemma}\label{Danila}\sloppy The projection ${\rm pr }_i : M \rTo M_i$ is
  $\Stab_G(i)$-equivariant and induces an isomorphism
$ M^G \rTo M_i^{\Stab_G(i)}$.
\end{lemma}
\begin{remark}\label{danirmk}Suppose  that the $G$ action on $I$ is not transitive. Taking the orbits $\FS_1, \dots, 
\FS_l$ for the $G$-action on $I$ we can always reduce the problem to the case of a transitive action,
  by applying Danila's lemma \ref{Danila} to
  $G$-homogeneous modules $M_{\FS_j} = \oplus_{i \in \FS_j} M_i$.  
\end{remark}
\begin{remark}We will apply Danila's lemma in the following situation.
Let $G$ be a group acting on a variety $Y$ and
on a finite set $I$. Let $F$ be a $G$-sheaf on $Y$, admitting a decomposition 
$F = \oplus_{i \in I}F_i$ compatible with the $G$-action on $I$.
 This
means that for all $i \in I$ and for all $g \in G$ we have canonical isomorphisms:
$ h_g : F_{i} \rTo^{\simeq} g^* F_{g(i)}$. 
Consider the spaces of sections 
$M_i = H^0(F_i)$ and $M = H^0(F)$. 
We inherit morphisms (see \cite[section 2.2]{Danila2001}):
$ \lambda_g: M_i \rTo M_{g(i)}$ setting $\lambda_g(s) := h_g s \circ
g^{-1}
$. 
In particular $M$
becomes a left-$G$-representation, with a decomposition $M = \oplus_{i \in I}M_i$ compatible with the $G$-action on $I$. 
\end{remark}
Since we will be using it throughout the next section we will include the following:
\begin{lemma}\label{lemmino}Let $X$ be a smooth quasi-projective algebraic surface. Let $E \rTo^f F \rTo^g G$ be a 
sequence of sheaves on $X^n$, with ${\rm im} f \subseteq \ker g$. Then the sequence of sheaves on $S^nX$:
\begin{equation}\label{sequenzq} E^{\perm_n} \rTo F^{\perm_n} \rTo G^{\perm_n} \end{equation}is exact if and only if, for all affine open subsets $U$ of $X$, the induced 
sequence
\begin{equation}\label{sequenz2}  H^0(U^n, E)^{\perm_n} \rTo H^0(U^n, F)^{\perm_n} \rTo H^0(U^n, G)^{\perm_n} \end{equation}is exact.
\end{lemma}
\begin{proof}Since, by lemma \ref{qproj}, the affine open sets of the form $S^nU$, where $U$ is an affine open subset of $X$, cover the variety $S^nX$, we can test on them the statement 
	the lemma. The sequence (\ref{sequenzq}) is exact if and only if the sequence $$H^0(S^nU, E^{\perm_n}) \rTo H^0(S^nU,F^{\perm_n}) \rTo H^0(S^nU, G^{\perm_n})$$ is exact for any affine open subset $U$ of $X$. By definition of the $\perm_n$-invariant push-forward $\pi_*^{\perm_n}$ this sequence is exactly the sequence (\ref{sequenz2}).
\end{proof}

\subsection{The case $k=1$.} In the case $k=1$ we already obtained a full 
answer for the image of a tautological bundle $L^{[n]}$ 
for the BKR transform $\bkrh$. Computing the invariants is now very simple: we obtain
another simple proof of the Brion-Danila formula \cite[proposition 6.1]{Danila2001}. Here and in the following, if $A$ is a finite set, we will indicate with $\perm(A)$ the symmetric group of the set $A$. Moreover, if $I \subseteq \{1, \dots,n \}$ is a multi-index, we will denote with $\bar{I}$ its complementary.
\begin{theorem}\label{thm:BDk=1}
Let $X$ be a smooth algebraic surface, $L$ a line bundle on $X$. Let 
$X^{[n]}$ be the Hilbert scheme of $n$ points on $X$, $S^n X$ the
symmetric variety and $\mu: X^{[n]} \rTo S^n X$ the Hilbert-Chow
morphism. Then 
$$ \bkrh(L^{[n]})^{\perm_n} \simeq \B{R} \mu_* ( L^{[n]} ) \simeq 
\pi_*^{\perm_n} \left ( \bigoplus_{i=1}^n L_i \right ) \;.$$
\end{theorem}
\begin{proof}
By theorem \ref{B3}  it suffices to apply the $\perm_n$-invariant push-forward 
$\pi_*^{\perm_n}$
to the complex 
$\comp{\mc{C}}_L$. Since $\pi_*^{\perm_n}$ is an exact functor, the complex $
(\comp{\mc{C}}_L)^{\perm_n}$  provides a right resolution of 
$\mu_*(L^{[n]})$. Then it suffices to prove that, for all $p>0$, 
$(\mc{C}^p_L)^{\perm_n}=0$.
By lemma \ref{lemmino}, this is equivalent to proving that
the $\perm_n$-invariant sections $H^0(U^n, \mc{C}^p_{L})^{\perm_n}$ vanish for $p>0$ for  all affine open sets $U$ of $X$. 
Now, by lemma \ref{Danila}, fixed a  multiindex $I_0$ of length $p+1$:
$$ H^0(U^n, \mc{C}^p_{L} )^{\perm_n}= H^0(U^n, \bigoplus_{|I|=p+1} 
L_I)^{\perm_n} \simeq H^0(U^n, L_{I_0})^{\Stab_{\perm_n}(I_0)} \;.$$ 
Since $\Stab_{\perm_n}(I_0) \simeq \perm(I_0) \times \perm(\overline{I_0})$ 
 and since 
$\perm(I_0)$ acts on the fibers of $L_{I_0}$ over $\Delta_{I_0}$ via the 
alternant representation $\epsilon_{|I_0|}$, we have:
\begin{eqnarray*}  H^0(U^n, L_{I_0})^{\Stab_{\perm_n}(I_0)} & \simeq & \left[ 
H^0(U,L) \tens H^0(U, \FS_X)^{\tens ^{|\overline{I_0}|}} \tens \epsilon_{|I_0|} \right]^{\perm(I_0) \times \perm(\overline{I_0})} \\
& = & \left[ H^0(U,L)  \tens \epsilon_{|I_0|} \right]^{\perm(I_0)} \tens S^{|\overline{I_0}|}H^0(U, \FS_X) =0 \end{eqnarray*}
for $p>0$, where the first equality follows from K\"unneth formula, remembering that 
the sheaf $L_{I_0}$ is supported in $\Delta_{I_0}= \Delta_{X^{I_0}} \times X^{\overline{I_0}}$.
\end{proof}As an easy  consequence, we could now compute the  cohomology $H^*(X^{[n]}, L^{[n]})$ of the Hilbert scheme with values in the tautological bundle $L^{[n]}$,   with 
the same proof as in Danila~\cite{Danila2001}. In section \ref{cohoext} we will give a generalization of this argument to the $k$-fold exterior power 
$\Lambda^kL^{[n]}$.
\subsection{The spectral sequence of invariants}
For $k \geq 2$, we  have to use results from section 
\ref{section3} and in particular the characterization of theorem \ref{treotto}
and corollary \ref{cor:imageidentif}. 
We remarked in \ref{GGG} that the invariants 
$\bkrh(L^{[n]} \tens \cdots \tens 
L^{[n]} )^{\perm_n}$
coincide with the derived direct 
image $\B{R} \mu_* (L^{[n]} \tens \cdots \tens L^{[n]})$ of $L^{[n]} \tens \cdots \tens L^{[n]}$ for the
Hilbert-Chow morphism. Applying the functor $[-]^{\perm_n}$ of 
$\perm_n$-invariants to the morphism $\alpha$ in (\ref{E4}) we get a morphism 
in $\B{D}^b(S^nX)$: 
$$ (\comp{\mc{C}}_{L} \tens^L \cdots 
\tens^L \comp{\mc{C}}_{L})^{\perm_n} \rTo^{\alpha^{\perm_n}} \B{R}\mu_*(L^{[n]} \tens \cdots \tens L^{[n]}) $$whose mapping cone is acyclic in degree $>0$; this implies the vanishing
\begin{equation*} R^{i}\mu_*(L^{[n]} \tens \cdots \tens L^{[n]}) =0 \quad \forall i>0 \;.
\end{equation*}Moreover we have the following characterization of $\mu_*(L^{[n]} \tens \cdots \tens L^{[n]})$.
\begin{crl}\label{crl: inva}The derived direct image $\mu_*(   
 L^{[n]} \tens \cdots \tens 
L^{[n]}) $ of the $k$-fold tensor power $L^{[n]} \tens \cdots  \tens L^{[n]}$ for Hilbert-Chow morphism  can be identified 
with the term $\mc{E}^{0,0}_{\infty}$ of the
spectral sequence $$ \mc{E}^{p,q}_1 = ( E^{p,q}_1)^{\perm_n}$$of $\perm_n$-invariants of the spectral sequence $E^{p,q}_1$, abutting to $H^{p+q}(\comp{\mc{C}}_{L} \tens^L \cdots 
\tens^L \comp{\mc{C}}_{L})^{\perm_n}$.
\end{crl}
The new spectral sequence $\mc{E}^{p,q}_1$ of coherent sheaves over $S^n X$ turns
out to be much simpler than the original one;
it was considered for $k=2$
in \cite{ScalaPhD}. We proved there that it degenerates at level $\mc{E}_2$; this
yields immediately a nice description of the direct image $\mu_*(L^{[n]} \tens L^{[n]})$ in terms of the short 
exact sequence  (\ref{eq:BDk=2}) below. The proof we give here simplifies the arguments of \cite{ScalaPhD} by working out only the restriction of the spectral sequence of invariants $\mc{E}^{p,q}_1$ to a big open subset $S^n_{**}X$ of $S^nX$, which will be precised in the next section. We will see that the information carried by this restriction is already sufficient to draw results 
on the direct image of the general exterior power of tautological bundles.
\subsection{Reduction to the open set $S^n_{**}X$.}
Consider the closed subscheme of $X^n$:
$$ W:=  \bigcup_{\substack{|I|=|I^{\prime}|=2 \\ I \neq I^{\prime}}} \Delta_I \cap \Delta_{I^{\prime}} \;.$$Its irreducible components are the partial diagonals $\Delta_{J}$, $|J| = 3$ and the intersections $\Delta_{I} \cap \Delta_{I^{\prime}}$, $|I|=|I^{\prime}|=2$, $I \cap I^{\prime}= \emptyset$; hence they are smooth of codimension $4$.
\begin{definition}We define the open set
 $X^n_{**}$ of $X^n$ as being the complementary of 
$W$ in $X^n$:
$ X^n_{**} := X^n \setminus W$. It is $\perm_n$-invariant.
Analogously we will define the open subsets  $S^n_{**}X$, 
$X^{[n]}_{**}$ of $S^nX$ and $X^{[n]}$, respectively, as:
$S^nX_{**} := \pi(X^n_{**})=X^n_{**}/\perm_n$, 
$X^{[n]}_{**}:= \mu^{-1}(S^nX_{**})$.
We will indicate with $j_{X^n}$, $j_{S^nX}$, 
$j_{X^{[n]}}$ the open immersions of  open sets $X^n_{**}$, $S^nX_{**}$, 
$X^{[n]}_{**}$ into $X^n$, $S^nX$,  $X^{[n]}$, respectively. 
When there will be no risk of confusion, we will omit the subscripts.
\end{definition}
\begin{remark}Remark that $X^n_{**}$ and $S^n_{**}X$ are complementary of closed of codimension 
$4$ in $X^n$ and $S^nX$, respectively, while 
$X^{[n]}_{**}$ is the
complementary of a closed of codimension $2$ 
in~$X^{[n]}$.
\end{remark}
In the rest of the article we will be concerned with the detailed study of the restrictions $j^* E^{p,q}_1$, $j^* \mc{E}^{p,q}_1$ of the spectral sequences 
$E^{p,q}_1$ and $\mc{E}^{p,q}_1$, to the  open sets $X^n_{**}$ and $S^n_{**}X$, respectively. Since $j^*$ is an exact functor, we obtain other two spectral 
sequences $j^*E^{p,q}_1$, $j^*\mc{E}^{p,q}_1$, associated to the derived tensor product $j^*\comp{\mc{C}}_{L} \tens^L 
\cdots 
\tens^L j^*\comp{\mc{C}}_{L}$ and its invariants $(j^*\comp{\mc{C}}_{L} \tens^L 
\cdots 
\tens^L j^*\comp{\mc{C}}_{L})^{\perm_n}$, and abutting to 
$H^{p+q}(j^*\comp{\mc{C}}_{L} \tens^L 
\cdots 
\tens^L j^*\comp{\mc{C}}_{L})$ and 
$H^{p+q}(j^*\comp{\mc{C}}_{L} \tens^L 
\cdots 
\tens^L j^*\comp{\mc{C}}_{L})^{\perm_n}$, respectively.  The 
big advantage of these spectral sequences is that now the complex 
$j^* \comp{\mc{C}}_L$ is the $2$-term complex, concentrated in degrees~$0$ and~$1$: $$ 0 \rTo j^* \mc{C}^0_L \rTo 
j^* \mc{C}^1_L \rTo 0 \; ;$$hence the terms $j^* E^{p,q}_1$ and $j^*\mc{E}^{p,q}_1$ vanish if 
$p>n$.
\begin{crl}
The sheaf $j^*\mu_*(L^{[n]} \tens \cdots \tens  L^{[n]})$ can be identified with 
the term $j^* \mc{E}^{0,0}_{\infty}$ of the spectral sequence 
$j^* \mc{E}^{p,q}_1 = (j^* E^{p,q}_1)^{\perm_n}$.
\end{crl}

The following lemmas allow to deduce results on the whole $S^nX$ from the results on the restriction to the open set $S^n_{**}X$. 
\begin{lemma}\label{414}Consider the open immersion: 
$ j : S^nX_{**} \rInto S^nX $. Then, for any vector bundle $A$ on $X^{[n]}$,
$ \mu_* (A)  \simeq j_* j^* (\mu_* A)$.
\end{lemma}
\begin{proof}Consider the commutative diagram:
\begin{diagram}
X^{[n]}_{**} & \rInto^j & X^{[n]} \\
\dTo^{\mu} & & \dTo^{\mu} \\
S^nX_{**} & \rInto ^j & S^nX 
\end{diagram}It is a flat change base,
hence $ j^* (  \mu_* A) \simeq  \mu_* ( j^* A)$. 
Now, applying $j_*$, we have: 
$$ j_* j^* (  \mu_* A) \simeq  j_* \mu_* ( j^* A) \simeq \mu_* j_* j^* A \simeq 
\mu_* A $$because $j_*j^*A \simeq A$, since $A$ is a vector bundle and 
$X^{[n]}_{**}$ is the complementary of a closed subset of codimension $2$ in $X^{[n]}$. \end{proof}
\begin{lemma}\label{415}Let $A$ be a $\perm_n$-vector bundle on $X^n$. Then 
$ j_*j^* A^{\perm_n} \simeq A^{\perm_n}$.
\end{lemma}
\begin{proof} By flat base change we prove again that $j^*$ commutes with 
$\pi_*$. It is clear that $j_*$ commutes with $\pi_*$ as well. 
The lemma follows from the fact that both $j^*$ and $j_*$ 
commute with $[-]^{\perm_n}$.
Then $ j_*j^* A^{\perm_n} \simeq (j_*j^* A)^{\perm_n} \simeq A^{\perm_n}$.
\end{proof}
\subsubsection{The spectral sequences $j^*E^{p,q}_1$ and $j^* \mc{E}^{p,q}_1$}
We make now some remarks on the terms $j^* E^{p,q}_1$. We introduce first  some notations.
\begin{notat}\label{not: puremultitor}
Let $Y$ be an algebraic variety and $F \in \Coh(Y)$. We will denote with 
$\Tor_q ^l(F)$ the $l$-fold multitor
$$ 	\Tor_q ^l(F) := \Tor_q(\underset{l-{\rm times}}{\underbrace{F,\cdots, F}} ) \;.
$$
\end{notat}
\begin{notat}\label{notaziona}
Let $\mc{P}:= \{ I \; | \; \emptyset \neq I \subseteq \{1, \dots, n \} \}$. 
For each $a \in \mc{P}^{\{1, \dots,k \}}$, $k \in \mbb{N}^*$, we will indicate with $I_0(a)$ the multi-index 
$I_0(a) := \cup_{i \, | \, |a(i)| \geq 2 } \, a(i)$, 
$l(a) = \sum_{i=1}^k |a(i)| -k$ and
$J(a) = \cup_{i \, | \, |a(i)|=1} \, a(i)$. We finally set 
$S_0(a):= \{i \in \{1,\dots,k \} \; | \; |a(i)| \geq 2 \}$
and $\lambda(a) : \{0\} \cup J(a) \rTo \mbb{N}$ the function defined by : $\lambda_0(a) = |S_0(a)|$, $\lambda_j(a) = 
|a^{-1}(\{j \})|$ for all\footnote{Here and in the following we consider $J(a)$ as a subset of $\{1, \dots,n\}$.} $j \in J(a)$. 
\end{notat}
\begin{remark}\label{remarchetto}Each term $j^*E^{p,q}_1$ can be written as the direct sum:
\begin{eqnarray*} 
j^*E^{p,q}_1  & = & \bigoplus_{i_1 + \dots + i_k =p} j^* \Tor_{-q}(\mc{C}^{i_1}_L, \dots, \mc{C}^{i_k}_L)  =  \bigoplus_{i_1 + \dots + i_k =p} \bigoplus_{|I_j|=i_j +1 } j^*\Tor_{-q}
(L_{I_1}, \dots, L_{I_k}) \\
& = &\bigoplus_{a \in \mc{P}^{\{1, \dots,k \}}, \; l(a)=p} j^*\Tor_{-q}(L_{a(1)}, \dots, L_{a(k)})
\end{eqnarray*}The support of each term $\Tor_{-q}
(L_{a(1)}, \dots, L_{a(k)})$ is contained in $\cap_{i \in S_0(a)} \Delta_{a(i)}$ if
$I_0(a) \neq \emptyset$, otherwise it is the whole $X^n$. Therefore $j^* \Tor_{-q}(L_{a(1)}, \dots, L_{a(k)})=0$ if $|I_0(a)| \gneq 2$. Consequently we have:
\begin{equation} \label{ee} j^*E^{p,q}_1 = \bigoplus_{\substack{a \in \mc{P}^{\{1, \dots,k\}}, \; l(a)=p \\ |I_0(a)|  \leq 2}} j^*\Tor_{-q}(L_{a(1)}, \dots, L_{a(k)}) \;.
\end{equation}
Motivated by this fact we define 
$\mc{I}^p  = \{ a  \in \mc{P}^{\{1, \dots,k\}}, \; l(a)=p \, , \; |I_0(a)|  \leq 2 \}$, and $(E^{p,q}_1)_0$ the subsheaf of $E^{p,q}_1$ given by: 
\begin{equation}\label{def: e} (E^{p,q}_1)_0 := \bigoplus_{a \in \mc{I}^p} \Tor_{-q}(L_{a(1)}, \dots, L_{a(k)}) \;.\end{equation}By definition of $(E^{p,q}_1)_0$  equation (\ref{ee}) becomes 
\begin{equation} \label{DDD}
j^* E^{p,q}_1 = j^* ( E^{p,q}_1)_0 \;.\end{equation}
It is immediate to see that 
$ (E^{0,0}_1)_0 = E^{0,0}_1$
and that $(E^{0,q}_1)_0 =0$ if $q <0$. Moreover, if $1 \leq p \leq k$ and $a \in \mc{I}^p$, then $|I_0(a)|=2$; hence, if $1 \leq p \leq k$:
$$ (E^{p,q}_1)_0 = \bigoplus_{a \in \mc{I}^p}  \Tor_{-q}(L_{a(1)}, \dots, L_{a(k)}) \simeq  \bigoplus_{a \in \mc{I}^p}
 \Tor_{-q}^p(L_{I_0(a)}) \tens \Tens_{j\in J(a)} L^{\tens^ {\lambda_j(a)}}_{j} \;.$$
We will prove in lemma \ref{torinv} that 
$ \Tor_{-q}^p(L_{I_0(a)}) \simeq 
\Lambda^{-q}(N_{I_0(a)}^* \tens \rho_p) \tens L_{I_0(a)}^{p}$, with 
$N_{I_0(a)}$ the normal bundle of $\Delta_{I_0(a)}$ in $X^n$, and 
$\rho_p$ the standard representation of the symmetric group $\perm_p$. Therefore the sheaf $(E^{p,q}_1)_0$ is a direct sum of vector bundles restricted to pairwise diagonals $\Delta_{I_0(a)}$, $a \in \mc{I}^p$, hence Cohen-Macauley and 
pure of codimension 2. It vanishes exactly when $q < 2(1-p)$.  
It is clear that $(E^{p,q}_1)_0=0$ if $p >k$.

\sloppy The differentials $d^{p,q}_1$ of the spectral sequence 
$E^{p,q}_1$ induce differentials $(d^{p,q}_1)_0: (E^{p,q}_1)_0 \rTo 
(E^{p+1,q}_1)_0$, making $(E^{\bullet,q}_1)_0$ into a complex;  the natural projection $E^{\bullet,q}_1 \rTo (E^{\bullet,q}_1)_0,$ is then an epimorphism of 
complexes. 
Since the 
$G$-action on the complex $E^{\bullet,q}_1$ preserves the dimensional components, the complex $(E^{\bullet,q}_1)_0$ inherit the $G$-linearization. We can therefore define 
the complexes $(\mc{E}^{\bullet,q}_1)_0$ as the $\perm_n$-invariants $(\mc{E}^{\bullet,q}_1)_0 := (E^{\bullet,q}_1)_0^{\perm_n}$ of the complexes $(E^{\bullet,q}_1)_0$. \end{remark}

We have the following lemma:
\begin{lemma}\label{localcoho}
Let $X$ be a smooth algebraic variety, $F \in \Coh(X)$, such that $F = 
\oplus_{i=1}^l F_i$, with $F_i$ vector bundles on smooth subvarieties $Z_i$ of $X$, extended by zero on $X \setminus Z_i$. Let  $Y$ be a  closed subscheme, $Y \subseteq \cap_i Z_i$, with 
$\codim_{Z_i}Y \geq 2$ for all $i$. Then, if $j$ is the open immersion $j : X \setminus Y \rInto X$, we have that $ j_*j^*F =F$. 
\end{lemma}
\begin{proof}Recall that for a locally noetherian scheme $X$, 
$Y$ a closed subscheme, $E \in \Coh(X)$, we have the following condition on the local cohomology module $\mc{H}^i_Y(E)$ (see \cite{HartshorneLC}):
$$ \forall \; \mf{p} \in Y \quad \depth E_{\mf{p}} \geq k \iff 
\mc{H}_Y^i(E) =0 \quad \forall \; i <k \;,$$where here $\mf{p}$ indicates a \emph{scheme-theoretic point} of $Y$.
If $E$ is a Cohen-Macauley 
sheaf and $j$ is the open immersion $j : X \setminus Y \rInto X$,
the above condition implies: 
$$ \forall \; \mf{p} \in Y \quad \dim_{\FS_{X, \mf{p}}} 
E_{\mf{p}} \geq 2 \iff 
j_* j^* E = E \;.$$Since every sheaf $F_i$ in the hypothesis of the lemma is Cohen-Macauley, because a vector bundle on the smooth subvariety $Z_i$, and since 
$ \dim_{\FS_{X, \mf{p}}} 
(F_i)_{\mf{p}} \geq 2 $ because of the hypothesis $\codim_{Z_i} Y \geq 2$, we can 
immediately infer from the facts on local cohomology that $j_*j^*F_i = F_i$ for all $i$, from which the lemma follows. 
\end{proof}
\begin{lemma}\label{le}We have the following relations:
$$ ( E^{p,q}_1)_0 = j_* j^* E^{p,q}_1 \qquad , \qquad  ( \mc{E}^{p,q}_1)_0 = j_* j^*
 \mc{E}^{p,q}_1 \;.$$
\end{lemma}
\begin{proof}The first relation  follows easily from  lemma \ref{localcoho} remarking firstly that $(E^{p,q}_1)_0$ is a 
vector bundle or a direct sum of vector bundles supported in  diagonals 
$\Delta_I$, $|I|=2$, which are
$2$-codimensional smooth subvarieties; secondly, that, if $|I|=2$,  $\codim_{\Delta_I} Z \cap \Delta_{I} \geq 2$ for every irreducible component $Z$ of $W$ intersecting $\Delta_I$.  The second relation follows from the first taking $\perm_n$-invariants. 
\end{proof}
\subsection{The case $k=2$}
Remarks made so far about the spectral sequences $E^{p,q}_1$, $\mc{E}^{p,q}_1$ and 
their 
restrictions to the open subsets $X^n_{**}$, $S^n_{**}X$, respectively, are sufficient to draw conclusions on the sheaf
$\mu_*(L^{[n]} \tens L^{[n]})$.
The method we follow consists in getting the result on $S^n_{**}X$ and 
then in extending it on the whole $S^nX$ by applying the functor $j_*$ using lemmas \ref{414}, \ref{415}. We now study in details the spectral sequence $j^* 
\mc{E}^{p,q}_1= (j^* E^{p,q}_1)^{\perm_n}$ for $k=2$.
\begin{pps}\label{318}
For $k=2$ the spectral sequence $j^* \mc{E}^{p,q}_1$ degenerates at level 
$\mc{E}_2$. 
\end{pps}
\begin{proof}Taking $\perm_n$-invariants in  (\ref{DDD}) we get
$j^* \mc{E}^{p,q}_1 = j^* (\mc{E}^{p,q}_1)_0$. For $q <0$, by remark \ref{remarchetto},   
$(\mc{E}^{p,q}_1)_0 \neq 0$ only if $0 \leq p \leq k$, 
$2(1-p) \leq q \leq 0$. We  analyse complexes 
$j^*\mc{E}^{\bullet ,q}_1$ for $q=-1,-2$. 
For $q=-1$ the only possibility is $j^* \mc{E}^{2,-1}_1$. 
We prove that it vanishes. To do this it is sufficient to prove that 
$(E^{2,-1}_1)_0^{\perm_n}=0$. By lemma \ref{lemmino}, it is sufficient to prove 
that $H^0(U^n,(E^{2,-1}_1)_0^{\perm_n} )=0$ for every affine open subset $U$
of $X$. 
Since 
$(E^{2,-1}_1)_0 = \bigoplus_{|I|=2} L_I^{\tens^2}\tens N_{I}^* $
we have: 
$$ 
H^0(U^n ,(E^{2,-1}_1)_0)^{\perm_n} = \left[ 
\bigoplus_{|I|=2} H^0(U^n, L_I^{\tens^2} \tens N^*_{I}) 
\right] ^{\perm_n} =   H^0(U^n, L_{12}^{\tens ^2} \tens 
N_{12}^*)^{\Stab_{\perm_n}(\{1,2\})} \;,$$applying Danila's lemma \ref{Danila} keeping 
into account that there is only one homogeneous component\footnote{there is only one $\perm_n$-orbit in the set of 
multi-indexes of length $2$}
for the action of 
$\perm_n$, indexed, for example, by the multi-index $\{1, 2\}$. 
Now $\Stab_{\perm_n}(\{ 1, 2 \}) \simeq \perm(\{ 1, 2 \}) \times \perm(
\overline{ \{ 1, 2 \}})$. 
The factor 
$\perm(\{1,2\})$
acts on $L_{12}$ with the alternant representation $\epsilon_{2}$; it acts
moreover fiberwise on $N_{12}^*$ with the representation $\epsilon_2 
\oplus \epsilon_{2}$. 
Therefore $\perm(\{1,2\})$ acts fiberwise on $L_{12}^{\tens 2} \tens 
N_{12}^*$ with the representation $\epsilon_2^{\tens 2} \tens 
(\epsilon_2 
\oplus \epsilon_{2})= \id \tens (\epsilon_2 \oplus \epsilon_2)$. Consequently
there are no $\perm(\{1,2\})$-invariant sections: hence $H^0(U^n, L_{12}^{\tens ^2}~\tens~N_{12}^*)^{\Stab_{\perm_n}(\{1,2\})} = 
0$ and $j^* \mc{E}^{2,-1}_1 =0$. 
The remaining term $j^* \mc{E}^{2,-2}_1$ is on the diagonal.
This proves that 
the spectral sequence degenerates at level $\mc{E}_2$.
\end{proof}
\begin{theorem}\label{thm:BDk=2}Let $X$ be a smooth quasi-projective algebraic surface, $L$ a line bundle on $X$. Let 
$X^{[n]}$ be the Hilbert scheme of $n$ points on $X$, $S^n X$ the
symmetric variety and $\mu: X^{[n]} \rTo S^n X$ the Hilbert-Chow
morphism.
We have the exact sequence:  
\begin{equation}\label{eq:BDk=2}
0 \rTo \mu_*(L^{[n]} \tens L^{[n]}) \rTo (\mc{C}^0_L \tens 
\mc{C}^0_L)^{\perm_n} \rTo^{(\del^0_L \tens \id)^{\perm_n}} (\mc{C}^1_L \tens 
\mc{C}^0_L)^{\perm_n} \rTo 0
\end{equation}
\end{theorem}
\begin{proof}The complex $j^* \mc{E}^{\bullet, 0}_1$ is the complex: 
$$ 0  \rTo (j^* \mc{C}^0_L \tens j^* \mc{C}^0_L)^{\perm_n} 
\rTo \begin{array}{c}
(j^* \mc{C}^0_L \tens j^* \mc{C}^1_L)^{\perm_n} \\ \bigoplus \\ 
(j^* \mc{C}^1_L \tens j^* \mc{C}^0_L)^{\perm_n} \end{array} \rTo^{(\partial^0_L \tens \id - \id \tens \partial^0_L)^{\perm_n}} 
(j^* \mc{C}^1_L \tens 
j^* \mc{C}^1_L )^{\perm_n} \rTo 0 \;.$$ 
We know \emph{a priori}, by the degeneration of 
the spectral sequence $j^* \mc{E}^{p,q}_1$ at level $\mc{E}_2$ and by the fact that the limit is zero above 
the diagonal, that this complex is exact in degree $>0$ and its 
cohomology in degree $0$ is $j^* \mu_*(L^{[n]} \tens L^{[n]})$. 
We show now that we have isomorphisms: \
$$ (j^* \mc{C}^0_L \tens j^* \mc{C}^1_L)^{\perm_n} \rTo^{(\partial^0_L \tens \id )^{\perm_n} }_{\simeq} (j^* \mc{C}^1_L \tens 
j^* \mc{C}^1_L )^{\perm_n} \lTo^{(\id \tens \partial^0_L)^{\perm_n}}_{\simeq} (j^* \mc{C}^1_L \tens j^* \mc{C}^0_L)^{\perm_n} \;.$$Since $j^*$ is exact, 
it is sufficient to 
prove that we have isomorphisms on $S^nX$: 
$$(\mc{C}^0_L \tens \mc{C}^1_L)^{\perm_n} \rTo^{(\partial^0_L 
\tens \id )^{\perm_n} }_{\simeq} (E^{2,0}_1)_0^{\perm_n} \lTo^{(\id \tens \partial^0_L)^{\perm_n}}_{\simeq} 
(\mc{C}^1_L \tens \mc{C}^0_L)^{\perm_n} $$By lemma \ref{lemmino} this is equivalent to proving that, for any affine open set $U$ of $X$ we have isomorphisms: 
$$ H^0(U^n,\mc{C}^0_L \tens \mc{C}^1_L)^{\perm_n}  \rTo ^{\simeq} 
H^0(U^n,(E^{2,0}_1 )_0)^{\perm_n}  \lTo^{\simeq} H^0(U^n,\mc{C}^1_L \tens 
\mc{C}^0_L)^{\perm_n} \; ;$$where the arrows are 
induced by the arrows above. Now: 
$$ 
H^0(U^n ,\mc{C}^0_L \tens \mc{C}^1_L )^{\perm_n} = \left[ 
\bigoplus_{i, |I|=2} H^0(U^n, L_i \tens L_I)\right]^{\perm_n} \;.$$ 
There are 
two homogeneous components for the action of $\perm_n$, since there are two 
$\perm_n$-orbits in the set $\{ 1, \dots, n\} \times \{I \subseteq \{1, \dots,n \} \; | \; |I|=2 \}$. They are indexed, for example, by $\{1\}, \{1,2\}$ and 
$\{1 \}, \{2,3\}$. Note that the stabilizer for the first couple is $\perm(\overline{\{1,2\}})$, while for the second is $\perm(\{2,3\}) \times 
\perm(\overline{\{1,2,3\}})$. Therefore: 
$$ H^0(U^n,\mc{C}^0_L \tens \mc{C}^1_L)^{\perm_n}  = 
H^0(U^n, L_1 \tens L_{12})^{\perm(\overline{\{1,2\}})} \bigoplus 
H^0(U^n, L_1 \tens L_{23})^{\perm(\{2,3\}) \times \perm(\overline{\{1,2,3\}})}\;,$$but now the second summand vanish, since $\perm(\{2,3\})$ acts with the alternant representation $\epsilon_2$ on $L_{23}$. Consequently:
$$ H^0(U^n,\mc{C}^0_L \tens \mc{C}^1_L)^{\perm_n}  = 
H^0(U, L^{\tens^2}) \tens S^{n-2}H^0(U, \FS_X) \;.$$The 
same computation works for 
$H^0(U^n,\mc{C}^1_L \tens \mc{C}^0_L)^{\perm_n} $. Finally: 
$$ H^0(U^n,(E^{2,0}_1 )_0)^{\perm_n}  = \left[ \bigoplus_{|I|=2} 
H^0(U^n, L_I^{\tens^2} ) \right]^{\perm_n} \simeq H^0(U^n, L_{12}^{\tens^2} 
)^{\perm(\{1,2\}) \times \perm(\overline{\{1,2\}})}$$because there is only one homogeneous component. Note that 
the action of $\perm(\{1,2\})$ is given by the representation 
$\epsilon_2^{\tens 2}=1$, hence trivial. Therefore: 
$$ H^0(U^n,(E^{2,0}_1 )_0)^{\perm_n}  \simeq H^0(U, L^{\tens^2}) \tens S^{n-2}H^0(U, \FS_X) \;.$$The induced maps are the identities. 
Therefore we have the wanted isomorphism. 

Consider now the complexes:
\begin{gather*} \comp{\mc{A}} \colon 0 \rTo (\mc{C}^0_L \tens \mc{C}^0_L)^{\perm_n} \rTo^{(\partial^0_L \tens \id )^{\perm_n} } (\mc{C}^1_L \tens 
\mc{C}^0_L )^{\perm_n} \rTo 0 \\
\comp{\mc{B}}\colon 0 \rTo (\mc{C}^0_L \tens \mc{C}^1_L)^{\perm_n} \rTo^{(\partial^0_L \tens \id )^{\perm_n} } (\mc{C}^1_L \tens 
\mc{C}^1_L )^{\perm_n} \rTo 0 \;.
\end{gather*}
The complex
$j^* \mc{E}^{\bullet, 0}_1$ coincides, up to a shift, with the mapping cone 
of the natural morphism $j^*\comp{\mc{A}} \rTo j^* \comp{\mc{B}}$. 
Since we just proved that $j^*\mc{B}$ is 
quasi-isomorphic to $0$, then $j^* \comp{\mc{A}} \simeq j^* 
\mc{E}^{\bullet, 0}_1$. Therefore $j^*\comp{\mc{A}}$ gives a resolution $ j^* \mu_*(L^{[n]} \tens L^{[n]}) \rTo ^{\simeq } j^* \comp{\mc{A}} $; consequently, applying the functor $j_*$ we obtain an exact sequence:
$$ 0 \rTo \mu_*(L^{[n]} \tens L^{[n]}) \rTo (\mc{C}^0_L \tens \mc{C}^0_L)^{\perm_n} \rTo^{(\partial^0_L \tens \id )^{\perm_n} } (\mc{C}^1_L \tens 
\mc{C}^0_L )^{\perm_n} \;,$$since, by lemma \ref{le},
$j_*j^* (\mc{C}^1_L \tens 
\mc{C}^0_L )^{\perm_n} = (\mc{C}^1_L \tens 
\mc{C}^0_L )^{\perm_n}$.
Now the map $(\partial^0_L \tens \id )^{\perm_n}$ is surjective, since, applying lemma \ref{lemmino}, it induces the following surjection on the sections on affine open subsets $U^n$: 
$$ \begin{array}{c}
H^0(U, L^{\tens ^2}) \tens S^{n-1}H^0(U, \FS_X) \\
\bigoplus \\
H^0(U,L)^{\tens ^2}\tens S^{n-2}H^0(U, \FS_X) \end{array} \rOnto H^0(U,L^{\tens^2}) 
\tens S^{n-2}H^0(U, \FS_X) \;.$$Therefore we get the wanted exact sequence in the statement of the theorem. 
\end{proof}
\section{Exterior powers of  tautological bundles}\label{sect:ext}
In the previous section we studied the invariants 
$\bkrh(L^{[n]}\tens \cdots \tens L^{[n]})^{\perm_n}$ of the image of a 
tensor power of a tautological bundle, or, equivalently, their direct 
image $\mu_*(L^{[n]}\tens \cdots \tens L^{[n]})$ for the 
Hilbert-Chow morphism $\mu$. The reduction to the open set $S^n_{**}X$ 
allowed to compute explicitely the invariants for $k=2$. For higher $k$, 
however, it is still quite difficult to get informations from the 
spectral sequence $j^* \mc{E}^{p,q}_1$, for the simple reason that it does 
not degenerate at level $\mc{E}_2$. In this section we introduce a further simplification, the idea being to consider the $k$-fold tensor power of 
tautological bundles  $L^{[n]}\tens \cdots \tens L^{[n]}$ as a 
$\perm_k$-equivariant bundle on the Hilbert scheme, for the action of $\perm_k$
that permutes the factors, and to decompose 
it in
Schur functors 
$S^{\lambda}L^{[n]}$. This method is effectively useful to compute 
the direct image of the simplest Schur functor, the exterior power 
$\Lambda^k L^{[n]}$.  
\subsection{Derived Schur functors}\label{derivedperm}
Consider the category of finite modules over a commutative ring, or of coherent sheaves over a quasi-projective variety.
The aim of this subsection is to
describe how the group $\perm_k$ acts on the $k$-fold tensor power $\comp{C}
\tens \cdots \tens \comp{C}$ of a complex $\comp{C}$, and to extend this
action to the derived tensor power $\comp{C}
\tens^L \cdots \tens^L \comp{C}$; we will finally  
define general Schur functors of a complex 
of locally free sheaves and its derived version. 

\subsubsection*{Derived $\perm_k$-action}The action of $\perm_k$ is fully
understood once it is understood on \emph{consecutive transpositions}. For the 
double tensor power 
of a complex $\comp{C}$
the right involution on $\comp{C} \tens \comp{C}$
is defined by
$ i(u \tens v) := (-1)^{pq} v \tens u$ if $u \in C^p$, $v
  \in C^q$; the sign $(-1)^{pq}$ balances the sign in the definition
of the differential 
$d_{\comp{C} \tens \comp{C}}^n := \oplus_{p+q=n} \left[    
d_{\comp{C}}^p \tens \id_{C^q} + (-1)^p \id_{C^p} \tens d_{\comp{C}}^q 
\right ]$. 
Suppose now that the complex $\comp{C}$ is right bounded.
To extend this action to the derived double tensor power  $\comp{C} \tens^L
  \comp{C}$ it suffices to replace the complex $\comp{C}$ by a
  $\tens$-acyclic or projective resolution $\comp{K}$ of the complex
  $\comp{C}$ and to take the involution just defined on $\comp{K}
  \tens \comp{K}$. 
\sloppy
Consider now the $k$-fold derived
  tensor power $\comp{C} \tens ^L \cdots \tens ^L \comp{C}$. If
  $\comp{K}$ is an $\tens$-acyclic resolution of $\comp{C}$, the group
  $\perm_k$ acts on the derived tensor product $\comp{C} \tens ^L \cdots \tens
  ^L \comp{C} \simeq \comp{K} \tens \cdots \tens \comp{K}$ by
  permutation of the factors $\comp{K}$, where the action of a transposition on
  two consecutive factors is exactly the one described above. 
This action does not
  depend on the choice of the resolution $\comp{K}$.
\begin{remark}\label{rmk:ssaction}We want now to understand how the 
trasposition of two consecutive terms acts 
on the hyperderived spectral sequence 
\begin{equation}\label{eq: css} E^{p,q}_1 = \bigoplus_{i_1+\dots +i_k =p} \Tor_{-q}(C^{i_1}, \dots, 
C^{i_k}) \end{equation}
associated to $\comp{C} \tens^L \cdots 
 \tens^L \comp{C}$. 
In the following we will always consider multiple complexes 
$L^{\bullet, \dots, \bullet}$ with \emph{commuting partial differentials} $\del_i$;
the signs will be only introduced in the differential of the total complex
$\tot L^{\bullet, \dots, \bullet}$, which is defined, on the component 
$L^{i_1, \dots, i_k}$, by
$d= \sum_{j=1}^k(-1)^{\sum_{l=1}^{j-1}i_l}\del_j$.
Consider now projective resolutions
$R^{i, \bullet}$  of the terms $C^i$; form the double complex $R^{\bullet,\bullet}$, with commuting differentials $d^{\prime}$ and $d^{\prime \prime}$: its total complex 
$\comp{K} := \tot \bicomp{R}$ is a projective resolution of the complex 
$\comp{C}$. The multiple complex $\bicomp{R} \tens \cdots \tens 
\bicomp{R}$, $k$-fold tensor product of the double complex $\bicomp{R}$, 
has commuting partial differentials \begin{gather*}
d_i^{\prime} := \underset{i-1-{\rm times}}{\underbrace{\id \tens \cdots \tens  \id }}\tens d^{\prime} \tens \id \tens \cdots \tens \id \qquad
d_i^{\prime \prime} = \underset{i-1-{\rm times}}{\underbrace{\id \tens \cdots \tens  \id }}\tens d^{\prime \prime} \tens \id \tens \cdots \tens \id \;.
\end{gather*}Consider now the following two multiple complexes, obtained by contracting partially this $2k$-complex. 
Let $\mcomp{\tilde{M}}$ be the $k+1$-complex
defined by: 
$$\tilde{M}^{i_1, \dots,i_k,q} = \bigoplus_{h_1+\dots+h_k=q}R^{i_1,h_1} \tens \cdots \tens  R^{i_k,h_k}$$with commuting partial differentials 
$\del^{\prime}_{\tilde{M},i} := (-1)^{\epsilon^{\second}_j} d_j^{\prime}$; $\del^{\second}_{\tilde{M}} := (-1)^{\sum_{l=1}^k i_l}\sum_{j=1}^k (-1)^{\epsilon_j^{\prime} + \epsilon_j^{\second} +i_j} d^{\second}_j$, where $\epsilon_j^{\prime}= \sum_{l=1}^{j-1}i_l$ and $\epsilon_j^{\second} = \sum_{l=1}^{j-1} h_l$.
Let moreover 
$\bicomp{\tilde{N}}$ be the bicomplex $$\tilde{N}^{p,q} := \bigoplus_{i_1+\dots+i_k=p}
\tilde{M}^{i_1, \dots,i_k,q}$$ with commuting differentials
$d^{\prime}_{\tilde{N}}= \sum_{j=1}^k
(-1)^{\epsilon^{\prime}_j} \del^{\prime}_{\tilde{M},j}= \sum_{j=1}^k(-1)^{\epsilon_j^{\prime} + \epsilon_j^{\second}} d^{\prime}_j$ and $d^{\second}_{\tilde{N}} = \del^{\second}_{\tilde{M}}$. Remark that
the derived tensor product ${\comp{C} \tens ^L \cdots \tens^L \comp{C}} \simeq 
\comp{K} \tens \cdots \tens \comp{K}$ coincides  with the total complexes:
$$ \comp{K} \tens \cdots \tens \comp{K} = \tot (\bicomp{R} \tens \cdots \tens 
\bicomp{R}) = \tot (\mcomp{\tilde{M}}) = \tot (\bicomp{\tilde{N}}) \;.$$
The spectral sequence (\ref{eq: css}) is then exactly 
the spectral sequence ${}^{\prime}E^{p,q}_1(\bicomp{\tilde{N}}):= H^{q}_{II}(\tilde{N}^{p,\bullet}) = \oplus_{i_1 + \dots +i_k=p} H^q_{(k+1)}(\tilde{M}^{i_1, \dots,i_k,\bullet})$, 
associated to the double complex $\bicomp{\tilde{N}}$ (see \cite[XVII, 2]{CartanEilenbergHA}). The complex $E^{\bullet,q}_1$ is exactly the complex
$\tot H^q_{(k+1)}(\mcomp{\tilde{M}})$, where the cohomology 
$H^q_{(k+1)}$ is taken on the last component.

The transposition of factors $\tau_{j,j+1}$ acts on the term $\tilde{N}^{p,q}$
with the involution $\widehat{\tau_{j,j+1}}$
mapping: 
$$ \widehat{\tau_{j,j+1}}: 
R^{i_1, h_1} \tens \cdots \tens R^{i_j, h_j} \tens 
R^{i_{j+1}, h_{j+1}} \tens \cdots \tens R^{i_k, h_k} \rTo R^{i_1, h_1} \tens \cdots \tens R^{i_{j+1}, h_{j+1}} \tens 
R^{i_{j}, h_j} \tens \cdots \tens R^{i_k, h_k}$$and defined by
\begin{equation}\label{eq: invoss}\widehat{\tau_{j,j+1}}(u_1 \tens \cdots \tens  u_j \tens u_{j+1} \tens \cdots \tens u_{k} ) =
 (-1)^{(i_j+h_j)(i_{j+1} +h_{j+1})} u_1 \tens \cdots \tens  u_{j+1} \tens u_{j} \tens \cdots \tens u_{k} \;.\end{equation}The map $\widehat{\tau_{j,j+1}}$ is an automorphism 
of the double complex $\bicomp{\tilde{N}}$:
 we deduce an action $\widehat{\tau_{j,j+1}}$ 
on the spectral sequence (\ref{eq: css}) above, exchanging the multitors:
\begin{equation}\label{eq: invotorss}  \widehat{\tau_{j,j+1}}: \Tor_{-q}(C^{i_1}, \dots, C^{i_j}, C^{i_{j+1}}, C^{i_k}) \rTo \Tor_{-q}(C^{i_1}, \dots, C^{i_{j+1}}, C^{i_j}, C^{i_k})\end{equation}as the application 
induced induced in the $q$-cohomology by the isomorphism of complexes:
$$\widehat{\tau_{j,j+1}}: (\tilde{M}^{i_1, \dots, i_j,i_{j+1}, \dots, i_k,\bullet},
d^{\second}_{\tilde{M}})
 \rTo  (\tilde{M}^{i_1, \dots, i_{j+1},i_{j}, \dots, i_k,\bullet}, d^{\second}_{\tilde{M}})$$defined by  (\ref{eq: invoss}).

In order to clearly interprete this action and to compare (\ref{eq: invotorss})
with the permutation of factors 
$\widetilde{\tau_{j,j+1}}$ in the multitor $\Tor_{-q}(C^{i_1}, \dots, C^{i_j}, 
C^{i_{j+1}}, \dots, C^{i_k})$ studied in appendix \ref{app: perm}, where 
the terms $C^{i_l}$ are considered just as sheaves and not as elements of a 
complex, we introduce the more natural multiple complex $\mcomp{M}$, defined as
$M^{i_1, \dots,i_k,q} := \tilde{M}^{i_1, \dots,i_k,q}$, but with commuting partial 
differentials $\del^{\prime}_{M,i}=d^{\prime}_i$, $\del^{\second}_M := \sum_{j=1}^k 
(-1)^{\epsilon_j^{\second}} d^{\second}_i$. Consider moreover the double complex $\bicomp{N}$, 
defined as: $N^{p,q}:= \tilde{N}^{p,q}$, $d^{\prime}_N = 
\sum_{j=1}^{k}(-1)^{\epsilon_j^{\prime}} \del^{\prime}_{M,j}$; $d^{\second}_N := 
\del^{\second}_M$. The multiple complexes $\mcomp{\tilde{M}}$ and 
$\mcomp{M}$ are isomorphic via the isomorphism $\phi$ given, on the 
component $R^{i_1, h_1} \tens \cdots \tens R^{i_k, h_k}$, by: 
$\phi^{i_1, h_1, \dots, i_k, h_k}= (-1)^{\sum_{l=1}^{k} i_l \epsilon^{\second}_l}$;
the isomorphism $\phi$ induces an isomorphism between $\bicomp{\tilde{N}}$ and
$\bicomp{N}$; moreover $\phi$ is compatible with the action of $\perm_k$ 
on $\bicomp{\tilde{N}}$ and $\bicomp{N}$.
The spectral sequence ${}^{\prime}E^{p,q}_1(\bicomp{N})$ is 
\emph{the same} as  ${}^{\prime}E^{p,q}_1(\bicomp{\tilde{N}})$ and 
the isomorphism $\phi$ induces the identity at the spectral sequence level.
The action $\widehat{\tau_{j,j+1}}$ of the 
transposition $\tau_{j,j+1}$ on $\bicomp{N}$ is now clear: it exchanges 
$$ \widehat{\tau_{j,j+1}}: R^{i_1, h_1} \tens \cdots \tens R^{i_j, h_j} \tens 
R^{i_{j+1}, h_{j+1}} \tens \cdots \tens R^{i_k, h_k} \rTo R^{i_1, h_1} \tens \cdots \tens R^{i_{j+1}, h_{j+1}} \tens 
R^{i_{j}, h_j} \tens \cdots \tens R^{i_k, h_k}$$and it is now defined by:
\begin{equation}\label{eq: invasssimpler}\widehat{\tau_{j,j+1}}(u_1 \tens \cdots \tens  u_j \tens u_{j+1} \tens \cdots \tens u_{k} ) =
 (-1)^{i_j i_{j+1} +h_j h_{j+1}} u_1 \tens \cdots \tens  u_{j+1} \tens u_{j} \tens \cdots \tens u_{k}\;.\end{equation}
The action $\widehat{\tau_{j,j+1}}$ induced on the spectral sequence $E^{p,q}_1$ 
exchanges 
\begin{equation}\label{eq: invatorsssimpler}\widehat{\tau_{j,j+1}} : \Tor_{-q}(C^{i_1}, \dots, C^{i_j}, 
C^{i_{j+1}}, \dots, C^{i_k}) \rTo \Tor_{-q}(C^{i_1}, \dots, C^{i_{j+1}}, 
C^{i_{j}}, \dots, C^{i_k}) \end{equation}as the map induced in $q$-cohomology 
by the map of complexes: 
$$ \widehat{\tau_{j,j+1}}: (M^{i_1, \dots, i_j,i_{j+1}, \dots, i_k, \bullet}, d^{\second}_M ) \rTo 
(M^{i_1, \dots, i_{j+1},i_{j}, \dots, i_k, \bullet}, d^{\second}_M )$$given by (\ref{eq: invasssimpler}). It is easy now to compare the actions $\widehat{\tau_{j,j+1}}$
and $\widetilde{\tau_{j,j+1}}$, since the complexes $(M^{i_1, \dots, i_j,i_{j+1}, \dots, i_k, \bullet}, d^{\second}_M )$ and $
(M^{i_1, \dots, i_{j+1},i_{j}, \dots, i_k, \bullet}, d^{\second}_M )$ above are 
tensor products of resolutions $R^{i_l,\bullet}$ of $C^{i_l}$ and are
exactly
the complexes used in appendix~\ref{app: perm} to compute $\widetilde{\tau_{j,j+1}}$: the sign $(-1)^{h_j h_{j+1}}$ in the definition (\ref{eq: invasssimpler})  coincides with the action of
$\widetilde{\tau_{j,j+1}}$, while the sign $(-1)^{i_j i_{j+1}}$ is due to the compatibility with the differential $d^{\prime}_N$, because we are considering the terms
$C^{i_l}$ as terms of the complex $\comp{C}$. As a consequence the action
$\widehat{\tau_{j,j+1}}$ on the spectral sequence $E^{p,q}_1$, exchanging 
multitors in (\ref{eq: invatorsssimpler}), differs from the action $\widetilde{\tau_{j,j+1}}$, considered in appendix~\ref{app: perm}, 
for the sign $(-1)^{i_j i_{j+1}}$:
$$\widehat{\tau_{j,j+1}} = (-1)^{i_j i_{j+1}} \widetilde{\tau_{j,j+1}} \;.$$
\end{remark}
\subsubsection*{Derived Schur functors}Let $V_{\nu}$ be the irreducible representation
of the group $\perm_k$ associated to the partition $\nu: \nu_1 \geq \nu_2 \geq 
\dots  \geq \nu_l$ of $k$. 
On 
a quasi-projective variety $M$, on which $\perm_k$ acts trivially, a locally free $\perm_k$-sheaf $E$
decomposes  as a direct sum of 
locally free subsheaves:
$ E \simeq \bigoplus_{\nu} V_{\nu} \tens \mc{H}om_{\perm_k}(V_{\nu} \tens \FS_M, E) $. Now, for every locally free sheaf $W$ on $M$, the $k$-fold tensor power 
$W^{\tens^k}$ is naturally equipped with the $\perm_k$-action that permutes the factors. We denote with 
$S^{\nu} W$ the Schur functor\footnote{The irreducible representation 
$V_{\nu}$ is isomorphic to its dual $V^*_{\nu}$, since any irreducible representation is determined by its character and since, by Frobenius formula (\cite{FultonHarrisRT}, 4.10), the character of $V_{\nu}$ is real}, associated 
to
the partition $\nu$~of~$k$:
$$ S^{\nu} W : = \mc{H}om_{\perm_k} (V_{\nu} \tens \FS_M, W^{\tens^k}) \simeq 
(W^{\tens ^k} \tens V_{\nu} )^{\perm_k} \; .$$ 
\sloppy Consider now a complex of coherent sheaves $\comp{C}$ on $M$ and 
$\comp{K}$ a locally free resolution of $\comp{C}$. Let us form the $k$-fold tensor power: 
 $\comp{P} =\comp{K} \tens \cdots \tens \comp{K}$. We can decompose the complex $\comp{P}$ as:
$\comp{P} = \bigoplus_{\nu} V_{\nu} \tens~\mc{H}om_{\perm_k}
(V_{\nu} \tens \FS_M, \comp{P})$. 
We denote with $S^{\nu}_L \comp{C} \in \B{D}^-(M)$ 
and we call it the \emph{derived Schur functor 
of the complex $\comp{C}$, associated to the partition $\nu$}, the complex $$ 
S^{\nu}_L \comp{C} : = \mc{H}om_{\perm_k}
(V_{\nu} \tens \FS_M, \comp{P}) \simeq (\comp{P} \tens V_{\nu})^{\perm_n} 
\;.$$ 
Its isomorphism class does not depend on the 
choice of the resolution $\comp{K}$.
\begin{remark}
Suppose now that $X$ is a quasi-projective variety with the action of a finite
group $G$ and that $\comp{C}$ is a complex of $G$-equivariant sheaves on
$X$, $\comp{C} \in \B{D}^{-}_G(X)$. Then the $\perm_k$-action on $\comp{C} \tens ^L \cdots \tens
  ^L \comp{C}$ commutes with the diagonal action of $G$ on $
 \comp{C} \tens ^L \cdots \tens
  ^L \comp{C}$, defined as the diagonal action on the complex
  $\comp{K} \tens \cdots \tens \comp{K}$, where $\comp{K}$ is a $G$-equivariant locally free
  resolution of $\comp{C}$. We then have a well defined $G \times \perm_k$
  action on the derived tensor power $\comp{C} \tens ^L \cdots \tens
  ^L \comp{C}$. 
\end{remark}
Since tensorizing with $V_{\lambda}$ and taking invariants 
$[-]^{\perm_k}$ are exact functors, we deduce from  \ref{cor:imageidentif} and 
\ref{crl: inva} the following corollary. Let $\lambda$ be a partition of $k$.
The spectral sequence $(E^{p,q}_1 \tens V_{\lambda})^{\perm_k}$ is the 
hyperderived spectral sequence associated to 
the derived Schur functor $S^{\lambda}_L \comp{\mc{C}}_L$  
of the complex $\comp{\mc{C}}_L$ and abutting to the cohomology $H^{p+q}(S^{\lambda}_L \comp{\mc{C}}_L)$. Then:
\begin{crl}\label{qoro}
The term 
$(E^{0,0}_{\infty} \tens V_{\lambda})^{\perm_k}$ of the spectral sequence 
$(E^{p,q}_1 \tens V_{\lambda})^{\perm_k}$ 
can be identified with 
the image $\bkrh(S^{\lambda}L^{[n]})$ of the Schur functor 
$S^{\lambda}L^{[n]}$ for the BKR transform. 
As a consequence, the  direct 
image $\mu_*(S^{\lambda}L^{[n]})$ 
of the Schur functor $S^{\lambda}L^{[n]}$ can be identified with the term 
$ \mc{E}^{0,0}_{\infty}(\lambda) 
$ of the 
spectral sequence of invariants: 
$$ \mc{E}^{p,q}_{1}(\lambda):=(E^{p,q}_1 \tens V_{\lambda})^{\perm_n \times \perm_k} $$of 
$E^{p,q}_1 \tens V_{\lambda}$ for the group $\perm_n \times \perm_k$.
\end{crl}
\subsection{The exterior power}Let $\lambda_{\epsilon} = 1 + \dots + 1$ be the partition 
of $k$ associated to the alternant representation $\epsilon_k$ of $\perm_k$.
In this subsection we will analyse the spectral sequence \begin{equation*}
\label{spe} \mc{E}^{p,q}_1(\lambda_{\epsilon}): = (E^{p,q}_1 \tens \epsilon_k)^{\perm_n \times \perm_k} \; ,\end{equation*}and, in particular, its restriction  $j^* \mc{E}^{p,q}_1(\lambda_{\epsilon})$  
to the open set $S^n_{**}X$. By corollary \ref{qoro} the term $\mc{E}^{0,0}_{\infty}(\lambda_{\epsilon})$ can be identified with the direct image $\mu_*(\Lambda^kL^{[n]})$ of the exterior power $\Lambda^kL^{[n]}$ for the Hilbert-Chow morphism. Remembering the notation introduced in remark \ref{remarchetto}, we  indicate with $\mc{E}^{p,q}_1(\lambda_{\epsilon})_0$ the 
component: $$\mc{E}^{p,q}_1(\lambda_{\epsilon})_0 := [(E^{p,q}_1)_0 \tens \epsilon_k]^{\perm_n \times \perm_k} \;.$$In corollary \ref{crl: degeneration} we will prove that
\emph{the spectral sequence $j^*\mc{E}^{p,q}_1(\lambda_{\epsilon})$ degenerates at level $\mc{E}_2$}. 
\begin{remark}\label{rem:rest}To prove the previous statement it will be sufficient to prove that for $q<0$ the complexes $j^*\mc{E}^{\bullet,q}_1(\lambda_{\epsilon}) = j^* \mc{E}^{\bullet,q}_1(\lambda_{\epsilon})_0$, 
are exact in degree $l \neq -q/2+1$, by remark \ref{remarchetto}.   
Since $j^*$ is exact, this will be implied by the exactness of the complexes $\mc{E}^{\bullet,q}_1(\lambda_{\epsilon})_0$ for $q<0$ in degree $l \neq -q/2+1$. To prove this, by lemma 
\ref{lemmino} it suffices to \emph{prove that, for $q<0$ and for every affine open subset $U$ of $X$, the complexes 
\begin{equation} \label{eq: epq}
\mf{E}^{\bullet,q}:= \left[ H^0(U^n, (E^{\bullet,q}_1)_0 \tens \epsilon_k) \right]^{\perm_n \times \perm_k}	
\end{equation}are exact in degree $l \neq -q/2+1$.}
\end{remark}The rest of  this section will be devoted to prove this last statement. 
\subsubsection{The comprehensive $G \times H$-action}
From now on we will make heavy use of appendix C.
\begin{notat}In the following we will indicate with $G$ the symmetric group $G:= \perm_n$, 
   and with $H$ the symmetric group $H:= \perm_k$. 
 Finally, if $A \subseteq \{1, \dots,n\}$ (resp. $A \subseteq \{1, \dots,k\}$) we will indicate with $G(A)$ (resp. $H(A)$) the subgroup $\perm(A)$ of $G$ (resp. of $H$) consisting of permutations of $A$. 
 \end{notat}
\begin{notat}Let $a \in \mc{I}^l$, $l \geq 0$. For brevity's sake, indicate with $\mc{F}^{l,q}(a)$ the sheaf $\mc{F}^{l,q}(a) :=\Tor_{-q}(L_{a(1)}, \dots, L_{a(k)})$. Remembering notation \ref{not: puremultitor}, lemma \ref{torinv} and remark \ref{C8}, it can be \emph{identified} with
\begin{equation}\label{eq: identiftor}\mc{F}^{l,q}(a)  \simeq \Tor^l_{-q}(L_{I_0(a)}) \tens \Tens_{j \in J(a)} 
L_j^{\tens ^{\lambda_j(a)}} \simeq 
\Lambda^{-q}(N^*_{I_0(a)} \tens \rho_l) \tens L_{I_0(a)}^{\tens^l} \tens \Tens_{j \in J(a)} L_{j}^{\lambda_j(a)}
\;.\end{equation}Set now $F^{l,q}(a):=H^0(U^n,\mc{F}^{l,q}(a))$, where $U$ is an affine open set in $X^n$.
\end{notat}
\begin{remark}\label{totalaction}
We summarize here the comprehensive $G \times H$-action on
 the $G \times H$-equivariant term: 
\begin{equation} \label{eq: tor4}(E^{p,q}_1)_0 = \bigoplus_{a \in \mc{I}^p}  \mc{F}^{p,q}(a)  \;.\end{equation}The group $G \times H$ acts on the left on the set of indexes $\mc{I}^p$  by setting $(\sigma, \tau)\cdot a := \sigma a \tau^{-1} $.

\emph{The comprehensive $G$-action.} An element $\sigma \in G$ acts on the equivariant term $(E^{p,q}_1)_0$ above with the action induced by the $G$-action on the complex $\comp{\mc{C}}_L$, explicited in remark \ref{farfalla}; the element $\sigma$ carries the sheaf $L_{a(i)}$ to the sheaf $L_{\sigma a(i)}$, introducing the sign $\epsilon_{\sigma, a(i)}$.
The subgroup $G(I_0(a))$ hence
acts with the representation $\epsilon_2$ on the fibers of the sheaf $L_{I_0(a)}$
over the diagonal $\Delta_{I_0(a)}$; moreover it
acts geometrically on the fibers of the conormal bundle $N^*_{I_0(a)}$ with the 
representation $\epsilon_{2} \oplus \epsilon_{2} \simeq \mbb{C}^2 \tens \epsilon_2$. 

\emph{The comprehensive $H$-action.} The group 
$H$ acts on the term $(E^{p,q}_1)_0$ with the  action induced by the action 
on the hyperderived spectral sequence $E^{p,q}_1$, explained in remark \ref{rmk:ssaction}. Hence a consecutive transposition $\tau_{j, j+1} \in H$ operates 
on $(E^{p,q}_1)_0$
via the automorphism $\widehat{\tau_{j,j+1}}$, whose restriction to $\mc{F}^{p,q}(a)$:
$ \widehat{\tau_{j,j+1}}: \mc{F}^{p,q}(a) \rTo \mc{F}^{p,q}(a \tau^{-1}_{j,j+1})$, for any $a \in \mc{I}^l$, 
is given by $$ \widehat{\tau_{j,j+1}} = (-1)^{(|a(j)|-1)(|a(j+1)|-1)} \widetilde{\tau_{j,j+1}} \;,$$where $\widetilde{\tau_{j,j+1}}$ is the permutation of factors considered in appendix \ref{app: perm}; the sign $(-1)^{(|a(j)|-1)(|a(j+1)|-1)}$ is due to the fact that 
 $L_{a(i)}$ is a subsheaf of the term $\mc{C}^{|a(i)|-1}_L$ of the complex $\comp{\mc{C}}_L$.

Consider now a general permutation $\tau$ of $H$, and let, as in notations of 
lemma \ref{permsign}, $\sigma_0(\tau)$, $\sigma_j(\tau)$, for $j \in J(a)$, be the unique increasing 
bijections $\sigma_0(\tau): \tau(S_0(a)) \rTo S_0(a)$, $\sigma_j(\tau)
: \tau(a^{-1}(\{ j \})) \rTo a^{-1}(\{ j \})$, respectively;  set $\beta_0(\tau):= \sigma_0(\tau) \circ \trest_{S_0(a)}$, 
 $\beta_j(\tau) := \sigma_j(\tau) 
\circ \tau \trest_{a^{-1}(\{ j \})}$, seen as permutations 
in $\perm_{p}$, $\perm_{\lambda_{j}(a)}$, respectively; denote moreover
 with
$\alpha(\tau)$  
the automorphism induced by $\beta_0(\tau)$ on $\Lambda^{-q}(N^*_{I_0(a)} \tens \rho_{p})$. Then by what we just explained and by lemma \ref{permsign}, \emph{once identified $\mc{F}^{p,q}(a)$ and $\mc{F}^{p,q}(a \tau^{-1})$ via 
(\ref{eq: identiftor})},
 the comprehensive action of an element $\tau \in H$
$$ \widehat{\tau} : \mc{F}^{p,q}(a) \rTo 
\mc{F}^{p,q}(a\tau^{-1})$$ can be identified with 
the automorphism $$ \sgn(\beta_0(\tau)) \;\alpha(\tau)\tens \beta_0(\tau) \tens \Tens_{j \in J(a)} \beta_j(\tau) \in \Aut \Big( \Lambda^{-q}(N^*_{I_0(a)} \tens \rho_{p}) \tens 
L_{I_0(a)}^{\tens^p}\tens \Tens_{j \in J(a)} L_j^{\tens \lambda_j(a)} \Big) \;,$$
where $\beta_j(\tau)$, $j \in J(a) \cup \{0 \}$ act by permutation of factors.
\end{remark}
\subsubsection{Homogeneous components and stabilizers}
We begin the study of the terms $\mf{E}^{p,q}$ 
decomposing the $G \times H$-equivariant term $H^0(U^n, (E^{p,q}_1)_0)$ in homogeneous components in order to use Danila's lemma to compute invariants.
Let  $\mc{O}$ be 
 a $G \times H$-orbit in $\mc{I}^p$: we define the \emph{homogeneous component} 
$W^{p,q}_{\mc{O}}$ as
$ W^{p,q}_{\mc{O}} := \bigoplus_{a \in \mc{O}} 
F^{p,q}(a)
$. Therefore the term $\mf{E}^{p,q}$ can be rewritten as: 
$$ \mf{E}^{p,q} = \bigoplus_{\mc{O} \in \mc{I}^p/G \times H} \left(W^{p,q}_{\mc{O}} \tens \epsilon_{k} \right)^{G \times H} \;.$$
\begin{notat}Let $K \subseteq G \times H$ a subgroup of $G \times H$. 
We will indicate with $\Ind^H_{K} \epsilon_k $ the representation of 
$K$ induced by the alternant representation $\epsilon_k$ of $H$ 
via the natural morphism 
$K \rTo H$ induced by the second projection. 
\end{notat}In these notations
Danila's lemma \ref{Danila} becomes: 
\begin{lemma}\label{lmm: danilaversion}Let $\mc{O} \in \mc{I}^p/G \times H$ be a $G \times H$-orbit in $\mc{I}^p$
and let $a \in \mc{O}$. Then:
$$ (W^{p,q}_{\mc{O}} \tens \epsilon_k)^{G \times H} \rTo^{\simeq}
\Big ( F^{p,q}(a)
\tens \Ind^H_{\Stab_{G \times H}(a)} \epsilon_k \Big)^{\Stab_{G \times H}(a)}$$
\end{lemma}
The classification of all the $G \times H$-orbits and the determination of stabilizers in all generality turns out to be 
an unnecessary technical computation, due to
the following lemma; as a consequence, we will consider only some of the orbits, which will be called relevant.
\begin{lemma}\label{lmm: rele}Let $\mc{O}$ the $G \times H$-orbit of an element $a \in \mc{I}^p$ such that $a \trest_{\{1, \dots,k\} \setminus S_0(a)}$ is not injective. Then 
$( W^{p,q}_{\mc{O}} \tens \epsilon_k)^{G \times H} =0$.
\end{lemma}
\begin{proof}The hypothesis on $a$ is equivalent to the existence of a $j_0 \in J(a)$ such that $|a^{-1}(\{ j_0 \})| = \lambda_{j_0}(a) \geq 2$. The group $H(a^{-1}(\{ j_0 \}))$ is then a nontrivial subgroup 
of $\Stab_{G \times H}(a)$. Applying  lemma \ref{lmm: danilaversion}, we get:
$$ ( W^{p,q}_{\mc{O}} \tens \epsilon_k)^{G \times H} \subseteq \Big(
F^{p,q}(a)
\tens \Ind^H_{H(a^{-1}(\{ j_0 \}))} \epsilon_k \Big)^{H(a^{-1}(\{ j_0 \}))} \;.$$Now the group $H(a^{-1}(\{ j_0 \}))$ acts nontrivially only on the factor 
$L_{j_0}^{\tens {\lambda_{j_0}(a)}}$ of $\mc{F}^{p,q}(a)$,
by remark \ref{totalaction}.  
Since the representation $\Ind^H_{H(a^{-1}(\{j_0\}))} \epsilon_k =\epsilon_{\lambda_{j_0}(a)}$ and
since $(L_{j_0}^{\tens \lambda_{j_0}(a) } \tens \epsilon_{\lambda_{j_0}(a)} )^{H(a^{-1}(\{ j_0 \}))} = 
\Lambda^{\lambda_{j_0}(a)} L_{j_0} =0 $, the right hand side above vanishes.
 \end{proof}
\begin{remark}Let $\mc{I}^p_0 := \{a \in \mc{I}^p \; | \; a \trest_{\{1, \dots,k\} \setminus S_0(a)} \; \textrm{is injective} \}$. The set $\mc{I}^p_0$ is $G \times H$-invariant.
We will call \emph{relevant} 
  the orbits in 
$\mc{I}^p_0 / G \times H$; the other ones will be called \emph{irrelevant}.
\end{remark}
\begin{remark}\label{rmk: rel}By the preceding lemma we have only to classify relevant orbits. 
The $G \times H$-invariant map $ t : \mc{I}^p_0  {\rTo}  \{ 0,1,2 \} $, defined by $
a  \rMapsto  |I_0(a) \cap J(a)| $, induces easily an injection $t : \mc{I}^p_0 / G \times H {\rInto} \{ 0,1,2 \}$ at the quotient level; as a consequence relevant orbits $\mc{O}$ are 
classified 
by the value $t(\mc{O}) \in \{ 0,1,2\}$. For a relevant orbit $\mc{O}$, correspondent to the value $t(\mc{O})=t$, the \emph{relevant homogeneous component} $W^{p,q}_{\mc{O}}$ is: 
\begin{equation*}\label{eqn: re} W^{p,q}_{\mc{O}}= \bigoplus_{\substack{ a \in \mc{I}^p_0 \\ t(a)=t}}
F^{p,q}(a)
  \;. \end{equation*}
\end{remark}
The next proposition determines the stabilizer of an element in a relevant orbit. 
In what follows, for brevity's sake, where there will be no risk of confusion, 
we will write $S_0$, $I_0$, $J$ instead of $S_0(a)$, $I_0(a)$,~$J(a)$.
\begin{notat}If $a \in \mc{I}^p_0$ and $R \subseteq J$, denote with 
$\Delta(R)$ the diagonal subgroup of $G(R) \times H(a^{-1}(R)) $ given by $\Delta(R) := 
\{(\sigma, a^*\sigma) \; | \; \sigma \in G(R) \}$, where $a^* \sigma$ is the permutation of $a^{-1}(R)$ 
defined as: $a^* \sigma := a^{-1} \sigma a$. \end{notat}
\begin{pps}\label{pps: stab}
Let $a \in \mc{I}^p_{0}$. Then\footnote{Here and in the following we set $G(\emptyset)=H(\emptyset)=\{1 \}$.} 
$$ \Stab_{G \times H}(a) = G(I_0 \setminus J) \times  H(S_0) \times \Delta(I_0 \cap J) \times \Delta (J \setminus I_0) \times G(\overline{I_0 \cup J})\;.$$
\end{pps}
\begin{proof} An element $(\sigma, \tau)$ is in $\Stab_{G \times H}(a)$ if 
and only if 
 $\sigma a = a \tau$. This implies that the sets $I_0, J$ and, consequently, 
 $I_0 \cap J$, $I_0 \setminus J$, $J \setminus I_0$, $\overline{I_0 \cup J}$ are globally fixed by $\sigma$; hence $$\sigma = \sigma_1 \sigma_2 \sigma_3 \sigma_4 \in G(I_0 \setminus J) \times G(I_0 \cap J) \times 
 G( J \setminus I_0) \times G(\overline{I_0 \cup J})$$
On the other hand we can always write $\tau = \tau_1 \tau ^{\prime} \in H(S_0) \times H(\overline{S_0})$, since $S_0$ and $\overline{S_0}$ have to be preserved by 
$\tau$.
Since $\sigma_1$, $\sigma_4$, $\tau_1$ are already in $\Stab_{G \times H}(a)$, then we have: 
 $ \sigma^{\prime} a = a \tau^{\prime} $, where $\sigma^{\prime}=\sigma_2 \sigma_3 \in G(I_0 \cap J) \times 
 G( J \setminus I_0)$. Suppose that $\sigma^{\prime} (i) =j$, with $i,j \in J$ such that $\{ i,j \} \subseteq I_0 \cap J$ or $\{ i,j \} \subset J \setminus I_0$. 
Then one has to have $\tau^{\prime}(a ^{-1}( \{ i \})) = a^{-1}(\{ j \})$. Hence $\tau^{\prime} = a^* \sigma_2 a^* \sigma_3$. 
This proves the lemma. 
\end{proof}
\subsubsection{Invariants}
\begin{notat}\label{not: f}
Denote with $\mc{F}^{l,q}_i(a)$ the sheaves on $X^{I_0}$, $X^{J \setminus I_0}$, 
$X^{\overline{I_0 \cup J}}$, respectively, 
\begin{gather*}
 \mc{F}^{l,q}_1(a) := \Tor^l_{-q}(L_{\Delta}) \tens \Tens_{j \in I_0 \cap J} L_j 
\simeq \Lambda^{-q}(N^*_{\Delta} \tens \rho_l) \tens L_{\Delta}^{\tens^{l}} \tens \Tens_{j \in I_0 \cap J} L_j 
\\
\mc{F}^{l,q}_2(a):= \boxtimes_{j \in J \setminus I_0} 
L_{j} \\
\mc{F}_3^{l,q}(a):= \FS_{X^{\overline{I_0 \cup J}}} \;,
\end{gather*}
where $\Delta$ is the diagonal in $X^{I_0} \simeq X^2$.
Therefore we can write: 
$ \mc{F}^{l,q}(a)  \simeq \mc{F}^{l,q}_1(a) 
\boxtimes 
\mc{F}^{l,q}_2(a) \boxtimes \mc{F}^{l,q}_3(a)$. Moreover, denote with $F^{l,q}_i(a)$
the corresponding spaces of sections of $\mc{F}^{l,q}_i(a)$ over 
affine open subsets of the form $U^s$, that is, $F_i^{l,q}(a) := 
H^0(U^s, \mc{F}^{l,q}_i(a))$, where $s=I_0, J \setminus I_0, \overline{J \cup I_0}$, if $i=1,2,3$, respectively.  
Therefore, by K\"unneth formula: 
$$ F^{l,q}(a) 
= F^{l,q}_1(a) \tens 
F^{l,q}_2(a) \tens F^{l,q}_3(a) \;.$$If $a \in \mc{I}^l_0$, $l=0$, 
the definitions of 
$\mc{F}^{l,q}_i(a)$ 
and $F^{l,q}_i(a)$ have still a sense if $i=2,3$. In this case we have: $F^{0,q}(a)
= 
F^{0,q}_2(a) \tens 
F^{0,q}_3(a)$. It is clear that $F^{0,q}(a)=0$ if $q\neq 0$. Finally we indicate with
$\mf{I}^{p,q}(a)$ the invariants: 
$$ \mf{I}^{p,q}(a) := \left( F^{p,q}(a) \tens \Ind^H_{\Stab_{G \times H}(a)} \epsilon_k \right)^{\Stab_{G \times H}(a)} \;.$$In these notations lemma \ref{lmm: danilaversion} can be rewritten as follows: for a relevant orbit $\mc{O}$
and for any fixed $a \in \mc{O}$:
$ \left( W_{\mc{O}}^{p,q} \tens \epsilon_k \right)^{G \times H} \simeq 
\mf{I}^{p,q}(a)$.
\end{notat}
\begin{notat}Denote with $G_i(a)$ the following subgroups of 
$\Stab_{G \times H}(a)$ 
isolated in proposition \ref{pps: stab}: 
$G_1(a) := G(I_0 \setminus J) \times H(S_0) \times \Delta(I_0 \cap J)$, $
G_2(a) := \Delta(J \setminus I_0)$, 
$G_3(a) := G(\overline{I_0 \cup J})$.
\end{notat}
\begin{remark}\label{rimarcotensor}In the above notations, 
since the group $G_i(a)$ acts nontrivially only on the factor 
$F_i^{l,q}(a)$, we can write: 
\begin{equation}\label{Inva} \mf{I}^{l,q}(a) = \Tens_{i=1}^3 
\left( F^{l,q}_i(a) \tens \Ind_{G_i(a)}^H \epsilon_k \right)^{G_i(a)} 
\end{equation}
If $l=0$ we have $\mc{I}^{0,0}(a) = 
(F_2^{0,0}(a) \tens \Ind_{G_2(a)}^H \epsilon_k )^{G_2(a)} \tens (F_3^{0,0}(a) \tens \Ind_{G_3(a)}^H \epsilon_k )^{G_3(a)}$. 
\end{remark}
\begin{remark}\label{rmk: g1action}Let $a \in \mc{I}^p_0$, $p \geq 0$. 
It is clear that the $G_i(a)$-action on 
the factors $F_i^{l,q}(a)$ is induced by a natural $G_i(a)$-action  on the sheaves
$\mc{F}_i^{l,q}(a)$. The $G_2(a)$- and $G_3(a)$-actions on $\mc{F}_2^{p,q}(a)$ and $\mc{F}^{p,q}_3(a)$ are clear. We 
explicit here the $G_1(a)$-action on 
$\mc{F}^{p,q}_1(a)$ in terms of what explained in remark \ref{totalaction}. \begin{itemize}
\item The group $G(I_0 \setminus J)$, identified with $\perm_{2-t}$, acts 
on $\mc{F}^{p,q}_1(a)$, fiberwise on $\Delta \subseteq X^{I_0}$, with the 
representation 
$$ \Lambda^{-q}(\id_{\mbb{C}^2} \tens \epsilon_{2-t} \tens \id_{\mbb{C}^{p-1}}) \tens \epsilon_{2-t}^{\tens ^p}
\simeq \id_{\Lambda^{-q}( \mbb{C}^{2(p-1)} )} \tens \epsilon_{2-t}^{\tens^{ p-q}} \;.$$ \item 
The action of the group $\Delta(I_0 \cap J)$, identified with $\perm_t$, 
is trivial on the factor 
$\Tens_{j \in I_0 \cap J} L_j \trest_{\Delta} \simeq L_{\Delta}^{\tens ^t}$; 
on the other hand, the factors in $G(I_0 \cap J)$ of the elements in $\Delta(I_0 \cap J)$ act nontrivially on $\Tor_{-q}^l(L_{\Delta})$: hence $\Delta(I_0 \cap J)$ acts on $\mc{F}^{p,q}_1(a)$ with the representation 
$$ \Lambda^{-q}(\id_{\mbb{C}^2} \tens \epsilon_{t} \tens \id_{\mbb{C}^{p-1}}) \tens \epsilon_{t}^{\tens ^p}
\simeq \id_{\Lambda^{-q}(\mbb{C}^{2(p-1)}) }\tens \epsilon_{t}^{\tens^{ p-q}} \;.$$\item 
The  $H(S_0)$-sheaf $\mc{F}^{p,q}_1(a)$ is isomorphic, as $\perm_p$-sheaf,  to
the sheaf
 $\Lambda^{-q}(N^*_{\Delta} \tens \rho_p) \tens L_{\Delta}^{\tens ^{p+t}} \tens \epsilon_p$.
\end{itemize}Hence, considering the comprehensive $G_1(a)$-action, 
 the sheaf $\mc{F}^{p,q}_1(a)$ is isomorphic, as $\perm_{2-t} \times \perm_p \times \perm_t$-sheaf, to the sheaf: 
$$ \mc{F}^{p,q}_1(a) \simeq \Lambda^{-q}(N^*_{\Delta} \tens \rho_p) \tens L_{\Delta}^{\tens ^{p+t}} \tens \epsilon_p \tens \epsilon_{2-t}^{\tens ^{p-q}} \tens \epsilon_{t}^{\tens ^{p-q}} \;.$$
\end{remark}
\subsubsection*{The terms $\mf{E}^{p,q}$.}
We are ready to compute the invariants 
$ \mf{I}^{p,q}(a) $.
For $q \leq 0$ even, $l>0$, denote with $\mf{F}^{l,q}$ the module: 
$$ \mf{F}^{l,q}:=  H^0(U, L^{\tens^l} \tens 
K_X^{^{\tens^{-q/2}}}) \tens \Lambda^{k-l}H^0(U,L) \tens
S^{n-k+l-2}H^0(U,\FS_X) \;.$$If $l=0$, set $\mf{F}^{0,0} = \Lambda^kH^0(U,L) \tens S^{n-k}H^0(U,\FS_X)$. Remark that $\mf{F}^{l,q}=0$ if $l>k$. 
\begin{pps}\label{lapropo}Let $a \in \mc{I}^p_0$, $t(a)=t$. The term 
$\mf{I}^{p,q}(a)$ is non zero  only if $p=q=0$ or
$$1 \leq p \leq k, \; 2-2p \leq q \leq 0, \;
 \mbox{$q$ is even and} \left \{ 
\begin{array}{l} \mbox{$p$ even, $0 \leq t \leq 1$} \\ \mbox{$p$ odd, $1 \leq t \leq 2$}  \end{array} \right. $$In this cases we have $\mf{I}^{p,q}(a) = \mf{F}^{p+t,q}$.
\end{pps}
\begin{proof}
By remark \ref{rimarcotensor}, it is sufficient to compute separately 
the factors in (\ref{Inva}). 

\emph{Step 1.}
The invariants $( F_3^{p,q}(a) \tens \Ind_{G_3(a)}^H\epsilon_k )^{G_3(a)}$ 
are very easy to compute since $G_3(a) = 
G(\overline{I_0 \cup J})$ just permutes geometrically the factors 
and $\Ind_{G_3(a)}^H\epsilon_k=1$; then:
$$\left( F_3^{p,q}(a) \tens \Ind_{G_3(a)}^H\epsilon_k \right)^{G_3(a)} 
\simeq \left[ H^0(U, \FS_X)^{\tens ^{|\overline{I_0 \cup J}|}} \right]^{\perm_{|\overline{I_0\cup J}|}} \simeq 
S^{n-k+p+t-2}H^0(U, \FS_X)
\;.$$

\emph{Step 2.} We pass now to the computation of the invariants $( F_2^{p,q}(a) \tens \Ind_{G_2(a)}^H\epsilon_k )^{G_2(a)}$. 
By K\"unneth formula 
$ F_2^{p,q}(a) := H^0(U^{J \setminus I_0}, \boxtimes_{j \in  J \setminus I_0} L_{j}) \simeq 
H^0(U,L)^{\tens ^{|J \setminus I_0|}}$. The group 
$\Delta(J \setminus I_0)$ acts on $F_2^{p,q}(a)$ just exchanging the factors; 
consequently, since $\Ind^H_{\Delta(J \setminus I_0)} \epsilon_k \simeq \epsilon_{|J \setminus I_0|}$, we have: 
\begin{equation*}
\Big (F_2^{p,q} \tens \Ind_{G_2(a)}^H\epsilon_k \Big )^{G_2(a)} 
 \simeq  (H^0(U,L)^{\tens ^{|J \setminus I_0|}} \tens \epsilon_{|J \setminus I_0|})^{\perm_{|J \setminus I_0|} }  \simeq \Lambda^{k-p-t}H^0(U,L) 
\end{equation*}The computation done so far can be remade in exactly the same way for $p=0$, considering $I_0= \emptyset$: it gives: $\mf{I}^{0,0}(a) = \mf{F}^{0,0}$.

\emph{Step 3.} We compute now the invariants 
$( F^{p,q}_1(a) \tens \Ind^{G_1(a)}_{H} \epsilon_k )^{G_1(a)}$.  
Remark that $F^{p,q}_1(a)$ is the space of sections over $U^{I_0}$ of the sheaf 
$\mc{F}^{p,q}_1(a)$, supported on the diagonal 
$\Delta$, on which $G_1(a)$ acts trivially. Hence: 
$$ \Big(F^{p,q}_1(a) \tens \Ind_{G_1(a)}^H \epsilon_k \Big)^{G_1(a)} \simeq 
H^0\Big(\Delta, \Big( \mc{F}^{p,q}_1(a) \tens \Ind_{G_1(a)}^{H} \epsilon_k 
\Big)^{G_1(a)} \Big) \;.$$
The $G_1(a)$-invariants of the sheaf $\mc{F}^{p,q}_1(a) \tens \Ind_{G_1(a)}^{H} \epsilon_k$ are now easy to compute, due to remark \ref{rmk: g1action}. Keeping into account that $\Ind^{H}_{G_1(a)} \epsilon_k \simeq \Ind_{H(S_0)}^H \epsilon_k \tens \Ind_{\Delta(I_0 \cap J)}^H \epsilon_k \simeq \epsilon_p \tens \epsilon_t$, 
we have:
$$ \Big( \mc{F}^{p,q}_1(a) \tens \Ind_{G_1(a)}^{H} \epsilon_k \Big)^{G_1(a)} \simeq  
\Big( \Lambda^{-q}(N^*_{\Delta} \tens \rho_p )\Big)^{\perm_p} \tens 
\left(\epsilon_{2-t}^{\tens ^{p-q}} \right)^{\perm_{2-t}} \tens \left(\epsilon_t^{\tens ^{p-q+1}} \right)^{\perm_t}  \tens L_{\Delta}^{\tens^{p+t}}  \;.$$
By lemma \ref{cinquediciotto} the representation 
$\Lambda^{-q}(N^*_{\Delta} \tens \rho_p )$ has non zero $\perm_p$-invariants
if and only if $q$ is even; in this case, by corollary \ref{crl: invator}: 
$$ \Big( \mc{F}^{p,q}_1(a) \tens \Ind_{G_1(a)}^{H} \epsilon_k \Big)^{G_1(a)} \simeq  
\Big( \Lambda^{2}N^*_{\Delta} \Big)^{-q/2} \tens 
\left(\epsilon_{2-t}^{\tens ^{p-q}} \right)^{\perm_{2-t}} \tens \left(\epsilon_t^{\tens ^{p-q+1}} \right)^{\perm_t}  \tens L_{\Delta}^{\tens^{p+t}} $$
which is nonzero if and only if $0 \leq t \leq 1$, $p$ even, or 
$1 \leq t \leq 2$, $p$ odd. In these cases the invariants are: 
$$ \Big( F^{p,q}_1(a) \tens \Ind_{G_1(a)}^H \epsilon_k \Big)^{G_1(a)} \simeq 
H^0(U, K_X^{\tens^{-q/2}} \tens L^{\tens^{p+t}}) \;.$$This proves the proposition.
\end{proof}
\begin{remark}\label{rem:nonzeroinv}
Note that, in the case of nonzero invariants $0 \neq \mf{I}^{p,q}(a)$, both the groups 
$G(I_0 \setminus J)$ and $\Delta(I_0 \cap J)$ act trivially on $F_1^{p,q}(a) 
\tens \Ind_{G_1(a)}^H \epsilon_k$.
\end{remark}
\begin{crl} Let $0 \leq p \leq k$, $2(1-p) \leq q \leq 0$, $q$ even. If $p$ is even the term $\mf{E}^{p, q}$ is isomorphic to:
$$ \mf{E}^{p, q}  \simeq   \left \{ \begin{array}{ll} \mf{F}^{p,q} \oplus \mf{F}^{p+1,q}
  & \mathrm{if}  \; \; 0<p \leq k    \\ \mf{F}^{0,0} & \mathrm{if}  \;  p=q=0 \end{array} \right. 
\;.$$If $p$ is odd,  we have: 
$ \mf{E}^{p, q} \simeq     \mf{F}^{p+1,q} \oplus \mf{F}^{p+2,q}$. 
\end{crl}
\begin{remark}Remark that $\mf{E}^{k, q}= \mf{F}^{k,q}$ if $k$ even, while, if $k$ odd,
$\mf{E}^{k-1, q}= \mf{F}^{k,q}$, $\mf{E}^{k, q}= 0$.
\end{remark}
\subsubsection*{The differentials $\mf{E}^{p,q} \longrightarrow \mf{E}^{p+1,q}$}
\begin{remark}\label{rmk: diffinvprimo}The map $\mf{E}^{p,q} \rTo \mf{E}^{p+1,q}$ is the map of 
$G \times H$-invariants of the map of sections: 
\begin{equation}\label{eq: sec}
f: \bigoplus_{a \in \mc{I}^p_0} 
F^{p,q}(a)  \tens \epsilon_k\rTo^{\oplus_{a,b} f_{a,b}} \bigoplus_{b \in \mc{I}^{p+1}_0} F^{p+1,q}(b) \tens \epsilon_k \end{equation} and is induced by the differential $(E^{p,q}_1)_0 \rTo (E^{p+1,q}_1)_0$.
By definition of $(E^{p,q}_1)_0$ and $(E^{p+1,q}_1)_0$ any nontrivial morphism 
$\mc{F}^{p,q}(a) \rTo \mc{F}^{p+1}(b)$, whose map of sections $f_{a,b}$ appears in 
(\ref{eq: sec}), has to be induced by a restriction $L_{j_0} \rTo L_{I_0(a)}$ for $j_0 \in I_0(a) \cap J(a)$, twisted by a sign coming from the definition of the differential $\partial^p_{\comp{\mc{C}}_L \tens \cdots \tens \comp{\mc{C}}_L}$, which we will explicit later. Hence, if such a map $\mc{F}^{p,q}(a) \rTo \mc{F}^{p+1}(b)$ is nontrivial, this forces $t(a) \geq 1$, $t(b) = t(a)-1$. Moreover, 
$b$ has to be obtained by $a$ in the following way: $b(i) = a(i)$, if $i \in \{1, \dots, k\} \setminus a^{-1}(\{ j_0 \})$, $b(a^{-1}(\{ j_0 \}))=I_0(a)$. Hence $I_0(b) = I_0(a)$; 
$J(b) = J(a) \setminus \{ j_0 \}$. Consequently: 
\begin{gather*}
\Stab_{G \times H}(a) := H(S_0(a)) \times \Delta(I_0(a) \cap J(a)) \times G_2(a) \times G_3(a) \\
\Stab_{G \times H}(b) := G( (I_0(a) \setminus J(a)) \cup \{ j_0 \} ) 
\times H(S_0(a) \cup a^{-1}(\{ j_0 \})) \times G_2(a) \times G_3(a)
\end{gather*}Therefore, if $p \geq 0$, $t(a)=0$, there are no 
nontrivial maps $f_{a,b}$ from $F^{p,q}(a)$ in (\ref{eq: sec}).
\end{remark}
\begin{remark}
\label{rmk: diffinvsecondo}
By the preceding remark and by proposition \ref{lapropo},
if $p>0$ is even and $t(a) = 1$ all possible maps 
$F^{p,q}(a) \rTo F^{p+1,q}(b)$
induce the zero map between invariants. 
Consequently, if $p>0$ is even there are no nontrivial maps $\mf{E}^{p,q} \rTo 
\mf{E}^{p+1,q}$. On the other hand, if $p$ is odd the only possible morphism: 
$ \mf{E}^{p,q} = \mf{F}^{p+1,q} \oplus \mf{F}^{p+2,q} \rTo \mf{F}^{p+1,q} \oplus \mf{F}^{p+2,q}  = \mf{E}^{p+1.q}$ are diagonal. We will indicate with $\alpha_t$ the morphism $\mf{F}^{p+t,q} \rTo \mf{F}^{(p+1)+(t-1),q}$.
\end{remark}
\begin{pps}\label{pps: diffinv}
Let $p$ be odd, $0<p \leq k-1$, $q$ even, 
$-2p+2 \leq q \leq 0$, $1 \leq t \leq 2$.
The morphism of invariants 
$ \alpha_t: \mf{F}^{p+t,q} \rTo  \mf{F}^{(p+1)+(t-1),q} $ 
is an isomorphism.
\end{pps}
\begin{proof}The morphism  $\alpha_t$
is induced by the 
$G \times H$-equivariant morphism between relevant homogeneous components (see remark \ref{rmk: rel}): 
\begin{equation}\label{eq:Finv} 
f :
\bigoplus_{\substack{a \in \mc{I}^{p}_0 \\ t(a)=t}} F^{p,q}(a) \tens \epsilon_k
\rTo ^{\oplus_{a, b} f_{a, b}}  
\bigoplus_{\substack{b \in \mc{I}^{p+1}_0 \\ t(b) =t(a)-1}} F^{p+1,q}(b) 
\tens \epsilon_k  \;.
\end{equation}
Fix now $a \in \mc{I}^{p}_0$, 
with $t(a)=t$ 
and $b \in 
\mc{I}^{p+1}_0$, $t(b)=t-1$, such that the morphism $f_{a, b}~:~F^{p,q}(a) \rTo 
F^{p+1,q}( b) $ is non zero: it has  to be of the form described in remark
\ref{rmk: diffinvprimo}.
We can decompose $\Stab_{G \times H}(b)$ as a direct product $\Stab_{G \times H}(b) \simeq P \times Q$ where
$ Q = G((I_0(a) \setminus J(a)) \cup \{ j_0\})$, $
P  = H(S_0(a) \cup \{j_0\}) \times G_2(a) \times G_3(a)$. Now $Q$ acts trivially on $F^{p,q}(b) \tens \Ind^H_{\Stab_{G \times H}(b)} \epsilon_k
$, by remark \ref{rem:nonzeroinv}. Moreover $\Delta(I_0(a) \cap J(a))$ acts trivially on $F^{p,q}(a) \tens \Ind_{\Stab_{G \times H}(a)} \epsilon_k$ and 
$G_i(a) = G_i(b)$ for $i=2,3$. Hence: 
$$ \left( F^{p,q}(a) \tens \Ind^H_{\Stab_{G \times H}(a)} \epsilon_k\right) ^{\Stab_{G \times H}(a)} = \left( F^{p,q}(a) \tens \Ind^H_{\Stab_{G \times H}(a)} \epsilon_k\right) ^{\Stab_{G \times H}(a) \cap \Stab_{G \times H}(b)} \;.$$Moreover, the hypothesis of 
lemma \ref{DMS} are verified: we obtain that the morphism $f^{G \times H}$ of $G \times H$-invariants of (\ref{eq:Finv}) coincides 
with $|Q| \cdot \tilde{f}^P$ where $\tilde{f}^P$ is the morphism of $P$-invariants of the map:
$$ \tilde{f}:  \bigoplus_{[g] \in P/\Stab_{G \times H}(a)} F^{p,q}(ga) \tens 
\Ind^H_{P} \epsilon_k
 \rTo  
F^{p,q}(b) \tens \Ind^H_{P} \epsilon_k \;.
 $$Now the orbit $P/ \Stab_{G \times H}(a)$ of $P$ in $G \times H / \Stab_{G \times H}(a)$ is in bijection with the cosets $H(S_0(a) \cup \{ j_0 \}) / H(S_0(a))$.  Keeping into account that $G_i(ga)=G_i(b)$ for $i=2,3$ and for all $g \in H(S_0(a) \cup \{ j_0 \})$, and decomposing each term 
as in remark \ref{rimarcotensor}, we  see that the map $\tilde{f}^P$ is induced by the map 
of $H(S_0(a) \cup \{ j_0 \})$-invariants of the map: 
\begin{equation}\label{eq: f1} 
 \bigoplus_{[g] \in 
H(S_0(a) \cup \{ j_0 \})/H(S_0(a))} F_1^{p,q}(ga) \tens 
 \epsilon_{p+1} 
 \rTo 
F_1^{p,q}(b) \tens  \epsilon_{p+1}  \;,
\end{equation}
where here $\epsilon_{p+1}$ is the alternant representation of $H(S_0(a) \cup \{ j_0 \})$. The map (\ref{eq: f1}) is in turn induced at the level of section by the map of $H(S_0 \cup \{ j_0 \})$-sheaves on $U^{I_0(a)} \simeq U^2$:
$$ \eta : \bigoplus_{[g] \in H(S_0 \cup \{ j_0\}) / H(S_0)} \mc{F}_1^{p,q}(ga)
\tens \epsilon_{p+1} \rTo^{\oplus_{[g]} \eta_{ga,b}} \mc{F}^{p+1,q}_1(b) \tens \epsilon_{p+1} \;.
$$
The proposition follows from the following lemma
and lemma \ref{invaperm}, point 4. We use notations \ref{not: repstandard(i)}.
\end{proof}
\begin{lemma}The map of $\perm_{p+1}$-invariants of the map 
$\eta$
 can be identified, up to a sign, with the map of 
$\perm_{p+1}$-invariants of the map: 
\begin{equation}\label{eq: identif} \oplus_{i=1}^{p+1} \gamma_i : \bigoplus_{i=1}^{p+1} \Lambda^{-q}(N^*_{\Delta} \tens \rho_{p}(i)) \tens L_{\Delta}^{\tens ^{p+t}} \rTo 
\Lambda^{-q}(N^*_{\Delta} \tens \rho_{p+1}) \tens L_{\Delta}^{\tens ^{p+t}} \; ,\end{equation}induced by natural inclusions $\rho_p(i) \rTo \rho_{p+1}$.
\end{lemma}
\begin{proof}It is clear that it is sufficient to prove the lemma when $L$ is trivial. 
We can identify $S_0 \cup \{j_0 \} \simeq \{1, \dots, p+1\}$ and $S_0 \simeq \{1, \dots,p\}$; hence  $H(S_0 \cup \{j_o \})/H(S_0) \simeq \perm_{p+1} / \perm_p$
can be identified with $\{1, \dots, p+1\}$ via the bijection carrying $i$ to the class of any permutation $\theta_i$ such that $\theta_i(p+1)=i$. Hence, by lemma \ref{Danilamorphism} and remark \ref{rmk: dmss}, the map of $\perm_{p+1}$-invariants $$ 
\eta^{\perm_{p+1}} : (\mc{F}^{p,q}_1(a) \tens \epsilon_{p})^{\perm_p} \simeq 
\left( \oplus_{i=1}^{p+1} \mc{F}^{p,q}_1( a\theta_i^{-1}) \tens \epsilon_{p+1}
\right)^{\perm_{p+1}} \rTo ( \mc{F}^{p+1,b}_1(b) \tens \epsilon_{p+1})^{\perm_{p+1}}$$is given 
by $$\eta^{\perm_{p+1}}(u) = 
\sum_{i=1}^{p+1} \theta_i
\eta_{a,b}(u) \;.
$$
In notations \ref{not: repstandard(i)}, we can now identify $\mc{F}^{p,q}_1(a) \tens \epsilon_{p} \simeq 
\Lambda^{-q}(N^*_{\Delta} \tens \rho_p(p+1))$ and 
$\mc{F}^{p+1,b}_1(b) \tens \epsilon_{p+1} \simeq \Lambda^{-q}(N^*_{\Delta} \tens \rho_{p+1})$, as $\perm_p$ and $\perm_{p+1}$-representations, respectively; 
moreover, the map $\eta_{a,b}$ identifies, up to the sign 
$(-1)^{p}$ coming from the definition of differential\footnote{We recall that the differential 
 $d^p_{\comp{\mc{C}}_L \tens \dots \tens \comp{\mc{C}}_L}$ acts on the term $\mc{C}^1 \tens \dots \tens \mc{C}^1 \tens \mc{C}^0 \tens \mc{C}^1 \tens \dots \tens \mc{C}^1$, where $\mc{C}^0$ is on the $i$-th factor, as $(-1)^{i-1} \id \tens \dots \tens \id \tens d^0_{\mc{C}^0} \tens \id \tens \dots \tens \id$.
  Hence the map $\eta_{a\theta_i^{-1},b}:\mc{F}^{p,q}_1(a \theta_i^{-1}) \rTo \mc{F}^{p,q}_1(b)$
 will also be affected by this sign. }
 $d^p_{\comp{\mc{C}}_L \tens \dots \tens \comp{\mc{C}}_L}$, to the natural inclusion
$\gamma_{p+1}$, considered in  corollary \ref{naturalissimo}. Hence $\eta^{\perm_{p+1}}$ can be identified, up to a sign, to $\sum_{i=1}^{p+1} \theta_i \gamma_{p+1}(u)$, which 
is, by remark \ref{rmk: dmss}, the map of $\perm_p$-invariants of the map (\ref{eq: identif}).
\end{proof}The following proposition is now consequence of remarks \ref{rmk: diffinvprimo}, \ref{rmk: diffinvsecondo} and proposition \ref{pps: diffinv}.
Consider the complexes \begin{gather*} A_i^{\bullet,q} := 0 \rTo \mf{F}^{2i,q} \oplus \mf{F}^{2i+1,q} \rTo^{(\alpha_1,\alpha_2)} \mf{F}^{2i,q} \oplus \mf{F}^{2i+1,q} \rTo 0 
\end{gather*}concentrated in degrees -1 and 0, where $\alpha_1$ and $\alpha_2$ 
 are isomorphisms.
Remark that 
they vanish if $i > k/2$; moreover
for 
$k$ even, $A_{k/2}^{\bullet,q}$ becomes: $A_{k/2}^{\bullet,q} =   0 \rTo \mf{F}^{k,q} \rTo^{\alpha_1} \mf{F}^{k,q} \rTo 0$. 
\begin{pps}\label{pps:isocomplexes}The complexes $\mf{E}^{\bullet, q}$ are isomorphic to: 
$$ \mf{E}^{\bullet, q} \simeq \left \{ 
\begin{array}{ll}  \mf{F}^{0,0} \oplus \bigoplus_{i=1}^{k/2}A^{\bullet,q}_i[-2i]  & \mbox{if $q=0$} \\ \\
(\mf{F}^{-q/2 +1,q} \oplus\mf{F}^{-q/2 +2,q}) [+q/2-1] \oplus \bigoplus_{i=-q/4+3/2}^{k/2}A^{\bullet,q}_i[-2i]  & \mbox{if $-2k+2 \leq q<0$}  \\ & \quad \mbox{and $q \equiv 2 \mod 4$}\\
\bigoplus_{i= -q/4+1}^{k/2}A^{\bullet,q}_i[-2i]  & \mbox{if $-2k+2 \leq q<0$} 
\\ & \quad \mbox{and $q \equiv 0 \mod 4$} 
\end{array} \right.\;$$ 
and they vanish elsewhere. Consequently, for $q<0$, they are exact in degree $p > -q/2 +1$.
\end{pps}
The following corollary is a consequence of remark 
\ref{rem:rest} and 
propositions \ref{pps:isocomplexes}.  
\begin{crl}\label{crl: degeneration}
The spectral sequence $j^* \mc{E}^{p,q}_1(\lambda_{\epsilon})$ degenerates at 
level 
$\mc{E}_2$.  
\end{crl}We can finally prove the following formula.
\begin{theorem}\label{BDext}
Let $X$ be a smooth quasi-projective algebraic surface, 
$L$ a line bundle on $X$. Let $X^{[n]}$ be the Hilbert scheme of $n$ points over $X$ and $L^{[n]}$ the tautological bundle on $X^{[n]}$, naturally associated to the line bundle $L$ on $X$. Then, for any $0 \leq k \leq n$, we have: 
\begin{equation}\label{exterior}\B{R} \mu_*(\Lambda^k L^{[n]}) \simeq 
(\Lambda^k \mc{C}^0_L)^G \;.
\end{equation}
\end{theorem}
\begin{proof}The degeneration of the spectral sequence $j^*\mc{E}^{p,q}_1(\lambda_{\epsilon})$ at level $\mc{E}_2$, corollary \ref{qoro} and
the fact that the limit of $j^* \mc{E}^{p,q}_1(\lambda_{\epsilon})$ is zero above the diagonal, 
 imply that the complex $j^*\mc{E}^{\bullet,0}_1(\lambda_{\epsilon}) \simeq 
j^*(\Lambda^k \comp{\mc{C}}_L)^{G}$ is a resolution of
the sheaf $j^*\mu_*(\Lambda^k L^{[n]})$. Moreover, by proposition  \ref{pps:isocomplexes} the morphism $\mf{E}^{0,0} \rTo \mf{E}^{1,0}$ is zero; hence, by lemma 
\ref{lemmino}, the morphism: $j^*\mc{E}^{0,0}_1(\lambda_{\epsilon}) \rTo 
j^*\mc{E}^{1,0}_1(\lambda_{\epsilon})$ is zero. Consequently the complex $j^*
(\Lambda^k \comp{\mc{C}}_L)^G$ is quasi-isomorphic to the sheaf $j^*(\Lambda^k \mc{C}^0_L)^G$. Hence: 
$$ j^*\mu_*(\Lambda^kL^{[n]}) \simeq  j^*(\Lambda^k \comp{\mc{C}}_L)^G  \simeq j^*(\Lambda^k \mc{C}^0_L)^G \;.$$Now, applying the functor $j_*$ and recalling lemmas \ref{414} and \ref{415} we get the result. 
\end{proof} 
\section{Cohomology}\label{sect:cohom}
The aim of this section is to apply Brion-Danila-type formulas, found in the 
previous sections, to cohomology computations. In particular 
we will compute the cohomology of Hilbert schemes with values in the double tensor power $L^{[n]} \tens L^{[n]}$ of tautological bundles and in their general exterior powers $\Lambda^kL^{[n]}$. We obtain analogous results for the cohomology with values in  $L^{[n]} \tens L^{[n]} \tens \mc{D}_A$ and 
$\Lambda^k L^{[n]} \tens \mc{D}_A$, where $\mc{D}_A$ is the determinant line bundle on $X^{[n]}$ associated to a line bundle $A$ on $X$.
In this way we reobtain Danila's results 
\cite{Danila2001} on the cohomology of a tautological bundle, we generalize her results on 
the double symmetric power \cite{Danila2004}, and we give new general 
formulas for the exterior powers.
\subsection{Cohomology with values in the double tensor power}
Recall that we indicate with $G$ the symmetric group $\perm_n$. The short exact sequence (\ref{eq:BDk=2}):
$$ 0 \rTo \mu_*(L^{[n]} \tens L^{[n]}) \rTo (\mc{C}^0_L \tens \mc{C}^0_L)^G \rTo ^{(\del^0_L \tens \id)^G} (\mc{C}^1_L \tens \mc{C}^0_L)^G \rTo 0 $$and the fact that $R^i \mu_*(L^{[n]} \tens L^{[n]}) =0$ for $i>0$ induce a long exact cohomology sequence: 
\begin{multline} \cdots
\rTo H^*(X^{[n]}, L^{[n]} \tens L^{[n]}) \rTo H^*_G(X^n, \mc{C}^0_L \tens \mc{C}^0_L) \rTo  \\ \rTo H^*_G(X^n, \mc{C}^1_L \tens \mc{C}^0_L) \rTo H^{*+1}(X^{[n]}, L^{[n]} \tens L^{[n]}) \rTo \cdots \;.
\end{multline}Now, with the same computations made in the proof of proposition 
\ref{thm:BDk=2}, we get: 
\begin{align*}
 H^*_G(X^n, \mc{C}^0_L \tens \mc{C}^0_L) = & \; H^*(L^{\tens^2}) \tens S^{n-1} H^*(\FS_X) \bigoplus H^*(L)^{\tens ^2} \tens S^{n-2} H^*(\FS_X) \\
H^*_G(X^n, \mc{C}^1_L \tens \mc{C}^0_L) = &\; H^*(L^{\tens ^2}) \tens S^{n-2} H^*(\FS_X)
\end{align*}The differential $ (\del^0_{L } \tens \id)^{G} : (\mc{C}^0_L \tens \mc{C}^0_L)^G \rTo ^{} (\mc{C}^1_L \tens \mc{C}^0_L)^G $ induces the morphism: 
\begin{equation} \label{d} H^*(L^{\tens^2}) \tens S^{n-1}H^*(\FS_X) \bigoplus H^*(L)^{\tens^2} \tens
S^{n-2} H^*(\FS_X) \rTo H^*(L^{\tens^2}) \tens S^{n-2}H^*(\FS_X) \;.
\end{equation} 
We will prove now that it is an epimorphism: consequently the long exact cohomology sequence splits. It suffices to prove the surjectivity of the first component: 
$$ D_{L^{\tens 2}}:=H^*(L^{\tens ^2}) \tens S^{n-1}H^*(\FS_X) \rTo H^*(L^{\tens ^2}) \tens S^{n-2} H^*(\FS_X) \;.$$Remark that the second component is induced by the canonical coupling: $H^*(L)^{\tens 2} \rTo H^*(L^{\tens 2})$. We begin by the following: 
\begin{lemma}\label{lem:epicohom}
Let $F$ be a line bundle on $X$. Consider, for $k \in \mbb{N}^*$, the
embedding:
\begin{equation}\label{KKK}
 X \times S^{k-1}X \rInto X \times S^kX 
\end{equation}given by $(x,z) \rMapsto (x,
x+z)$. The restriction morphism: 
$$ D_F: H^*(F) \tens S^kH^*(\FS_X) \rTo H^*(F) \tens S^{k-1}
H^*(\FS_X)$$induced by this embedding is given, for $\alpha \in 
H^*(F)$ and $u_i \in H^*(\FS_X)$, $i=1, \dots, k$, homogeneous of
degree $p_i$, by the formula:
$$ \alpha \tens u_1 \cdots u_k \rTo \frac{1}{k} \sum_{i=1}^k
(-1)^{(\sum_{j<i} p_j)p_i} \alpha u_i \tens u_1 \cdots \hat{u_i} \cdots u_k \;.$$ 
\end{lemma}
\begin{proof}The embedding (\ref{KKK}) is induced by the embedding:
\begin{diagram}[height=.6cm] 
X \times X^{k-1} & \rInto & X \times X^k \;. \\
(x, z_1, \dots, z_{k-1}) & \rMapsto & (x,x,z_1, \dots, z_{k-1}) 
\end{diagram}
The induced morphism in cohomology is given by:
$ \alpha \tens u_1 \tens \cdots \tens u_k \rTo \alpha u_1 \tens u_2
\tens \cdots \tens u_k$. The canonical projection $X^k \rTo S^k X$
identifies $u_1 \cdots u_k \in H^*(S^kX) \simeq S^kH^*(\FS_X)$ with the
element:
$$ \frac{1}{k!} \sum_{\sigma \in \perm_k} \epsilon_{\sigma, p_1,
  \dots, p_k} \sigma(u_1 \tens \cdots \tens u_k) $$of $H^*(\FS_{X^k})
  \simeq H^*(\FS_X) ^{\tens k}$, where $\epsilon_{\sigma, p_1,
  \dots, p_k}$ is a sign depending on the permutation $\sigma$ and on the
  respective degrees of $u_i$ and it is characterized, in the graduated
  algebra $\mbb{C}[u_1, \dots, u_k]$ (where $\textrm{deg}(u_i) =p_i$), 
  by the relation: $\sigma(u_1 \tens \cdots \tens u_k) =   
\epsilon_{\sigma, p_1,
  \dots, p_k} u_1 \tens \cdots \tens u_k $. For the product 
of consecutive transpositions $\tau=\tau_{1,2} \cdots  \tau_{i-1,i}$  
we
  have $\epsilon_{\tau, p_1,
  \dots, p_k} = (-1)^{(\sum_{j<i}p_j )p_i}$. The formula follows. 
\end{proof}
\begin{lemma}
The restriction 
morphism $D_F : H^*(F) \tens S^k H^*(\FS_X) \rTo H^*(F) \tens
S^{k-1} 
H^*(\FS_X) $ is surjective and has a canonical section. 
\end{lemma}
\begin{proof} Let $u_i \in H^*(\FS_X)$ be of degree $p_i$, $i=1, \dots,
k$, $\alpha \in H^*(F)$. 
Let $\lambda_j$ be the morphism:
\begin{diagram}[height=.6cm]
\lambda_j : H^*(F) \tens S^j H^*(\FS_X) & \rTo &  H^*(F) \tens S^{k-1} 
H^*(\FS_X) \\
\alpha \tens u_1 \cdots u_j & \rMapsto & \alpha \tens
\underset{\textrm{$k-j-1$-times}}{\underbrace{1 \cdots1}} \cdot \; u_1 \cdots u_j
\end{diagram}We have $\lambda_{k-1}=\id$, $\lambda_{-1}=0$. Let $W_i$ be 
the image of $\lambda_i$. We have the filtration:
$$ \{ 0 \} = W_{-1} \subsetneq W_0 \subsetneq \dots \subsetneq W_{k-2}
\subsetneq W_{k-1} = H^*(F) \tens S^{k-1} 
H^*(\FS_X) \; .$$Let now $\sigma$ be the morphism:
\begin{diagram}[height=.6cm] H^*(F) \tens S^{k-1} 
H^*(\FS_X) & \rTo & H^*(F) \tens S^{k} 
H^*(\FS_X) \\
\alpha \tens u_1 \cdots u_{k-1} & \rMapsto & \alpha \tens 1 \cdot u_1 \cdots
u_{k-1} \end{diagram}Remembering the expression of the morphism $D$ from lemma 
\ref{lem:epicohom} we have the following relation:
\begin{align*}
(D_F \circ \sigma)(\lambda_j(\alpha \tens u_1 \cdots u_j) ) = & \;
D_F(\alpha \tens \underset{\textrm{$k-j$-times}}{\underbrace{1 \dots1}}
\cdot \; u_1 \cdots u_j ) \\  = &\; \frac{k-j}{k} \alpha  \tens
\underset{\textrm{$k-j-1$-times}}{\underbrace{1 \cdots 1}} \cdot \; 
u_1 \cdots u_j
+ \\ & \qquad \qquad \; +\sum_{h=1}^j C^j_h \alpha u_h \tens
\underset{\textrm{$k-j$-times}}{\underbrace{1 \cdots 1}} \cdot  \; u_1 \cdots
\hat{u_h} \cdots u_j \\ 
=&\; \frac{k-j}{k} \lambda_j(\alpha \tens u_1 \cdots u_j) +v  
\end{align*}where $v \in W_{j-1}$, for some rational constants
$C^j_h$. This means that, indicating with 
$\Psi$ the endomorphism $D_F \circ \sigma$ of $H^*(F) \tens
S^{k-1} H^*(\FS_X)$ we have: 
$ \left ( \Psi - (k-j)/k \right) (W_j) \subseteq W_{j-1}$, which implies:
$$ \prod_{j=0}^{k-1} \left( \Psi - \frac{k-j}{k} \right) =0 \;.$$
Consequently, 
$\Psi$ is invertible. Therefore $D_F$ is surjective and has a
canonical section. 
\end{proof}\noindent
A consequence of this lemma is that the kernel of $D_F$ is
isomorphic to a direct factor of the image of $\sigma$. 
The next lemma
allows us to characterize such a direct factor. 
\begin{lemma}
Let $a \in X$ be a point of $X$. Consider the morphism: 
$$ \id \tens \nu : H^*(F) \tens S^kH^*(\FS_X) \rTo H^*(F) \tens
S^{k-1}H^*( \FS_X) $$where $\nu$ is the morphism induced in cohomology
by the inclusion $S^{k-1}X \rInto S^kX$ given by $z \rMapsto
a+z$. Therefore $\id\tens \nu$ is surjective and its kernel is
isomorphic to the kernel of $D_F$. 
\end{lemma}
\begin{proof} The morphism $\id \tens \nu$ is given by: 
$$ \id \tens \nu (\alpha \tens u_1 \cdots u_k ) = 
\frac{1}{k} \sum_{i=1}^k
(-1)^{(\sum_{j<i} p_j)p_i} \alpha u_i(a) \tens u_1 \cdots \hat{u_i} \cdots
u_k \;.$$We know that $u_h(a)=0$ if $\textrm{deg} \;u_h >0$. Therefore,
if we denote with $\tilde{\Psi}$ the endomorphism 
$\id \tens \nu \circ \sigma$ of $H^*(F) \tens S^{k-1} H^*(\FS_X)$, we
have:
\begin{eqnarray*} 
\tilde{\Psi}(\lambda_j(\alpha \tens u_1 \cdots u_j )) & = & 
(\id \tens \nu)(\alpha \tens
\underset{\textrm{$k-j$-times}}{\underbrace{1 \cdots 1}} \cdot \; u_1 \cdots u_j
)  \\ & = & 
\frac{k-j}{k} \lambda_j (\alpha \tens u_1 \cdots u_j) +v 
\end{eqnarray*}where $v \in W_{j-1}$. Therefore 
$ \left (\tilde{\Psi} -(k-j)/k \right)(W_j) \subseteq W_{j-1}$
and 
we have again for 
$\tilde{\Psi}$ the relation:
$$ \prod_{j=0}^{k-1} \left(  \tilde{\Psi} - \frac{k-j}{k}
\right) =0 $$which implies that $\tilde{\Psi}$ is invertible, 
that $\id \tens \nu$ is surjective and that $\pim \, \sigma$ is a direct
factor of $\ker (\id \tens \nu)$. \end{proof}
The immediate consequence of the previous lemma is that:
\begin{equation}\label{eq: prel}
\ker D_F \simeq \ker(\id \tens \nu) \simeq H^*(F) \tens \ker \nu \;.
\end{equation}
Now it is an elementary fact that if $V$, $W$, $Z$ are 
three vector spaces, not necessarily of finite
dimension over a field $k$ and if 
$ F = (f,g) :V \oplus W \rTo Z$ is a linear map such that the component
$f$ is surjective, then $\ker F \simeq \ker f \oplus W$. This fact and equation
(\ref{eq: prel}) applied to morphism 
 (\ref{d}) yield:
\begin{theorem}Let $X$ be a smooth quasi-projective surface. 
Let $a$ be a point in $X$. Let $\mc{J}$ be the kernel of the morphism:
$S^{n-1}H^*(\FS_X)  \rTo S^{n-2} H^*(\FS_X) $, induced by the morphism:
\begin{diagram}[height=.6cm]
S^{n-2}X & \rTo & S^{n-1}X \\
x & \rMapsto & a+x
\end{diagram}
We have the isomorphism of $\mbb{Z}$-graded modules and $\perm_2$-representations:
$$ H^*(X^{[n]}, L^{[n]} \tens L^{[n]}) \simeq   H^*(L^{\tens^2}) \tens \mc{J}
\bigoplus H^*(L)^{\tens ^2} \tens S^{n-2}H^*(
\FS_X) \;. $$
\end{theorem}
We introduce now the determinant line bundle $\mc{D}_A$, associated to a line bundle $A$ on the surface
$X$. Consider the bundle $A^{ \boxtimes ^n} := 
A \boxtimes \dots
\boxtimes A$
on the product $X^n$. By
Drezet-Kempf-Narasimhan lemma \cite[Th\'eor\`eme 2.3]{DrezetNarasimhan1989}, 
it descends to a line bundle $A^{\boxtimes^{n}}/\perm_n$ on the quotient 
$X^n /\perm_n \simeq S^n X$. 
\begin{definition}
We call determinant line bundle $\mc{D}_A$ on the Hilbert scheme
$X^{[n]}$ the pull back of $A^{\boxtimes^n}/\perm_n$ for the Hilbert-Chow morphism:$$ \mc{D}_A := \mu^*( A^{\boxtimes ^n}/
\perm_n) \; .$$
\end{definition}
Tensorizing the short exact sequence 
(\ref{eq:BDk=2}) with $\mc{D}_A$ and taking the cohomology, we get:
\begin{theorem}
\label{thm: longexact}
Let $X$ be a smooth quasi-projective algebraic surface and $L$, $A$ be line bundles on $X$. Then we have the long exact sequence: 
\begin{multline*}\label{eq: longexact}
\cdots \rTo H^*(X^{[n]}, L^{[n]} \tens L^{[n]} \tens \mc{D}_A) \rTo \begin{array}
{c}H^*(L^{\tens ^2} \tens A) \tens S^{n-1} H^*(A) \\ \bigoplus \\  
H^*(L \tens A)^{\tens ^2} \tens S^{n-2}H^*(A)  \end{array} \\ 
\rTo^{m} H^*(L^{\tens ^2} \tens A^{\tens ^2} ) \tens S^{n-2}H^*(A) \rTo H^{*+1} 
(X^{[n]}, L^{[n]} \tens L^{[n]} \tens \mc{D}_A) \rTo \cdots 
\end{multline*}The two components $m_1$ and $m_2$ of the map $m$ are given by: 
\begin{align*}
m_1(\alpha \tens u_1 \cdots  u_{n-1}) = &\; \frac{1}{n-1}\sum_{i=1}^{n-1} \alpha u_i \tens (-1)^{(\sum_{j<i}p_j) p_i}u_1  \cdots  \hat{u_i}  \cdots  u_{n-1}
\\ 
m_2(\beta \tens \gamma \tens v) = &\; \beta \gamma \tens v
\end{align*}where $\alpha \in H^*(L^{\tens ^2}\tens A)$, 
$\beta, \gamma \in H^*(L \tens A)$, $u_1, \dots, u_n \in H^*(A)$, with $u_i$  of degree $p_i$ and
$v \in S^{n-2}H^*(A)$.
\end{theorem}
\begin{proof}With a completely analogous computation as the one done in the proof 
of theorem~\ref{thm:BDk=2} we can compute:
\begin{align*}
H^*(X^n, (\mc{C}^0_L \tens \mc{C}^0_L ) \tens A^{\boxtimes})^G = & \; H^*(L^{\tens ^2} \tens A) \tens S^{n-1} H^*(A)  \bigoplus  H^*(L \tens A)^{\tens ^2} \tens S^{n-2}H^*(A) \\
H^*(X^n, (\mc{C}^1_L \tens \mc{C}^0_L ) \tens A^{\boxtimes})^G = & H^*(L^{\tens ^2} \tens A^{\tens ^2} ) \tens S^{n-2}H^*(A) \;.
\end{align*}The expression of the map $m_1$ can be computed exactly as in the proof of lemma \ref{lem:epicohom}; the expression of $m_2$ is evident.
\end{proof}
\subsection{Cohomology with values in exterior powers}\label{cohoext}
By theorem \ref{BDext}, by definition of $\mc{D}_A$ and by the projection formula we get: 
$$ \B{R}\mu_*(\Lambda^k L^{[n]} \tens \mc{D}_A) =\left( \Lambda^k \mc{C}^0_L  
\right)^G \tens A^{\boxtimes ^n} /\perm_n  \;,$$generalizing Brion-Danila formula for all $k$ (see \cite[prop. 6.1]{Danila2001}).
We have the following formula for the cohomology of the general exterior power $\Lambda^kL^{[n]}$, twisted by the determinant~$\mc{D}_A$:
\begin{theorem}
Let $X$ be a smooth quasi-projective surface and $L$, $A$ be line bundles on $X$. Then: 
$$ H^*(X^{[n]},  \Lambda^k L^{[n]} \tens \mc{D}_A) = \Lambda^k H^*(L \tens A)
\tens S^{n-k}H^*(A) \;.$$
\end{theorem}
\begin{proof}
\begin{eqnarray*}
\B{R}\Gamma_{X^{[n]}}(\Lambda^k L^{[n]} \tens \mc{D}_A) &=& \B{R} 
\Gamma_{S^n X } \circ 
\B{R} \mu_* (\Lambda^k L^{[n]} \tens \mc{D}_A) \\
& = & \B{R} \Gamma_{S^n X} \circ \pi_*^{G} \left( \Lambda^k \mc{C}^0_L  
  \tens \pi^* A^{\boxtimes ^n}/\perm_n \right) \\
 & = & \B{R} \Gamma^{G}_{X^n}( \Lambda^k \mc{C}^0_L  
 \tens A^{\boxtimes ^n}) \\
& = & \B{R} \Gamma_{X^n}^{\perm_n \times \perm_k} \left( (\mc{C}^0_L )^{\tens ^k}  \tens A^{\boxtimes ^n} \tens \epsilon_k \right) \end{eqnarray*}Then, by Danila's lemma and K\"unneth formula:
\begin{eqnarray*}H^*(X^{[n]}, \Lambda^k L^{[n]} \tens \mc{D}_A) & = & 
\left[ H^*(L \tens A)^{\tens ^k } \tens \epsilon_k \right]^{\perm_k} \Tens \left [ 
H^*(A)^{\tens ^{n-k}}
\right]^{\perm_{n-k}} \\
& = & \Lambda^k H^*(L \tens A) \tens S^{n-k} H^*(A)
\end{eqnarray*}
\end{proof}
\appendix
\numberwithin{equation}{section}
\section{The {\v C}ech complex for closed subschemes}
\label{app: Cech}We prove here, under some reasonable transversality 
hypothesis, 
the existence of a {\v C}ech-type resolution for a finite 
scheme-theoretic union of closed subschemes of smooth scheme: it can be thought as a generalization of the Chinese remainder theorem. We need the following lemmas. 
\begin{lemmaa}[Peskine-Szpiro, Kempf-Laksov {\cite[lemma 7]{KempfLaksov1974}}]
\label{PS}
Let $(A, \G{m})$ be a Cohen-Macauley noetherian 
local ring and $I \subseteq A$ an ideal.
Let \[
0 \rTo K^0 \rTo K^1 \rTo \dots \rTo K^{n-1} \rTo K^n \rTo 0 
\]be a complex of free modules. Suppose that
$\mathrm{Supp}(\comp{K}):=\bigcup_{i=1}^n \mathrm{Supp}H^i(\comp{K}) \subseteq 
V(I)$. Then $H^i(\comp{K})=0$ for all $i < {\rm ht}(I)$.
\end{lemmaa}
\begin{lemmaa}\label{excess}
Let $(A, \G{m})$ be a  
noetherian regular local ring, $M_{1}, \dots, M_{k}$ finite Cohen-Macauley 
modules over $A$. 
 Let $$c(M_1, \dots, M_k) := \left( 
\sum_{i=1}^k \codim M_i  \right ) - \codim M_1 \tens \dots \tens M_k \; .$$
Then $ \Tor_i(M_1, \dots, M_k)=0 $ for $i \gneq c(M_1, \dots,M_k) $. 
\end{lemmaa}\begin{proof} The lemma is an easy consequence of 
Peskine-Szpiro
lemma and the existence, for the Cohen-Macauley modules $M_i$, 
of finite free resolutions of length equal to the codimensions $\codim M_i$, by Auslander-Buchsbaum formula \cite[Theorem 19.1]{MatsumuraCRT}.
For every module 
$M_i$ let's take its finite free resolution 
$\comp{R}_i \rTo M_i \rTo 0$, written:
$$ 0 \rTo R_i ^0 \rTo R_i^1 \rTo \dots \rTo R^{\codim M_i}_i \rTo
M_i 
\rTo 0 \;.$$
We can then compute $\Tor_i(M_1, \dots, M_k)$ as the cohomology of the 
total complex: $\comp{R} := \comp{R_1} \tens \dots \tens \comp{R}_k $. Now 
$\comp{R}$ is a finite complex of free modules of length $l= \sum_{i=1}^k 
\codim M_i$ and, for all $i$, 
$\Tor_i(M_1, \dots, M_k)= H^{l-i}(\comp{R})$ is supported in $\Supp(M_1 \tens 
\dots \tens M_k) = V(\Ann(M_1 \tens \dots \tens M_k))$. Therefore by 
Peskine and Szpiro lemma, $H^{l-i}(\comp{R}) =0$ for $l-i < 
\hth( \Ann(M_1 \tens \dots \tens M_k))$, that is if $i > 
l - \hth( \Ann(M_1 \tens \dots \tens M_k))$. Now for a 
noetherian regular
local ring, $\hth(I)= \codim V(I) = 
\dim A - \dim A/I $, and this implies the result. \end{proof}
Let $I \subseteq \{1, \dots, l \}$. We will indicate with $\overline{I}$ the 
complementary of $I$ in $\{1, \dots ,l\}$. If $M_1 ,\dots, M_l$ are modules 
on a ring $A$, we will indicate with $M_I = \tens_{i \in I} M_i$, with 
$\Tor_i(M_I) = H^{-i}(\tens_{i \in I}^L M_i)$ and with $c(M_I) = c(M_{i_1}, \dots, M_{i_h})$ if $I=\{i_1, \dots, i_h \}$.
\begin{ppsa}{\sloppy Let $(A, \mf{m})$ be a noetherian regular local ring and 
$M_i$, $i=1, \dots, l$, Cohen-Macauley modules on $A$. Consider the exact 
sequences: 
$$ 0 \rTo N_i \rTo E_i \rTo M_i \rTo 0 $$where $E_i$ are free $A$-modules. 
Let $\comp{K}_i$ be the complex (in degree 0 and 1): 
$\comp{K}_i:= E_i \rTo M_i \rTo 0$.
Suppose that $c(M_1, \dots , M_l)=0$. Then the complex 
$ \comp{K} := \tens_i \comp{K}_i :$
\begin{multline*} 0  \rTo \tens_{i=1}^l E_i \rTo \oplus_{i=1}^l M_i \tens E_{ \overline{\{ i \}}} 
\rTo \oplus_{|I|=2 } M_I \tens E_{\overline{I}} \rTo \dots \rTo \tens_{i=1}^l 
M_i \rTo 0
\end{multline*}is a right resolution of the module $\tens_{i=1}^lN_i$ .} 
\end{ppsa}\begin{proof}
We first prove that for all $\emptyset \neq H \subseteq \{1, \dots, l \}$ one has:
\begin{equation}\label{ttx}0 \leq c(M_H) \leq c(M_1, \dots, M_l) \;.
\end{equation} To prove the first inequality 
it is not restrictive to take $H= \{1, \dots, l \}$. Now, embed $X:= \Spec(A)$ 
into $X^l$ with the diagonal immersion $i : X \rInto X^l$. On $X^l$ we have: 
$\codim_{X^l}(M_1 \boxtimes \cdots \boxtimes M_l) = \sum_i \codim M_i$. Then, since 
$M_1 \tens \dots \tens M_l = i^*(M_1 \boxtimes \cdots \boxtimes M_l)$, in the identification given by $i$,
$ \Supp(M_1 \tens \cdots \tens M_l ) = 
\Delta \cap \Supp(M_1 \boxtimes \cdots \boxtimes M_l)$ and since $X$ is a 
smooth scheme: 
$$ \dim \Delta \cap \Supp(M_1 \boxtimes \cdots \boxtimes M_l)  \geq \dim \Delta 
+ \dim \Supp(M_1 \boxtimes \cdots \boxtimes M_l) - \dim X^l \;, $$which implies
$ \codim_{X}(M_1 \tens \cdots \tens M_l) \leq \sum_{i} \codim M_i$. 
To prove the second inequality, we have, because of the first inequality in 
(\ref{ttx}): 
\begin{eqnarray*} c(M_1, \dots, M_l) & = & \sum_{i=1}^l \codim M_i - \codim(M_1 \tens \cdots \tens M_l) \\
& \geq & c(M_H) + c(M_{\overline{H}}) \geq c(M_H) \;.
\end{eqnarray*}

Now, to prove the exactness of the complex $\tens_{i =1}^l \comp{K}_i$, remark that, by lemma \ref{excess} and by the inequality (\ref{ttx}), the hypothesis
$c(M_1, \dots, M_l)=0$ implies \begin{equation}\label{torello}
\Tor_i(M_H)=0 \quad  \forall \;  \emptyset \neq 
H \subseteq \{1, \dots, l \} \quad, \quad \forall i>0 \;.
\end{equation}
The complexes $\comp{K}_i$ are  right 
resolutions of the modules $N_i$. Consequently, 
$ \Tor_{-q}(N_1, \dots, N_l) = H^q(\tens^L_{i} \comp{K}_i) $. 
On the other hand the cohomology $H^{p+q}(\tens^L_{i} \comp{K}_i)$
is the limit of the fourth quadrant spectral sequence: 
 $$ E^{p,q}_1:= \oplus_{i_1 + \dots +i_l =p }\Tor_{-q}(K^{i_1}_1, \dots, K^{i_l}_l) \;.$$Since $K^0_i$ is acyclic for all $i$ because free, the term $E^{p,q}_1$ becomes a direct sum:
$$ E^{p,q}_1 = \oplus_{|H|=p }\Tor_{-q}(M_H) \tens E_{\overline{H}} \;.$$
The vanishing (\ref{torello}) implies $E^{p,q}_1=0$ for all $q<0$. Note that the complex $E^{\bullet,0}_1$ is exactly $\tens_{i =1}^l \comp{K}_i$. The spectral sequence degenerates at level $E_2$; then
$ E^{p}_{\infty} \simeq E^{p,0}_2 = \Tor_{-p}(N_1, \dots ,N_l)=0$ if $p>0$ and $
E^{0}_{\infty} \simeq E^{0,0}_2=  \tens_{i=1}^l N_i $, hence the complex $\tens_{i =1}^l \comp{K}$ gives a right resolution of the module
$\tens_{i=1}^l N_i $. 
\end{proof}
\begin{crla}[Chinese Remainder Theorem]\label{ccx}Let $(A, \mf{m})$ be a noetherian regular local ring, and let $\mf{a}_i$, $i=1, \dots, l$, be ideals 
of $A$ such that:
\begin{enumerate}
\item the modules $A/\mf{a}_i$ are Cohen-Macauley for all $i$ 
\item $\hth(\sum_{i=1}^l \mf{a}_i) = \sum_{i=1}^l \hth(\mf{a}_i) $.
\end{enumerate}Then $\cap_{i=1}^l \mf{a}_i = \prod_{i=1}^l \mf{a}_i$ and the complex
$ (\comp{\check{\mc{ K}}}, \partial^{\bullet})$, defined by 
$$ \check{\mc{K}}^p= \oplus_{|I|=p+1} \left( A/\sum_{i \in I} \mf{a}_i \right) \qquad \partial^p(f)_J = 
\sum_{i \in J} \epsilon_{i,J}  \left( f_{J \setminus \{ i \}} \mod \sum_{j \in J} \mf{a}_j  \right)\;,$$where 
$\epsilon_{i,J}= (-1)^{\sharp \{ h \in J, h <i \}}$, 
is a right resolution of the module $A/ \cap_{i=1}^l \mf{a}_i$.
\end{crla}
\begin{proof}The result follows by applying the preceding proposition to the complexes 
$ A \rTo A/\mf{a}_i \rTo 0$, keeping into account that the second hypothesis is exactly the condition $c(A/\mf{a}_1, \dots, A/\mf{a}_l)=0$, since for a Cohen-Macauley ring  the height $\hth(I)$ of an ideal $I$ is the codimension of the module 
$A/I$.
\end{proof}
\section{Danila's lemma for morphisms}
\label{app: Danila}
We extend here Danila's lemma \ref{Danila} to morphisms. Let $I$, $J$ be two finite sets with 
a transitive action of the finite group $G$; let moreover $M$, $N$ be $R[G]$-modules with decompositions $M = \oplus_{i \in I}M_i$, 
$N =\oplus_{j \in J} N_j$ compatible with the $G$-action, in the sense of subsection
\ref{preliminary}.  Let $f :M \rTo N$ a map of $R[G]$-modules;
 writing 
$f$ as a sum $f= \oplus_{i,j} f_{i,j}$, the $G$-equivariance is equivalent to 
$f_{g(i), g(j)} \circ g = g \circ f_{i,j}$ for all $i \in I$, $j \in J$, $g \in G$. Danila's lemma allows to identify the map of invariants $f^G : M^G \rTo N^G$ with a map: 
$$ f^G : M_{i_0} ^{\Stab_G(i_0)} \rTo N_{j_0} ^{\Stab_G(j_0)} \;.$$In what follows we will always assume that $f \neq 0$ and that $(i_0, j_0)$ is chosen in a way that 
$ f_{i_0,j_0} \neq 0$. Making identifications\footnote{Here and in the following, if $K$ is a subgroup of $G$, not necessarily normal, we will consider $G/K$ just as the set of (left) cosets, with the natural (left) $G$-action}
 $I \simeq G/\Stab_G(i_0)$, $J \simeq 
G/\Stab_G(j_0)$ we have: 
\begin{lemmaa}\label{Danilamorphism}The map $f^G: M_{i_0} ^{\Stab_G(i_0)} \rTo N_{j_0} ^{\Stab_G(j_0)} $ is given by: 
\begin{equation}\label{eq: Danilamorphism} f^G(u) = \sum_{[g] \in G/ \Stab(i_0)} f_{g(i_0),j_0}(gu) \;.\end{equation}
\end{lemmaa}\begin{proof}The identification $M_{i_0}^{\Stab_{G}(i_0)} \simeq M^G$ is given by $u \rMapsto \sum_{[g] \in G/\Stab_G(i_0)} gu$. Therefore, applying $f$ and using $G$-equivariance, we get: 
$$ f\Big(\sum_{[g] \in G/\Stab_G(i_0)} gu \Big) = \bigoplus_{[h] \in G/\Stab_G(j_0)} h \sum_{g \in G/\Stab_G(i_0)}
f_{h^{-1}g(i_0), j_0} (h^{-1}gu) = \bigoplus_{[h] \in G/\Stab_G(j_0)} 
hw
\;,$$since the the sum $w=\sum_{[g] \in G/\Stab_G(i_0)}
f_{h^{-1}g(i_0), j_0} (h^{-1}gu) $ does not depend on $h$, because the map $[g] {\rTo} [h^{-1}g]$ is a riparametrization of the set $G/\Stab_G(i_0)$. We have now only to prove that $w$ is $\Stab_G(j_0)$-invariant. But the only effect of $h \in \Stab_G(j_0)$ on $w$ is again to permute the order of the 
summands in the expression of $w$: hence $w$ is $\Stab_G(j_0)$-invariant and is identified with $f^G(u)$.
\end{proof}
Let now $K_1,K_2$ subgroups of $G$. The group $K_1$ acts on the cosets $G/K_2$. Denote with $K_1/K_2$ the orbit of $[e]$ in $G/K_2$ for the $K_1$-action\footnote{Remark that the orbit $K_1/K_2$ is in bijection with the cosets $K_1/K_1\cap K_2$}. It is clear that $\Stab_{K_1}[e] \simeq K_1 \cap K_2$. 
\begin{lemmaa}\label{DMS}Consider the $G$-equivariant morphism $f : M   
\rTo   N$. Suppose that $\Stab_G(j_0)$ decomposes in a direct product $\Stab_{G}(j_0) \simeq P \times Q$, with $Q$ acting trivially on $N_{j_0}$. Suppose moreover that $f_{g(i_0),j_0} =0 $ if $[g] \not \in \Stab_G(j_0)/ \Stab_G(i_0)$ and that $$M_{i_0}^{\Stab_G(i_0) \cap \Stab_G(j_0)} = M_{i_0}^{\Stab_G(i_0)}\;. $$ Then the map of invariants $f^G$ 
can be identified with $|Q|\cdot \tilde{f}^P$, where
$\tilde{f}$ is the $P$-equivariant map: 
$$ \tilde{f} : \bigoplus_{[g] \in P/ \Stab(i_0)} M_{g(i_0)} \rTo N_{j_0} \;.$$
\end{lemmaa}\begin{proof}By the hypothesis on $f$ and by lemma \ref{Danilamorphism}, it is easy to see that $f^G$ coincides 
with the map $(f^{\prime})^{\Stab_G(j_0)}$ of $\Stab_G(j_0)$-invariants of the $\Stab_G(j_0)$-equivariant restriction: 
$$ f^{\prime}
\colon \bigoplus_{[g] \in \Stab_G(j_0)/ \Stab_G(i_0)} M_{g(i_0)} \rTo N_{j_0} \;.$$
Therefore it is sufficient to identify $(f^{\prime})^{\Stab_{G}(j_0)}$ with $|Q|\cdot \tilde{f}^P$. Decomposing the expression of $f^{\prime}$ according to the decomposition of $\Stab_G(j_0) \simeq P \times Q $, we can rewrite: $$
f^{\prime}
\colon \bigoplus_{[k] \in \Stab_G(j_0)/ P} \bigoplus_{[g] \in 
 P/ \Stab_G(i_0)} M_{kg(i_0)} \rTo N_{j_0} \;.$$Now, since the $P$ and $Q$ actions commute, we can take first $Q$-invariants, and then $P$-invariants. 
Remark that $Q$ acts freely on $\Stab_G(j_0)/ P \simeq Q$ and therefore permutes
the components of the  direct sum indexed by $\Stab_G(j_0)/ P$.
Therefore: $(f^{\prime})^{\Stab_G(j_0)}$ is given by 
\begin{align*} (f^{\prime})^{\Stab_G(j_0)}(u) = \; & ((f^{\prime})^Q)^P(u) =  \sum_{k \in Q} k 
\sum_{[g] \in  P/ \Stab_G(i_0)} f_{g(i_0),j_0} (gu) \\
= \; & |Q| \sum_{[g] \in  P/ \Stab_G(i_0)} f_{g(i_0),j_0} (gu) =  |Q| \cdot \tilde{f}^P(u) 
 \end{align*}since $Q$ acts trivially on $N_{j_0}$, where $u \in 
M^{\Stab_G(i_0) \cap P}_{i_0} = M_{i_0}^{\Stab_G(i_0)}$. \end{proof}
\begin{remarka}\label{rmk: dmss}In the notations of lemma \ref{Danilamorphism}, if $J=\{j_0\}$, and hence $\Stab_G(j_0)=G$, we can rewrite (\ref{eq: Danilamorphism})~as: 
$$ f^G(u) = \sum_{g \in G/\Stab_G(i_0)} g f_{i_0,j_0}(u) \;.$$This 
applies to the map $\tilde{f}$ in lemma \ref{DMS}.
\end{remarka}
\section{Permutations of factors in a multitor}\label{app: perm}
\begin{notata}\label{not: torpart}Let $S_0, \dots, S_h$ be  a partition of the set $\{1, \dots, l\}$. 
Let $X$ be an algebraic variety and  $E_i$ coherent sheaves on $X$, $i=0, \dots,h$. We will denote with $\Tor^{S_0, \dots, S_h}_q(E_0, \dots, E_h)$ the multitor: 
$ \Tor_q^{S_0, \dots, S_h}(E_0, \dots, E_h):= \Tor_q(F_1, \dots, F_l)$ with
$F_j=E_i$ if $j \in S_i$.  
\end{notata}
 Let $Y$ be a smooth subvariety of codimension $r$ of a smooth variety $X$ and let $i: Y \rInto X$ be the closed immersion. Denote with $N_{Y/X}$ the normal bundle
of $Y$ in $X$.
Consider the line bundles $L, E_1, \dots, E_h$  on $X$. Let 
$L_Y$ be the restriction of the line bundle $L$ to the subvariety $Y$. 
It is clear that the multitor $\Tor_q^{S_0, \dots, S_h}(L_Y, E_1, \dots, E_h)$, where $\{ S_i \}_{i}$ is a partition of $\{1, \dots, l\}$, does not depend on the order of the factors and it is always 
isomorphic  to
$ \Tor_q^{|S_0|}(L_Y) \tens \bigotimes_{i=1}^h E_i^{\tens |S_i|}$, in the notations of \ref{not: puremultitor}.
However a permutation of factors in the multitor acts as an 
automorphism, which we are interested in studying. 

\begin{remarka}\label{rmka:auto}In all generality, if $F_i$ are coherent sheaves on a smooth variety $X$ and
if $\comp{R}_i$ are locally free resolutions of the sheaves $F_i$, then the action of a consecutive transposition $\tau_{j,j+1}$:
$$ \widetilde{\tau_{j,j+1}}: \Tor_{q}(F_1, \dots,F_j,F_{j+1}\dots, F_l) \rTo 
\Tor_q(F_1, \dots F_{j+1},F_j, \dots, F_l)$$is given by the action induced in 
$(-q)$-cohomology by the morphism of complexes:
$$ \widetilde{\tau_{j,j+1}}: \comp{R}_1 \tens \cdots \tens \comp{R}_j \tens \comp{R}_{j+1} \tens \cdots \tens
\comp{R}_l \rTo \comp{R}_1 \tens \cdots \tens \comp{R}_{j+1} \tens \comp{R}_{j} 
\tens \cdots \tens
\comp{R}_l$$defined as $\widetilde{\tau_{j,j+1}}(u_1 \tens \cdots \tens 
u_{j} \tens u_{j+1} \tens \cdots \tens u_l) = (-1)^{h_j h_{j+1}}u_1 \tens \cdots \tens 
u_{j+1} \tens u_{j} \tens \cdots \tens u_l$, where~$u_m \in R^{h_m}_m$. 
\end{remarka}
\subsection{Action on a pure multitor.}
We  suppose first that $S_i = \emptyset$, for $i=1, \dots, h$: we are in the case of the "pure" multitor 
$ \Tor_q^l(L_Y):= \Tor_{q}^{\{1, \dots,l \}}(L_Y)$.
The following lemmas explain how $\perm_l$ acts permutating the factors on a multitor of this kind and what are the invariants
for this action. 
\begin{notat}\label{not: repstandard(i)}
Let $R_l \simeq \mbb{C}^l$ and $\rho_l$ be the natural\footnote{This means that $\perm_l$ acts permutating the vectors of the canonical basis in $R_l$} and the 
standard representations of $\perm_l$, respectively. Let $e_i$ the canonical basis of $R_l$. Denote with $R_{l-1}(i)$, $1 \leq i \leq l$, the vector space
${R_{l-1}(i):=R_l/<e_i>}\simeq \sum_{1 \leq j \leq l, j\neq i} \mbb{C}e_j$. 
It is isomorphic to the natural representation of $\perm_{l-1}$ and it splits in $R_{l-1}(i)~=~1~\oplus~\rho_{l-1}(i)$, where $\rho_{l-1}(i)$ is the standard representation of $\perm_{l-1}$, embedded in $R_{l-1}(i)$. We will indicate with 
$\sigma_{l}$ and $\sigma_{l-1}(i)$ the elements $\sigma_l = \sum_{i=1}^l e_i$, $\sigma_{l-1}(i) = \sum_{1 \leq j \leq l, j \neq i} e_j$; we will call them \emph{the canonical elements of $R_{l}$ and $R_{l-1}(i)$}, respectively. They are invariants for the action of $\perm_l$, $\perm_{l-1}$, respectively.
\end{notat}
\begin{lemmaa}\label{torinv}We have an isomorphism of 
$\perm_l$-representations\footnote{the symmetric group acts here on both factors: on the first via the standard representation, on the second permuting the factors of $L_Y^{\tens l}$.}: \begin{equation} \label{torinvaequa}
\Tor^l_q(L_Y) \simeq
	\Lambda^q(N^*_{Y/X} \tens \rho_l) \tens L_Y^{\tens ^l}\;. \end{equation}
	\end{lemmaa}
	\begin{proof}Suppose that $L$ is trivial.
	We first verify the statement locally. Suppose $Y$ is the scheme of zeros of a section $s$ 
	of a vector bundle $F$ of rank $r$, transverse to the zero section. 
	Consider the Koszul resolution
	$\comp{K}:= \comp{K}(F,s)$ of the structural sheaf $\FS_Y$. 
	The Koszul complex $\comp{K}(F \tens R_l, s \tens \sigma_l)$ is $\perm_l$-isomorphic to the 
	tensor product $\comp{K} \tens \cdots \tens \comp{K}$ and 
	consequently its $(-q)$-cohomology 
	is $\perm_l$-isomorphic to 
	$\Tor^l_{q}(\FS_Y)$. 
	On the other hand, consider
	the Koszul complex $\comp{K}(F \tens \rho_l, 0)$; 
	we have the isomorphism of $\perm_l$-representations:
	$$ \comp{K}(F \tens R_l, s \tens \sigma_l) \rTo ^{\simeq} \comp{K}(F, s) \tens \comp{K}(F \tens 
	\rho_l,0)  
	$$where $\perm_l$ acts trivially on $\comp{K} (F ,s)$. 
Since $\comp{K}(F \tens 
	\rho_l,0)$ is a complex of locally free sheaves with zero differentials, we have: 
\begin{equation}\label{eq: pretorrepr} H^{-q}(\comp{K}(F \tens R_l, s \tens \sigma_l)) \simeq \bigoplus_{i+j=-q}
H^i(\comp{K} (F ,s)) \tens K^{j}(F \tens 
	\rho_l,0) \;.\end{equation}
Consequently we obtain an isomorphism of 
	$\perm_l$- representations:
	\begin{equation} 
\label{eq:torrepr}
\Tor^l_q(\FS_Y) \simeq \FS_Y \tens \Lambda^q(F^* \tens \rho_l) 
	\simeq \Lambda^q(N_{Y/X}^* \tens \rho_l) \end{equation}because $F\trest_Y \simeq N_{Y/X}$.  
It is easy to verify that the isomorphism (\ref{eq:torrepr}) does not depend
	on the choice of the vector bundle $F$ and on the section $s$; 
moreover, the local isomorphisms glue 
together and allow to define the above isomorphism globally.	

If $L$ is not trivial the lemma is an immediate consequence of what we just proved, together with the fact that a locally free resolution of $L_Y$ is provided by the complex: $\comp{K} \tens L$. 
\end{proof}Let $i \in \{1, \dots, l\}$; denote with $\overline{\{ i\}}$ the complementary of $\{ i \}$ in $\{1, \dots,l\}$.
\begin{crla}\label{naturalissimo}
The natural map:
\begin{equation}\label{naturalissima} 
\Tor_q^{\overline{ \{ i\}}, \{ i\}}(L_Y,L)  \simeq \Tor_q(\underset{i-1-{\rm times}}{\underbrace{L_Y, \dots, L_Y}}, L, \underset{l+1-i-{\rm times}}{\underbrace{L_Y, \dots, L_Y}}) \rTo \Tor_q(\underset{l+1-{\rm times}}{\underbrace{L_Y, \dots, L_Y}}) = \Tor_q^{l+1}(L_Y) \;,\end{equation}induced by the restriction $L \rTo L_Y$, can be identified to the natural map: 
$$ \gamma_i: \Lambda^q(N_{Y/X}^* \tens \rho_l(i)) \tens L_Y^{\tens l+1} \rInto \Lambda^q(N_{Y/X}^* \tens \rho_{l+1}) \tens L_Y^{\tens l+1}$$ induced by the inclusion $\rho_l(i) \rInto \rho_{l+1}$.
\end{crla}
\begin{proof}
It is clear that it is sufficient to 
prove the statement for $L$ trivial. 
As in
the proof of lemma \ref{torinv}, we have: 
$ \Tor_q^{\overline{ \{ i\}}, \{ i\}}(\FS_Y, \FS_X) \simeq H^{-q}(\comp{K}(F \tens R_l(i), s \tens \sigma_l(i)))$, in notation \ref{not: repstandard(i)}. 
The natural map: 
$$ \Tor_q^{\overline{ \{ i\}}, \{ i\}}(\FS_Y, \FS_X) \simeq H^{-q}(\comp{K}(F \tens R_l(i), s \tens \sigma_l(i))) \rInto H^{-q}(\comp{K}(F \tens R_{l+1}, s \tens \sigma_{l+1})) \simeq \Tor_q^{l+1}(\FS_Y) $$
is induced by the inclusion $R_l(i) \rInto R_{l+1}$. This inclusion yields the 
inclusion $\rho_l(i) \rInto \rho_{l+1}$; hence the natural map above can be identified, by (\ref{eq: pretorrepr}) and (\ref{eq:torrepr}) to the inclusion: 
$$ \gamma_i: \Lambda^q(N^*_{Y/X} \tens \rho_l(i)) \simeq \FS_Y \tens \Lambda^q(F^* \tens \rho_l(i)) \rInto \FS_Y \tens \Lambda^q(F^* \tens \rho_{l+1}) \simeq  
\Lambda^q(N^*_{Y/X} \tens \rho_{l+1}) \;.$$   
\end{proof}
We come now to the computation of the $\perm_k$-invariants 
of the representation (\ref{torinvaequa}). 
\begin{lemmaa}\label{cinquediciotto}Let $V$ be a vector space of dimension $2$ and consider the $GL(V)\times \perm_k$-- representation $\Lambda^q(V \tens \rho_k)$. It has 
$\perm_k$-invariants if and only if $q=2h$, $\exists h \in \mbb{N}$; in this case the invariants are isomorphic, as a $GL(V)$-representation, to the Schur functor $S_{h,h}V \simeq (\Lambda^2 V)^{\tens h}$. 
\end{lemmaa}
\begin{proof}
  The dimension of the space of invariants is 
        given by the scalar product 
	\begin{equation}\label{scalar} \dim \Lambda^q(V \tens \rho_k)^{\perm_k} = \langle \chi_{\Lambda^q(V \tens \rho_k)}, \chi_1 \rangle  \end{equation}between the character of the representation $\Lambda^q(V \tens \rho_k)$ and the character of the trivial one. Now
we have: $$
\Lambda^q( V \tens \rho_k) \simeq \Lambda^q(\rho_k \oplus \rho_k) = 
\bigoplus_{i=1}^q \Lambda^i \rho_k \tens \Lambda^{q-i}\rho_k \; ;
$$ hence the scalar product (\ref{scalar}) becomes:
\begin{equation*}
\dim \Lambda^q(V \tens \rho_k)^{\perm_k}  =
        \sum_{i=1}^q \langle \chi_{\Lambda^i \rho_k} , \chi_{\Lambda^{q-i} \rho_k} \rangle  
\;.\end{equation*}
If $i \neq q-i$, then $\Lambda^i \rho_k$, $\Lambda^{q-i} \rho_k$ are different irreducible representations of $\perm_k$ and $\langle \chi_{\Lambda^i \rho_k} , \chi_{\Lambda^{q-i} \rho_k} \rangle=0$. In the case
$ i=q-i$, which yields $q=2i$, we have $\langle \chi_{\Lambda^q(V \tens \rho_k)}, \chi_1 \rangle =1$. This proves the first sentence.

Consider now the isomorphism of $GL(V) \times \perm_k$-representations (see~\cite[ex. 6.11]{FultonHarrisRT}): 
	$$ \Lambda^q(V \tens \rho_k) \simeq \bigoplus_{\lambda} 
S^{\lambda}V \tens 
	S^{\lambda^{\prime}}\rho_k $$where $\lambda^{\prime}$ denotes the conjugate partition to $\lambda$ and 
where
the sum is on the partitions  of $q$
	having at most $\dim V$ rows and $\dim \rho_k$ columns. The invariants are: \begin{equation} 
		\label{schur}
		\Lambda^q(V \tens \rho_k)^{ \perm_k}   \simeq \bigoplus_{\lambda} S^{\lambda}V \tens [
	S^{\lambda^{\prime}}\rho_k]^{ \perm_k} \;.\end{equation}It means that if 
	$\dim \Lambda^q(V \tens \rho_k)^{ \perm_k}=1$, then $[S^{\lambda^{\prime}}\rho_k]^{ \perm_k}$ is non zero for only one partition $\lambda_0$ with at most $2$ rows and $k-1$ columns, and in this case both $S^{\lambda_0}V$ and $[S^{\lambda_0^{\prime}}\rho_k]^{ \perm_k}$ have to be of dimension $1$. Now the dimension of a Schur functor is given 
by the formula (see~\cite[theorem 6.3]{FultonHarrisRT}): 
	$$ \dim S^{\lambda}V = \prod_{1 \leq i<j \leq k} \frac{\lambda_i - \lambda_j +j-i}{j-i} \;.$$Imposing that $\dim S^{\lambda_0}V =1$ forces
	$\lambda_i = \lambda_j$ for all $i$ and $j$. Since $\lambda_0$ can have at most $2$ rows, this means that $\lambda_0$ has to be of the form $\lambda_0=h+h$, for $h \in \mbb{N}$. Now the Schur functor $S^{h,h}V$ embeds in $(\Lambda^2V)^{\tens h}$ (see~\cite[problem 6.15]{FultonHarrisRT}), but the two vector spaces have the same dimension, hence the result. 
\end{proof}
\begin{lemmaa}\label{invaperm} 
\begin{enumerate}
\item Let $u,v$ be a basis of $V$. Consider the
  $\perm_k$-invariant bivector
$$ \omega = \sum_{i=1}^k ue_i \wedge v e_i \in \Lambda^2( V
\tens R_k)$$and let $\omega^l \in \Lambda^{2l}( V
\tens R_k )$ its $l$-th exterior power. Consider the projection: $\pi^{R_k}: R_k
\rTo \rho_k$. If $1 \leq l \leq k-1$, the image $\omega ^{R_k}_{l}$ of
 $\omega^l / (l+1)!$ in $\Lambda^{2l}( V \tens \rho_k)$ is
nonzero.
\item Let $i \in \{ 1, \dots k \}$, and $G_i = \rm{Stab}_{\perm_k} \{ i
  \}$.
  The
  projection $$\phi_i :V \tens \rho_k \rTo V \tens 
\rho_{k-1}(i) $$induced from the projection from $\pi_i: R_k \rTo R_{k-1}(i)$, is
  $G_i$-equivariant and for $1 \leq l \leq k-2$ the image of
  $\omega^{R_k}_l$ for the projection $\Lambda^{2l}\phi_i$ is exactly
  $\omega^{R_{k-1}(i)}_l$. 
\item The $\perm_k$-equivariant map 
$ \Lambda^q(V \tens 
\rho_k) \rTo \bigoplus_{i=1}^k \Lambda^q(V \tens \rho_{k-1}(i)) $, induced by the natural projections $\rho_{k} \rTo \rho_{k-1}(i)$, gives an isomorphism between the $\perm_k$-invariants.
\item The $\perm_k$-equivariant map 
$ \bigoplus_{i=1}^k \Lambda^q(V \tens \rho_{k-1}(i)) \rTo  \Lambda^q(V \tens 
\rho_k)$, induced by the natural inclusions $\rho_{k-1}(i) \rTo \rho_{k}$, gives an isomorphism between the $\perm_k$-invariants.
\end{enumerate}\end{lemmaa}
\begin{proof}1. We can restrict ourselves to the case $l=k-1$. We have: 
\begin{equation*} 
\frac{\omega^{k-1}}{k!} =  \frac{1}{k!} \left( \sum_{i=1}^k ue_i \w ve_i
\right)^{k-1} 
  =  \frac{1}{k!} (k-1)! \sum_{i=1}^{k} \widehat{ ue_i \w ve_i}
\end{equation*}where $ \widehat{ ue_i \w ve_i}$ indicates $$ \widehat{
  ue_i \w ve_i} = ue_1 \w
ve_1 \w \dots \w ue_{i-1} \w ve_{i-1} \w ue_{i+1} \w ve_{i+1} \w \dots \w
ue_k \w ve_k \;.$$We now remark that the projection 
$ \Lambda^{2k-2}(\id \tens \pi) ( \widehat{ ue_i \w ve_i})$
is nonzero and always the same for
every~$i$, since $\Lambda^{2k-2}(V \tens \rho_k)$ is the trivial $\perm_k$-representation; therefore the projection of $\omega^{k-1} / k!$ on
$\Lambda^{2k-2}(V \tens \rho_k)$ is $\Lambda^{2k-2}(\id
\tens \pi)( \widehat{ ue_1 \w ve_1})$, that is, a volume element in
$\Lambda^{2k-2}(V \tens \rho_k)$. 

2. The statement is evident once remarked that the commutative
diagram:
\begin{diagram}
\Lambda ^{2l}(V \tens R_k) & \rTo ^{\Lambda ^{2l} \pi_i} & \Lambda
^{2l} (V \tens R_{k-1}(i)) \\
\dTo ^{ \Lambda ^{2l} \pi ^{R_k}} & & \dTo _{\Lambda ^{2l} \pi ^{R_{k-1}(i)}} \\
\Lambda ^{2l} (V \tens \rho_k) & \rTo ^{\Lambda ^{2l} \phi_i} & 
\Lambda ^{2l} (V \tens \rho_{k-1}(i)) 
\end{diagram}is $G_i$-equivariant for $l \leq k-2$, since the diagram:
\begin{diagram}
V \tens R_k & \rTo ^{\pi_i} & V \tens R_{k-1}(i) \\
\dTo ^{\pi ^{R_k}} & & \dTo _{\pi ^{R_{k-1}(i)}} \\
V \tens \rho_k & \rTo ^{\phi_i} & V \tens \rho_{k-1}(i) 
\end{diagram}is $G$-equivariant. 

3. The $\perm_k$--invariants are both $1$-dimensional vector spaces, by Danila's lemma \ref{Danila} and by lemma \ref{cinquediciotto}; therefore it suffices to prove that 
the map between invariants: 
$$ \Lambda^q(V \tens \rho_k)^{\perm_k} \rTo \left[ 
\bigoplus_{i=1}^k \Lambda^q(V \tens \rho_{k-1}(i)) \right ]^{\perm_k}$$is non zero.  This reduces to the following easy consequence of points 1 and 2. Consider the $\perm_{k}$-equivariant diagram: 
\begin{diagram}
\Lambda ^{2l}(V \tens R_k) & \rTo ^{\oplus_i \Lambda ^{2l} \pi_i} & \oplus_{i=1}^{k} 
\Lambda
^{2l} (V \tens R_{k-1}(i)) \\
\dTo ^{ \Lambda ^{2l} \pi ^{R_k}} & \rdTo 
^{\Lambda ^{2l} \psi} & \dTo _{\oplus_i \Lambda ^{2l} \pi ^{R_{k-1}(i)}} \\
\Lambda ^{2l} (V \tens \rho_{k}) & \rTo ^{\oplus_i \Lambda ^{2l} \phi_i} &  \oplus_{i=1}^{k}
\Lambda ^{2l} (V
 \tens \rho_{k-1}(i)) 
\end{diagram}
What we want to prove 
is equivalent
to proving that the morphism $\Lambda ^{2l} \psi : \oplus_i (\Lambda^{2l} 
(\pi^{R_{k-1}(i)} \circ \pi_i ))$ above induces a nonzero morphism
between the vector spaces of invariants:
$$ \Lambda ^{2l}(V \tens R_k) ^{\perm_{k}} \rTo^{ (\Lambda ^{2l} \psi)^{\perm_{k}} }\left [  
\oplus_{i=1}^{k}
\Lambda ^{2l} (V \tens \rho_{k-1}(i)) 
\right]^{\perm_{k}} \; .$$Let us take the $\perm_{k}$-invariant
element $\omega ^l \in \Lambda ^{2l}(V \tens R_k)$ considered in
point 1. We know that the images $\omega_{i,l} = {\rm pr}_i
\circ \Lambda ^{2l} \psi (\omega^l)$ are nonzero. Therefore $\Lambda
^{2l} \psi (\omega^l) = (\omega_{i,l})_i$ 
is a nonzero element; since $\omega^l$ is
$\perm_{k}$-invariant and $\Lambda ^{2l} \psi$ is $\perm_{k}$-equivariant, the element
$\Lambda
^{2l} \psi (\omega^l)$ is necessarily $\perm_{k}$-invariant. Since we proved
that it is nonzero, we are done. 

4. It follows immediately from point 3 by taking duals.
\end{proof}
\begin{crla}\label{crl: invator}If
  $\codim Y =2$ the $\perm_l$-equivariant sheaf 
$\Tor_q^l(L_Y)$ has non zero
$\perm_l$-invariants if and only if $q=2h$, $0 \leq h \leq l-1$: in this case:
$$ \Tor_{2h}^l(L_Y)^{\perm_l} \simeq (\Lambda^2 N^*_{Y/X})^{\tens ^h} \tens L_Y^{\tens ^l} \;.$$ 
\end{crla}
\subsection{Action on a mixed multitor.} We pass now to the general case of permutation of factors in a mixed multitor. 
\begin{remarka}\label{C8}
Let $S_0, \dots, S_h$ a partition of $\{1, \dots,l\}$ and let $L$, $E_i$, $i=1, \dots,h$ locally free sheaves on the algebraic variety $X$.
By lemma \ref{torinv}, the mixed multitor $\Tor_q^{S_0, \dots, S_h}(L_Y, E_1, \dots, E_h)$, 
can be \emph{naturally identified}
to:
$$ i(S_0, \dots, S_h): \Tor_q^{S_0, \dots, S_h}(L_Y, E_1, \dots, E_h) \rTo^{\simeq}   
\Lambda^q(N_{Y/X}^* \tens \rho_{|S_0|}) \tens L_Y^{\tens ^{|S_0|}} \tens  \bigotimes_{i=1}^h
E_i ^{|S_i|} \;.$$With natural identification we precisely mean the following.
As in notation \ref{not: torpart}, indicate with 
$F_j := L_Y$ if $j \in S_0$ and $F_j := E_i$ if $j \in S_i$. We can work locally.
Denote with $\comp{K}_j$ a locally free resolution of the sheaves $F_j$, 
for all $j=1, \dots, l$; if $j \in S_0$ we can take $\comp{K}_j$ to be 
the Koszul 
resolution $\comp{K}_Y \tens L$ of $L_Y$, otherwise, for $j \in S_i$, $i \neq 0$, take 
$\comp{K}_j$ to be the complex $0 \rTo E_i \rTo 0$. In these notations the multitor 
$\Tor_{q}^{S_0, \dots, S_h}(L_Y, E_1, \dots, E_h)$ is \emph{by definition} 
$H^{-q}(\tens_{j =1}^l \comp{K}_j)$. Let now $\theta \in \perm_l$ the unique 
permutation of $\{1, \dots,l\}$ such that
$\theta(S_i) < \theta(S_j)$, if $i<j$, and
$\theta\trest_{S_i} : S_i \rTo \theta(S_i)$ is increasing. 
The identification $i(S_0, \dots, S_h)$ is then given by the composition: 
\begin{align*} H^{-q}(\Tens_{j =1}^l \comp{K}_j ) \rTo^{H^{-q}(\theta)}_{\simeq} & H^{-q}( 
\Tens_{j=1}^l \comp{K}_{\theta^{-1}(j)}) 
\simeq  \; H^{-q}(\Tens_{i =0}^h \Tens_{j \in S_i} \comp{K}_i)
\simeq H^{-q}(\Tens_{i \in S_0} \comp{K}_{i} \tens \Tens_{i=1}^h E^{\tens|S_i|}_i) 
\\ \simeq & \; 
\Lambda^q(N^*_{Y/X} \tens \rho_{|S_0|}) \tens L_Y^{\tens^{|S_0|}}\tens
\Tens_{i=1}^h E^{\tens|S_i|}_i \;,
\end{align*}since the functor $- \tens \Tens_{i=1}^h E^{\tens |S_i|}_i$ is exact, because $E_i$ are locally free, and by lemma \ref{torinv}.
\end{remarka}
The following lemma is now immediate. 
\begin{lemmaa}\label{permsign}
Let $\tau \in \perm_l$. For $i=0, \dots, h$, let $\sigma_{i}(\tau)$ be the 
unique increasing bijection 
$\sigma_i(\tau) : \tau(S_i) \rTo S_i$;
let $\beta_i(\tau)$ be the permutation $\beta_i(\tau):= \sigma_i(\tau) \circ \tau \trest_{S_i}$ of $S_i$, seen in $\perm_{|S_i|}$.
Let moreover $\alpha(\tau)$ be the automorphism of 
$\Lambda^q(N^*_{Y/X} \tens \rho_{|S_0|})$ given by the action of $\beta_0(\tau)$. 
Then the 
 following diagram commutes: 
\begin{diagram}
 \Tor_q^{S_0, \dots, S_h }(L_Y,E_1, \dots, E_h) & \rTo ^{i(S_0, \dots, S_h)} & \Lambda^q(N^*_{Y/X} \tens \rho_{|S_0|})  \tens L_Y^{\tens |S_0|} \tens 
\bigotimes_{i=1}^h E_i^{\tens |S_i|} \\
  \dTo ^{\tilde{\tau}} & & \dTo_{\alpha(\tau) \tens \beta_0(\tau) \tens 
\Tens_{i=1}^h \beta_i(\tau)} \\
\Tor_q^{\tau(S_0), \dots, \tau(S_h) }(L_Y,E_1, \dots, E_h) 
 & \rTo^{i(\tau(S_0), \dots, \tau(S_h))}  & \Lambda^q(N^*_{Y/X} \tens \rho_{|S_0|})  \tens L_Y^{\tens |S_0|} \tens 
 \bigotimes_{i=1}^h E_i^{\tens |S_i|} \;.
 \end{diagram}where $\beta_0(\tau)$, $\beta_i(\tau)$ act 
on $L_{Y}^{\tens {|S_0|}}$, $E_i^{\tens |S_i|}$, respectively, by permutation of factors.
\end{lemmaa}


\begin{thebibliography}{10}

\bibitem{BourbakiAC}
N.~Bourbaki, \emph{{\'E}l{\'e}ments de {M}ath{\'e}matiques, {A}lg{\`e}bre
  {C}ommutative, {C}hapitre 10}, Masson, Paris, 1998.

\bibitem{Boutot1987}
J.-F.~Boutot, \emph{Singularit\'es rationnelles et quotients par
  les groupes r\'eductifs}, Invent. Math. \textbf{88} (1987), no.~1, 65--68.
  \MR{877006 }

\bibitem{BridgelandKingReid2001}
T.~Bridgeland, A.~King, and M.~Reid, \emph{The {M}c{K}ay
  correspondence as an equivalence of derived categories}, J. Amer. Math. Soc.
  \textbf{14} (2001), no.~3, 535--554 (electronic). \MR{1824990}

\bibitem{Burns1974}
D.~Burns, \emph{On rational singularities in dimensions {$>2$}}, Math. Ann.
  \textbf{211} (1974), 237--244. \MR{0364672}

\bibitem{CartanEilenbergHA}
H.~Cartan and S.~Eilenberg, \emph{Homological algebra}, Princeton
  Landmarks in Mathematics, Princeton University Press, Princeton, NJ, 1999,
  With an appendix by David A. Buchsbaum, Reprint of the 1956 original.
  \MR{1731415}

\bibitem{CrawReid2002}
A.~Craw and M.~Reid, \emph{How to calculate {$A$}-{H}ilb {$\Bbb C\sp
  3$}}, Geometry of toric varieties, S\'emin. Congr., vol.~6, Soc. Math.
  France, Paris, 2002, pp.~129--154. \MR{2075608}

\bibitem{Danilathese}
G.~Danila, \emph{Formule de verlinde et dualit\'e \'etrange sur le plan
  projectif}, Ph.D. thesis, Universit\'e Paris 7, 1999, http://www.institut.math.jussieu.fr/theses/1999/danila/these.ps

\bibitem{Danila2000}
\bysame, \emph{Sections du fibr\'e d\'eterminant sur l'espace de modules des
  faisceaux semi-stables de rang 2 sur le plan projectif}, Ann. Inst. Fourier
  (Grenoble) \textbf{50} (2000), no.~5, 1323--1374. \MR{1800122
  }

\bibitem{Danila2001}
\bysame, \emph{Sur la cohomologie d'un fibr\'e tautologique sur le sch\'ema de
  {H}ilbert d'une surface}, J. Algebraic Geom. \textbf{10} (2001), no.~2,
  247--280. \MR{1811556}

\bibitem{Danila2002}
\bysame, \emph{R\'esultats sur la conjecture de dualit\'e \'etrange sur le plan
  projectif}, Bull. Soc. Math. France \textbf{130} (2002), no.~1, 1--33.
  \MR{1906190}

\bibitem{Danila2004}
\bysame, \emph{Sur la cohomologie de la puissance sym\'etrique du fibr\'e
  tautologique sur le sch\'ema de {H}ilbert ponctuel d'une surface}, J.
  Algebraic Geom. \textbf{13} (2004), no.~1, 81--113. \MR{2008716}

\bibitem{DeCataldoMigliorini2004}
M.~A. de~Cataldo and L.~Migliorini, \emph{The {C}how motive of
  semismall resolutions}, Math. Res. Lett. \textbf{11} (2004), no.~2-3,
  151--170. \MR{2067464}

\bibitem{DrezetNarasimhan1989}
J.-M.~Drezet and M.~S. Narasimhan, \emph{Groupe de {P}icard des vari\'et\'es de
  modules de fibr\'es semi-stables sur les courbes alg\'ebriques}, Invent.
  Math. \textbf{97} (1989), no.~1, 53--94. \MR{999313}

\bibitem{EisenbudCA}
D.~Eisenbud, \emph{Commutative algebra}, Graduate Texts in Mathematics, vol.
  150, Springer-Verlag, New York, 1995, With a view toward algebraic geometry.
  \MR{1322960}

\bibitem{EllingsrudGoettscheLehn2001}
G.~Ellingsrud, L.~G{\"o}ttsche, and M.~Lehn, \emph{On the cobordism
  class of the {H}ilbert scheme of a surface}, J. Algebraic Geom. \textbf{10}
  (2001), no.~1, 81--100. \MR{1795551}

\bibitem{FultonHarrisRT}
W.~Fulton and J.~Harris, \emph{Representation theory}, Graduate Texts in
  Mathematics, vol. 129, Springer-Verlag, New York, 1991, A first course,
  Readings in Mathematics. \MR{1153249}

\bibitem{EGA3.2}
A.~Grothendieck, \emph{\'{E}l\'ements de g\'eom\'etrie alg\'ebrique.
  {III}. \'{E}tude cohomologique des faisceaux coh\'erents. {II}}, Inst. Hautes
  \'Etudes Sci. Publ. Math. (1963), no.~17, 91. \MR{0163911}

\bibitem{Haiman1999}
M.~Haiman, \emph{Macdonald polynomials and geometry}, New perspectives in
  algebraic combinatorics (Berkeley, CA, 1996--97), Math. Sci. Res. Inst.
  Publ., vol.~38, Cambridge Univ. Press, Cambridge, 1999, pp.~207--254.
  \MR{1731818}

\bibitem{Haiman2001}
\bysame, \emph{Hilbert schemes, polygraphs and the {M}acdonald positivity
  conjecture}, J. Amer. Math. Soc. \textbf{14} (2001), no.~4, 941--1006
  (electronic). \MR{1839919}

\bibitem{Haiman2002}
\bysame, \emph{Vanishing theorems and character formulas for the {H}ilbert
  scheme of points in the plane}, Invent. Math. \textbf{149} (2002), no.~2,
  371--407. \MR{1918676}

\bibitem{HartshorneLC}
R.~Hartshorne, \emph{Local cohomology}, A seminar given by A. Grothendieck,
  Harvard University, Fall, vol. 1961, Springer-Verlag, Berlin, 1967.
  \MR{0224620}

\bibitem{ItoNakamura1999}
Y.~Ito and I.~Nakamura, \emph{Hilbert schemes and simple singularities}, New
  trends in algebraic geometry (Warwick, 1996), London Math. Soc. Lecture Note
  Ser., vol. 264, Cambridge Univ. Press, Cambridge, 1999, pp.~151--233.
  \MR{1714824}

\bibitem{ItoNakamura1996}
\bysame, \emph{Mc{K}ay correspondence and {H}ilbert
  schemes}, Proc. Japan Acad. Ser. A Math. Sci. \textbf{72} (1996), no.~7,
  135--138. \MR{1420598}

\bibitem{KempfLaksov1974}
G.~Kempf and D.~Laksov, \emph{The determinantal formula of {S}chubert
  calculus}, Acta Math. \textbf{132} (1974), 153--162. \MR{0338006}

\bibitem{KollarMori1998}
J.~ Koll{\'a}r and S.~Mori, \emph{Birational geometry of algebraic
  varieties}, Cambridge Tracts in Mathematics, vol. 134, Cambridge University
  Press, Cambridge, 1998, With the collaboration of C. H. Clemens and A. Corti,
  Translated from the 1998 Japanese original. \MR{1658959}

\bibitem{Lehn1999}
M.~Lehn, \emph{Chern classes of tautological sheaves on {H}ilbert schemes
  of points on surfaces}, Invent. Math. \textbf{136} (1999), no.~1, 157--207.
  \MR{1681097}

\bibitem{MatsumuraCRT}
H.~Matsumura, \emph{Commutative ring theory}, second ed., Cambridge
  Studies in Advanced Mathematics, vol.~8, Cambridge University Press,
  Cambridge, 1989, Translated from the Japanese by M. Reid. \MR{1011461}

\bibitem{Mukai1981}
S.~Mukai, \emph{Duality between {$D(X)$} and {$D(\hat X)$} with its
  application to {P}icard sheaves}, Nagoya Math. J. \textbf{81} (1981),
  153--175. \MR{607081}

\bibitem{Nakamura2001}
I.~Nakamura, \emph{Hilbert schemes of abelian group orbits}, J. Algebraic
  Geom. \textbf{10} (2001), no.~4, 757--779. \MR{1838978}

\bibitem{Reid2002}
M.~Reid, \emph{La correspondance de {M}c{K}ay}, Ast\'erisque (2002),
  no.~276, 53--72, S\'eminaire Bourbaki, Vol.\ 1999/2000. \MR{1886756}

\bibitem{ScalaPhD}
L.~Scala, \emph{Cohomology of the {H}ilbert scheme of points on a surface
  with values in representations of tautological bundles. {P}erturbations of
  the metric in {S}eiberg-{W}itten equations}, Ph.D. thesis, Universit\'e Paris
  7, 2005, http://www.institut.math.jussieu.fr/theses/2005/scala/tesi.pdf

\bibitem{Scala2006}
\bysame, \emph{Cohomology of the {H}ilbert scheme of points on a surface with
  values in the double tensor power of a tautological bundle}, C. R. Math.
  Acad. Sci. Paris \textbf{343} (2006), no.~11-12, 741--746. \MR{2284703}

\bibitem{Scala2007}
\bysame, \emph{Dualit\'e \'etrange de {L}e {P}otier et cohomologie du sch\'ema
  de {H}ilbert ponctuel d'une surface}, Gaz. Math. (2007), no.~112, 53--65.
  \MR{2319915}

\bibitem{SerreGAGA}
J.-P.~Serre, \emph{G\'eom\'etrie alg\'ebrique et g\'eom\'etrie
  analytique}, Ann. Inst. Fourier, Grenoble \textbf{6} (1955--1956), 1--42.
  \MR{0082175}

\end{thebibliography}

\providecommand{\bysame}{\leavevmode\hbox to3em{\hrulefill}\thinspace}
\providecommand{\MR}{\relax\ifhmode\unskip\space\fi }
\providecommand{\MRhref}[2]{%
  \href{http://www.ams.org/mathscinet-getitem?mr=#1}{#2}
}
\providecommand{\href}[2]{#2}

\vspace{1cm}
\noindent
{Luca Scala \\
Department of Mathematics \\
University of Chicago\\
5734 S. University Avenue \\
Chicago, IL 60637, USA \\
{\em Email Address:} {\tt lucascala@math.uchicago.edu}
}

\end{document}